\documentclass[11pt]{amsart}
\usepackage{amssymb,amsthm,amsmath}
\RequirePackage[dvipsnames,usenames]{xcolor}
\usepackage{hyperref}
\usepackage{mathtools}
\usepackage[all]{xy}
\usepackage{tikz}
\usetikzlibrary{decorations.markings}
\usepackage{enumitem}
\usepackage{chngcntr}
\usepackage{stmaryrd}
\usepackage[framemethod=TikZ]{mdframed}
\usepackage{minibox}

\mdfsetup{
 skipabove=6pt,
 skipbelow=3pt,
 roundcorner=10pt
}

\usepackage{xparse}

\NewDocumentCommand\DownArrow{O{2.0ex} O{black}}{%
   \mathrel{\tikz[baseline] \draw [<-, line width=0.5pt, #2] (0,0) -- ++(0,#1);}
}

\NewDocumentCommand\SlantLeftDownArrow{O{2.0ex} O{black}}{%
   \mathrel{\tikz[baseline] \draw [<-, line width=0.5pt, #2] (0,0) -- ++(#1,#1);}
}

\NewDocumentCommand\UpArrow{O{2.0ex} O{black}}{%
   \mathrel{\tikz[baseline] \draw [line width=0.5pt, decoration={markings,mark=at position 1 with {\arrow[scale=2, line width=0.25pt]{to}}},  postaction={decorate}, #2] (0,0) -- ++(0,#1);}
}

\newcommand{\explright}[2]{\mathrlap{\minibox[c]{\quad $\leftarrow$  \fbox{\footnotesize #2}   \\[3pt]}} #1}

\newcommand{\explleft}[2]{\mathllap{\minibox[c]{ \fbox{\footnotesize #2}  $\rightarrow$ \  \\[3pt]}} #1}

\newcommand{\explabove}[2]{\overset{\mathclap{\minibox[c]{ \fbox{\footnotesize #2} \\ $\DownArrow[10pt]$}}}{#1}}

\newcommand{\explaboveshiftslantleft}[3]{\overset{\mathclap{\minibox[c]{ \hspace{#1} \fbox{\footnotesize #3} \\ \qquad $\SlantLeftDownArrow[10pt]$}}}{#2}}

\newcommand{\expl}[2]{\underset{\mathclap{\minibox[c]{$\UpArrow[10pt]$\\ \fbox{\footnotesize #2}}}}{#1}}

\newcommand{\explshift}[3]{\underset{\mathclap{\minibox[c]{$\UpArrow[10pt]$\\ \hspace{#1} \fbox{\footnotesize #3}}}}{#2}}

\newcommand{\explparshift}[4]{\underset{\mathclap{\minibox[c]{$\UpArrow[10pt]$\\ \hspace{#2} \fbox{\parbox{#1}{\footnotesize #4}}}}}{#3}}

\hypersetup{
bookmarks,
bookmarksdepth=3,
bookmarksopen,
bookmarksnumbered,
pdfstartview=FitH,
colorlinks,backref,hyperindex,
linkcolor=Sepia,
anchorcolor=BurntOrange,
citecolor=MidnightBlue,
citecolor=OliveGreen,
filecolor=BlueViolet,
menucolor=Yellow,
urlcolor=OliveGreen
}

\DeclareMathAlphabet{\mathchanc}{OT1}{pzc}%
                                 {m}{it}

\newcommand{\bC}{\mathbb{C}}

\newcommand{\bG}{\mathbb{G}}

\newcommand{\bN}{\mathbb{N}}

\newcommand{\bP}{\mathbb{P}}
\newcommand{\bQ}{\mathbb{Q}}

\newcommand{\bZ}{\mathbb{Z}}

\newcommand{\scr}{\mathcal}
\newcommand{\sA}{\scr{A}}

\newcommand{\sE}{\scr{E}}
\newcommand{\sF}{\scr{F}}
\newcommand{\sG}{\scr{G}}
\newcommand{\sH}{\scr{H}}
\newcommand{\sI}{\scr{I}}
\newcommand{\sJ}{\scr{J}}
\newcommand{\sK}{\scr{K}}
\newcommand{\sL}{\scr{L}}
\newcommand{\sM}{\scr{M}}

\newcommand{\sO}{\scr{O}}
\newcommand{\sP}{\scr{P}}
\newcommand{\sQ}{\scr{Q}}

\newcommand{\sX}{\scr{X}}

\newcommand{\tf}{\tilde{f}}

\newcommand{\tx}{\tilde{x}}

\newcommand{\tV}{\tilde{V}}

\newcommand{\tX}{\tilde{X}}

\newcommand{\tZ}{\tilde{Z}}

\newcommand{\ot}{\overline{t}}

\newcommand{\oy}{\overline{y}}
\newcommand{\oz}{\overline{z}}

\newcommand{\oK}{\overline{K}}

\DeclareMathOperator{\KSBA}{{KSBA}}
\DeclareMathOperator{\norm}{{norm}}
\DeclareMathOperator{\todd}{{td}}
\DeclareMathOperator{\ch}{{ch}}

\DeclareMathOperator{\CM}{{CM}}

\DeclareMathOperator{\vol}{{vol}}

\DeclareMathOperator{\lct}{{lct}}

\DeclareMathOperator{\Sym}{{Sym}}

\DeclareMathOperator{\Aut}{Aut}

\DeclareMathOperator{\Bl}{{Bl}}

\DeclareMathOperator{\codim}{codim}

\DeclareMathOperator{\Ker}{{Ker}}
\DeclareMathOperator{\pre}{{pre}}

\DeclareMathOperator{\Gr}{{Gr}}

\DeclareMathOperator{\Vol}{Vol}

\DeclareMathOperator{\id}{{id}}

\DeclareMathOperator{\im}{{im}}

\DeclareMathOperator{\Isom}{Isom}

\DeclareMathOperator{\GL}{{GL}}

\DeclareMathOperator{\Proj}{{Proj}}

\DeclareMathOperator{\red}{red}

\DeclareMathOperator{\reg}{reg}

\DeclareMathOperator{\sm}{sm}

\DeclareMathOperator{\rk}{{rk}}

\DeclareMathOperator{\Spec}{{Spec}}

\DeclareMathOperator{\Supp}{{Supp}}

\DeclareMathOperator{\var}{{Var}}

\newcommand{\factor}[2]{\left. \raise 2pt\hbox{\ensuremath{#1}} \right/
        \hskip -2pt\raise -2pt\hbox{\ensuremath{#2}}}

\counterwithin*{equation}{section}
\counterwithin*{equation}{subsection}

\makeatletter
\renewcommand\subsection{
  \renewcommand{\sfdefault}{pag}
  \@startsection{subsection}%
  {2}{0pt}{.8\baselineskip}{.4\baselineskip}{\raggedright
    \sffamily\itshape\small\bfseries
  }}
\renewcommand\section{
  \renewcommand{\sfdefault}{phv}
  \@startsection{section} %
  {1}{0pt}{\baselineskip}{.8\baselineskip}{\centering
    \sffamily
    \scshape
    \bfseries
}}
\makeatother

\renewcommand{\theenumi}{\arabic{enumi}}

\usepackage[left=1.02in,top=1.0in,right=1.02in,bottom=1.0in]{geometry}
\usepackage{mabliautoref}
\usepackage{multirow}

\setlist[enumerate]{leftmargin=0.8cm}
\setlist[itemize]{leftmargin=0.8cm}
\setlist[description]{leftmargin=0.0cm}

\DeclareMathOperator{\Kss}{K-ss}
\DeclareMathOperator{\uKs}{u-K-s}
\DeclareMathOperator{\Kps}{K-ps}

\newcommand{\rM}{\mathrm{M}}

\title{Positivity of the CM line bundle for families of K-stable klt Fano varieties}
\author{Giulio Codogni}
\address{
Dipartimento di Matematica, Universit\`{a} degli Studi di Roma  ``Tor Vergata" ,   \\
Via della ricerca scientifica, 00133 Roma, Italy.}
\email{codogni@mat.uniroma2.it}
\author{Zsolt Patakfalvi}
\address{
EPFL\\
SB MATHGEOM CAG  \\
MA B3 635 (B\^atiment MA) \\
Station 8 \\
CH-1015 Lausanne}
\email{zsolt.patakfalvi@epfl.ch}
\date{\today}

\begin{document}

\maketitle

\begin{abstract}

The Chow-Mumford (CM) line bundle is a functorial line bundle on the base of any family of klt Fano varieties. It is conjectured that it yields a polarization on the moduli space of $K$-poly-stable klt Fano varieties. Proving ampleness of the CM line bundle boils down to showing semi-positivity/positivity statements about the CM-line bundle for families with  $K$-semi-stable/$K$-polystable fibers. We prove the necessary semi-positivity statements in the $K$-semi-stable situation,
and the necessary positivity statements in the uniform $K$-stable situation, including in both cases variants  assuming $K$-stability  only for  general fibers. Our statements work in the most general singular situation (klt singularities), and the proofs are algebraic, except the computation of the limit of a sequence of real numbers via the central limit theorem of probability theory.
We also  present an application to the classification of Fano varieties. Additionally, our semi-positivity statements work in general for log-Fano pairs.

\end{abstract}

\tableofcontents

\section{Introduction}

\emph{Throughout the article, the base field is an algebraically closed field $k$ of characteristic zero.}

\subsection{Main theorem}
\label{sec:intro_boundary_free}

The interest in the moduli space of singular $K$-polystable Fano varieties stems from the classification theory of algebraic varieties. 
The \emph{birational part} of the classification theory, also called the Minimal Model Program \cite{Birkar_Cascini_Hacon_McKernan_Existence_of_minimal_models,Hacon_McKernan_Existence_of_minimal_models_for_varieties_of_log_general_type_II,
Hacon_Xu_Existence_of_log_canonical_closures,
Birkar_Existence_of_log_canonical_flips_and_a_special_LMMP,
Fujino_Introduction_to_the_log_minimal_model_program_for_log_canonical_pairs,
flips_and_abundance_for_algebraic_threefolds,Keel_Matsuki_McKernan_Log_abundance_theorem_for_threefolds,
Alexeev_Hacon_Kawamata_Termination_of_many_4-dimensional_log_flips,
Birkar_On_termination_of_log_flips_in_dimension_four,Birkar_Ascending_chain_condition_for_log_canonical_thresholds_and_termination_of_log_flips}, predicts that up to specific birational equivalences, each projective variety decomposes into iterated fibrations with general fibers of 3 basic types: \underline{Fano}, \underline{weak Calabi-Yau}, and \underline{general type} To be precise, one here needs to allow pairs, see \autoref{sec:intro_logarithmic}, but the boundary free case is a good first approximation.  

The above 3 types  are defined by having a specific class of mild singularities and negative/numerically trivial/positive canonical bundles. Then the \emph{moduli part} of the classification theory is supposed to construct a projective, compactified moduli spaces for the above 3 basic types of varieties. According to our current understanding, the moduli part seems to be doable only in the presence of a singular K\"ahler-Einstein metric, e.g., \cite[Conj 8.11, and the following 2 paragraphs]{Tian_Kahler-Einstein_metrics_with_positive_scalar_curvature},  which is predicted to be equivalent to the algebraic notion of $K$-polystability \cite{Berman_Guenancia_Kahler-Einstein_metrics_on_stable_varieties_and_log_canonical_pairs,Chen_Donaldson_Sun_Kahler-Einstein_metrics_on_Fano_manifolds_I_Approximation_of_metrics_with_cone_singularities,
Chen_Donaldson_Sun_Kahler-Einstein_metrics_on_Fano_manifolds_II_Limits_with_cone_angle_less_than_2pi,
Chen_Donaldson_Sun_Kahler-Einstein_metrics_on_Fano_manifolds_III_Limits_as_cone_angle_approaches_2pi_and_completion_of_the_main_proof,
Tian_K-stability_and_Kahler-Einstein_metrics,Odaka_The_GIT_stability_of_polarized_varieties_via_discrepancy,Odaka_The_Calabi_conjecture_and_K-stability,Odaka_Xu_Log-canonical_models_of_singular_pairs_and_its_applications, Recent_Tian}. We refer the reader to \autoref{def:K_stability} and to \autoref{def:K_stable} for the precise definition and for a characterization of $K$-semistability used in the present article. Additionally, see \autoref{sec:K_stability_versions} for an explanation on $K$-polystability.

In particular, on the Fano side, for the moduli part of the classification theory one should construct algebraically the following two spaces:
\begin{itemize}
\item the stack $\sM^{\Kss}_{n,v}$ of $K$-semistable Fano varieties of dimension $n$ and anti-canonical volume $v$, as well as, 
\item the projective good moduli space  $\rM^{\Kps}_{n,v}$ of $\sM^{\Kss}_{n,v}$  parametrizing  $K$-polystable Fano varieties of dimension $n$ and anti-canonical volume $v$.
\end{itemize}
We note that the construction of the above two spaces is known except for the properness and the projectivity of $\rM_{n,v}^{K-ps}$ via a sequence of recent papers
\cite{Blum_Liu_Xu_Openness_of_K-semistability_for_Fano_varieties,Blum_Xu_Uniqueness_of_K-polystable_degenerations_of_Fano_varieties,
Alper_Blum_Halpern-Leistner_Xu_Reductivity_of_the_automorphism_group_of_K-polystable_Fano_varieties,
Blum_Liu_Openness_of_uniform_K-stability_in_families_of_Q-Fano___varieties,
Xu_A_minimizing_valuation_is_quasi-monomial,
Jiang_Boundedness_of_Q-Fano_varieties_with_degrees_and___alpha-invariants_bounded_from_below,
Birkar_Anti-pluricanonical_systems_on_Fano_varieties,Birkar_Singularities_of_linear_systems_and_boundedness_of_Fano_varieties}.
That is, $\sM_{n,v}^{\Kss}$ is known to exist as an Artin stack of finite type over $k$ that admits a good moduli space $\rM_{n,v}^{\Kps}$. Additionally, $\rM_{n,v}^{\Kps}$ is  known to be a  separated algebraic space, which is of finite type over $k$, and the uniformly $K$-stable locus $\rM_{n,v}^{\uKs} \subseteq \rM_{n,v}^{K-ps}$ is known to be an open sub-algebraic space \cite[Thm A]{Blum_Liu_Openness_of_uniform_K-stability_in_families_of_Q-Fano___varieties}.  Furthermore, the polarization on $\rM^{\Kps}_{n,v}$ is predicted to be given by the descent $L$ to $\sM^{\Kss}_{n,v}$ of the Chow-Mumford (CM) line bundle $\lambda$. We refer the reader to the paragraph after \autoref{eq:CM_line_bundle_log} for the definition of the CM line bundle, and see \autoref{lem:descent_of_CM_line_bundle} for the definition of the descent as well as for the proof of its existence. Our main theorem concerns this prediction:

\begin{theorem}
\label{thm:main}
Fix an integer $n>0$ and a rational number $v>0$, and let $\lambda$ be the CM line bundle on the moduli stack $\sM_{n,v}^{\Kss}$ of $K$-semistable Fano varieties of dimension $n$ and anti-canonical volume $v$.  Let $\pi : \sM_{n,v}^{\Kss} \to \rM_{n,v}^{\Kps}$ be the good moduli space of $\sM_{n,v}^{\Kss}$, and let $L$ be the descent of $\lambda$ along $\pi$. Then:
\begin{enumerate}
\item \label{itm:main:nef} Both $\lambda$ and $L$ are nef.
\item \label{itm:main:big} 
Let $V \subseteq \rM_{n,v}^{\Kps}$ be a proper closed subspace intersecting $\rM_{n,v}^{\uKs}$. Then $L|_V$ is big. 
\item \label{itm:main:q_proj}  If $V \subseteq \rM_{n,v}^{K-ps}$ is a proper closed subspace, then the normalization of $V \cap \rM_{n,v}^{\uKs}$ is a quasi-projective scheme. \end{enumerate}
\end{theorem}

We  address later, in \autoref{rem:why_only_unif_K_stable},  the reasons of the specific generality of \autoref{thm:main}, and we present in  \autoref{sec:intro_logarithmic} our results for pairs.

\begin{remark}
Notably, \autoref{thm:main} deals with non-smoothable singular Fano varieties too, about which we remark that:
\begin{enumerate}
\item This is the first result about (semi-)positivity of the CM line bundle dealing with non-smoothable singular Fano varieties.
\item A typical $K$-semistable Fano variety is non-smoothable. In fact,  smooth Fano varieties of a given dimension  are bounded  regardless of K-semi-stability
\cite{Kollar_Miyaoka_Mori_Rational_connectedness_and_boundedness_of_Fano_manifolds},  and so are smoothable $K$-semi-stable varieties \cite{Jiang_Boundedness_of_Q-Fano_varieties_with_degrees_and___alpha-invariants_bounded_from_below}. On the other hand, non-smoothable $K$-semistable Fano varieties of a given dimension are unbounded if one does not fix the volume, as can be seen by considering quasi-\'etale quotients by bigger and bigger finite subgroups of  $\Aut(\bP^2)$, which quotients are $K$-semi-stable according to \cite[Cor. 1.7]{Fujita_Uniform_K-stability_and_plt_blowups_of_log_Fano_pairs}.

\end{enumerate}

\end{remark}

\begin{remark}
The proof of \autoref{thm:main} uses the Central Limit Theorem of probability theory. See \autoref{sec:outline_semi_positivity} for an outline of our argument or \autoref{prop:vectP2} for the precise place where the Central Limit Theorem is used.
\end{remark}

\begin{remark}
\label{rem:partial_results}
Let $\rM^{\Kps,\sm}_{n,v}$ be the closure of the locus of smooth Fano varieties. Using analytic methods one can show that $\rM^{\Kps,\sm}_{n,v}$ is proper and that $L|_V$ is big on every closed $V \subseteq \rM^{\Kps,\sm}_{n,v}$ intersecting the smooth locus
 \cite{Li_Wang_Xu_Qasi_projectivity_of_the_moduli_space_of_smooth_Kahler_Einstein_Fano_manifolds,Li_Wang_Xu_On_the_proper_moduli_spaces_of_smoothable_Kahler-Einstein_Fano_varieties,SpottiYTD,Odaka_Compact_moduli_spaces_of_Kahler-Einstein_Fano_varieties}. Our theorem extends this to the case of $V$ intersecting the uniformly $K$-stable locus, which then yields the quasi-projectivity of the normalization of an open set of $\rM^{\Kps,\sm}_{n,v}$ that is possibly bigger than the smooth locus.

\end{remark}

\begin{remark}
An equivalent way of stating point \autoref{itm:main:nef} and \autoref{itm:main:big} of \autoref{thm:main} is the following: $\lambda$ and $L$ are nef, and for every  proper closed subspace $V \subseteq \rM_{n,v}^{\Kps}$ the augmented base locus $B_+(L|_V)$ is contained in $V \setminus \rM_{n,v}^{\uKs}$. This follows immediately from  \cite[Thm 0.3]{Nakamaye_Stable_base_loci_of_linear_series}.
\end{remark}

\begin{remark}
Uniformly K-stable Fano varieties have finite automorphism group; this implies that, when $\sM^{\uKs}_{n,v}$  is  smooth, the coarse moduli space $\rM_{n,v}^{\uKs}$ has finite quotient singularities, and hence the normalization in \autoref{thm:main}.\autoref{itm:main:q_proj} can be  dropped from the statement. 

We know that $\sM^{\uKs}_{n,v}$ is smooth at the points corresponding to smooth Fano varieties 
\cite{Ran_Deformations_of_manifolds_with_torsion_or_negative_canonical_bundle,Kawamata_Unobstructed_deformations_A_remark_on_a_paper_of_Z_Ran}, and to terminal Fano $3$-folds
\cite[Thm 1.7]{Sano_On_deformations_of_Q-Fano_3-folds}.  Unfortunately, these unobstructedness statements do not hold for all Fano varieties, for example,  \cite[Rem 2.13]{Sano_On_deformations_of_Q-Fano_3-folds} gives a counterexample. However, the counterexample is a cone over a del Pezzo surface of degree 6, which is not uniformly $K$-stable, as its automorphism groups is not finite. This leads to the following question.
\end{remark}

\begin{question}
 Is the deformation space of uniformly $K$-stable Fano varieties  unobstructed?
\end{question}

\subsection{Technical statements}
\label{sec:intro_logarithmic}

Our most general statements implying \autoref{thm:main}, just as points \autoref{itm:main:nef} and \autoref{itm:main:big} of \autoref{thm:main}, come in two flavors: semi-positivity and positivity statements. We start with the semi-positivity statements, which we are able to show also in the logarithmic case. Let us first present the precise definition of the CM-line bundle in this setting.

If $f : (X, \Delta) \to T$ is a flat morphism of relative dimension $n$ from a projective normal pair to a normal projective variety such that $-(K_{X/T} + \Delta)$ is $\bQ$-Cartier and $f$-ample. Then we define the CM line bundle by 
\begin{equation}
\label{eq:CM_line_bundle_log}
\lambda_{f, \Delta}:= -f_* ((-(K_{X/T} + \Delta))^{n+1} ).
\end{equation}
This cycle, up to multiplying with a positive rational number, is the first Chern class of the
functorial line bundle on $T$ defined in  \cite{Paul_tian1,Paul_tian2}, see also  \autoref{prop:2_defs_CM_same} and \cite{Schumacher,Fine_Ross_A_note_on_positivity_of_the_CM_line_bundle,PRS}. 
In particular, one defines $\lambda$ to be the unique $\bQ$-line bundle $\lambda$ on $\sM^{\Kss}_{n,v}$ such that for every $\nu : T \to \sM^{\Kss}_{n,v}$, if   $f : X \to T$ is the   associated family, then $\nu^* \lambda = \lambda_f:= \lambda_{f,0}$. :

Our most general semi-positivity statements then are the following. We note that by a general geometric fiber we mean a fiber over any geometric point $\Spec \overline{L} \to U$, where $U \subseteq T$ is a fixed  non-empty open set.

\begin{theorem}
\label{thm:semi_positive_boundary}
Let $f : X \to T$ be a flat morphism of relative dimension $n$ with connected fibers between normal projective varieties and let $\Delta$ be an effective $\bQ$-divisor on $X$  such that $-(K_{X/T} + \Delta)$ is $\bQ$-Cartier and $f$-ample. Let $\lambda_{f,\Delta}$ be the CM  line bundle on $T$ as defined in \autoref{eq:CM_line_bundle_log}. 
\begin{enumerate}
\item {\scshape Pseudo-effectivity:} \label{itm:semi_positive_boundary:pseff} If $T$ is smooth and  $(X_t, \Delta_t)$ is K-semi-stable  for general geometric fibers $X_t$, then $\lambda_{f, \Delta}$ is pseudo-effective.
\item {\scshape Nefness:} \label{itm:semi_positive_boundary:nef} If all fibers $X_t$ are normal, $\Delta$ does not contain any fibers (so that we may restrict $\Delta$ on the fibers), and $(X_t, \Delta_t)$ is $K$-semi-stable for all geometric fibers $X_t$, then $\lambda_{f, \Delta}$ is nef. 
\end{enumerate}
\end{theorem}

Next we  state our positivity statements. These pertain to families with maximal variation. Here, a family $f: X\to T$ of Fano varieties as in \autoref{thm:semi_positive_no_boundary} has maximal variation if there is a non-empty open set of $T$ over which the isomorphism equivalence classes of the fibers are finite. In the logarithmic case one faces considerable extra difficulties when the variation comes partially also from the variation of the boundary, as it was also the case for the KSBA stable moduli \cite{Kovacs_Patakfalvi_Projectivity_of_the_moduli_space_of_stable_log_varieties_and_subadditvity_of_log_Kodaira_dimension}. Hence, to keep the length of the article under control, we address here only the question of positivity in the boundary free case. The logarithmic version was addressed after the initial submission of the present article in  \cite{Quentin}.

\begin{theorem}
\label{thm:semi_positive_no_boundary}
Let $f : X \to T$ be a flat morphism with connected fibers between normal projective varieties such that $-K_{X/T}$ is $\bQ$-Cartier and $f$-ample, and let $\lambda_f$ be the CM line bundle defined in equation \autoref{eq:CM_line_bundle_log}. 
\begin{enumerate}
\item \label{itm:semi_positive_no_boundary:big} {\scshape Bigness:} If  $T$ is smooth, the  general geometric fibers of $f$ are uniformly $K$-stable,  the variation of $f$ is maximal, and either $\dim T=1$ or the fibers of $f$ are reduced, then {$\lambda_f$ is big}.
\item \label{itm:semi_positive_no_boundary:ample} {\scshape Ampleness:} If all the geometric fibers of $f$ are uniformly $K$-stable and the isomorphism equivalence classes of the closed fibers are finite, then {$\lambda_f$ is ample}.
\item \label{itm:semi_positive_no_boundary:q_proj} {\scshape Quasi-projectivity:} If $T$ is only assumed to be a proper normal algebraic space, all the geometric fibers are $K$-semi-stable and there is an open set $U \subseteq T$ over which the geometric fibers are uniformly $K$-stable and the isomorphism classes of the fibers are finite, then $U$ is a quasi-projective variety. 

\end{enumerate}
\end{theorem}

\begin{remark}
\label{rem:openness}
We note that both, $K$-semistability \cite[Thm 1.1]{Blum_Liu_Xu_Openness_of_K-semistability_for_Fano_varieties} \cite[Thm 1.4]{Xu_A_minimizing_valuation_is_quasi-monomial} and uniform $K$-stability 
\cite[Thm A]{Blum_Liu_Openness_of_uniform_K-stability_in_families_of_Q-Fano___varieties} are open properties. 

We also remark that in \autoref{thm:semi_positive_no_boundary} we carefully said ``geometric fiber'' instead of just ``fiber''. The reason is that we use the $\delta$-invariant description of $K$-stability, and the $\delta$-invariant of a variety is not invariant under base extension to the algebraic closure, see \autoref{rem:delta_invariant_non_albraically_closed_field}. So, for scheme theoretic fibers over non algebraically closed  fields the $\delta$-invariant can have non semi-continuous behavior.  
\end{remark}

\begin{remark}
We proved \autoref{thm:semi_positive_boundary} and \autoref{thm:semi_positive_no_boundary} in the stated generality, as in this setting the relative canonical divisor exists and admits reasonable base-change properties  on very general curves in moving families of curves on the base, see \autoref{sec:relative_canonical} for details.
Nevertheless, in situations where this base-change is automatic, \autoref{thm:semi_positive_no_boundary} directly implies statements over non-normal, non-projective, and even non-scheme bases. This is made precise for example in the following statement:
\end{remark}

\begin{corollary}
\label{cor:proper_base}
Let $f : X \to T$ be a flat, projective morphism with connected fibers to a proper algebraic space, such that there is an integer $m > 0$ for which  $\omega_{X/T}^{[m]}$  is a line bundle and all  the geometric fibers are $K$-semi-stable klt Fano varieties. Let $N$ be the CM-line bundle associated to the polarization $\omega_{X/T}^{[-m]}$. Then, $N$ is nef, and if the variation of $f$ is maximal and the very general geometric fiber is uniformly $K$-stable, then $N$ is big. 
\end{corollary}
The CM line bundle over a general base is defined in \autoref{notation:Paul_Tian}, following \cite{Paul_tian2}.

\begin{remark}
Note that over $\bC$ the positivity properties of \autoref{thm:semi_positive_no_boundary} (nefness, pseudo-effectivity, bigness, ampleness) can be also characterized  analytically, e.g., \cite[Prop 4.2]{Demailly_Singular_Hermitian_metrics_on_positive_line_bundles}
\end{remark}

\begin{remark}
{\scshape Negativity of $-K_{X/T}$ point of view.}
Unwinding definition \autoref{eq:CM_line_bundle_log}, 
we obtain that, in the case of one dimensional base, \autoref{thm:semi_positive_no_boundary} states that $(-K_{X/T})^{n+1}$ is at most zero/smaller than 0. Using this in conjunction with the base-change property of the CM line bundle proved in \autoref{prop:CM_base_change} we obtain that \autoref{thm:semi_positive_no_boundary}, especially the last 3 points, prove strong negativity properties of $-K_{X/T}$ for families of klt Fano varieties. For example, one obtains that if $C \to T$ is a general enough curve, then the top self intersection of $(-K_{X/T})|_{f^{-1}C}$ is negative. 

There do exist birational geometry statements claiming that $-K_{X/T}$ is not nef, e.g., \cite[Prop 1]{Zhang_On_projective_manifolds_with_nef_anticanonical_bundles}. Our negativity statements point in this direction but go further. However, it is not a coincidence that strong negativity statements on $-K_{X/T}$ did not show up earlier, as in fact \autoref{thm:semi_positive_no_boundary} is not true for every family of klt Fano varieties. Indeed, \autoref{ex:negative_degree} shows that   in  \autoref{thm:semi_positive_no_boundary} one cannot relax the $K$-semi-stable Fano assumption to just assuming klt Fano. 
The development of the notions of $K$-stability in the past decade was essential for creating the chance of proving negativity statements for $-K_{X/T}$ of the above type. 

We also note that as $-K_{X/T}$ is not nef usually in the  situation of \autoref{thm:semi_positive_no_boundary}, c.f., \autoref{thm:nef_threshold} and \autoref{ex:positive_and_big}, the negativity of $(-K_{X/T})^{n+1}$ is independent of the negativity of $\kappa(-K_{X/T})$. In fact,  assuming the former, $\kappa(-K_{X/T})$ can be $- \infty$ (\autoref{ex:anti_canonical_no_section}), 
$\dim X$ (\autoref{ex:positive_and_big}), and also something in between $-\infty$ and $\dim X$ (\autoref{ex:anti_canonical_no_section}). 
\end{remark}

\begin{remark}
\label{rem:why_only_unif_K_stable}
There are two main reasons why our positivity statements \autoref{itm:semi_positive_no_boundary:big}, \autoref{itm:semi_positive_no_boundary:ample} and \autoref{itm:semi_positive_no_boundary:q_proj} of \autoref{thm:semi_positive_no_boundary} work in the uniformly $K$-stable case, but not in the $K$-polystable case:
\begin{enumerate}
\item We rely on the characterization of $K$-semistability and uniform $K$-stability via the $\delta$ invariant given by \cite{Fujita_Odaka_On_the_K-stability_of_Fano_varieties_and_anticanonical_divisors,BJ}. Such characterization is not available for the $K$-polystable case.
\item Our \autoref{thm:nef_threshold}  about the nef threshold, on which the above 3 points of \autoref{thm:semi_positive_no_boundary} rely, fails in the $K$-polystable case according to \autoref{ex:not_nef}. Hence, one would need a significantly different approach to extend points \autoref{itm:semi_positive_no_boundary:big}, \autoref{itm:semi_positive_no_boundary:ample} and \autoref{itm:semi_positive_no_boundary:q_proj} of \autoref{thm:semi_positive_no_boundary} to the $K$-polystable case.
\end{enumerate}

\end{remark}

\begin{remark}
One could make definition \autoref{eq:CM_line_bundle_log} also without requiring flatness. We do not know if \autoref{thm:semi_positive_no_boundary} holds in this situation. Nevertheless, we note that it would be interesting to pursue this direction for example for applications to Mori-fiber spaces with higher dimensional bases, see \autoref{cor:bounding_volume}.  

Also we expect that the reduced fiber assumption of point \autoref{itm:semi_positive_no_boundary:big} of \autoref{thm:semi_positive_no_boundary} can be removed, as we need it for technical reasons, namely we want the base changes over normalizations of  general elements of movable families of curves to be nice, and also because the conjectured  $K$-semi-stable reduction should eliminate it.  Here, $K$-semi-stable reduction means the conjecture that $K$-semsitable  families of Fano varieties over function fields of DVR's can be extended over the DVR after a finite base-change. 
\end{remark}

\subsection{Boundedness of the volume}
\label{sec:intro_applications}

Fujita showed in \cite[Thm 1.1]{Fujita_Optimal_bounds_for_the_volumes_of_Kahler-Einstein_Fano_manifolds} that $\vol(-K_X) \leq (n+1)^n$ for every $K$-semistable Fano variety $X$ of dimension $n$, see \cite[Thm 3]{Liu_The_volume_of_singular_Kahler-Einstein_Fano_varieties} for better bounds in the presence of quotient singularities. Using \autoref{thm:semi_positive_boundary} we can show similar bounds for Fano varieties $X$ admitting a  Fano fibration structure with $K$-semi-stable general fiber. 

\begin{corollary}
\label{cor:bounding_volume}
If $(X,\Delta)$ is a normal Fano pair, and $f : (X, \Delta) \to \bP^1$ is a fibration with $K$-semi-stable  general geometric fibers $(F, \Delta_F)$, then 
\begin{equation*}
\vol(-(K_X+\Delta)) \leq 2 \dim  \left( X\right) \vol (-(K_F+\Delta_F)) \leq 2\dim \left(X\right)^{\dim \left(X\right)}.
\end{equation*}
\end{corollary}

\begin{remark}
\autoref{cor:bounding_volume} is sharp  for surfaces and threefolds. Indeed, a del Pezzo surface of degree $8$ and the blow-up of $\bP^3$ at a line, whose anti-canonical volume is $54$, can be fibred over $\bP^1$ with $K$-semis-table fibres. 
\end{remark}

\begin{remark}
\label{rem:K_stability_classification}
{ \scshape Classification of (uniform) $K$-(semi/poly)-stable Fano varieties}: to explain which varieties \autoref{cor:bounding_volume} pertains to, we provide a short list on Fano varieties that are either known to be $K$-semi-stable or not $K$-semi-stable. In fact, one typically wants to figure out for a given Fano variety the behavior with respect to all four $K$-stability properties, see \autoref{sec:K_stability_versions}. This has been an active area of research recently. To start with, let us recall that $K$-semi-stable Fano varieties are always klt. 

A Del-Pezzo surface is K-polystable if and only if it is not  of degree $8$ or $7$ \cite{Tian_Yau_Kahler-Einstein_metrics_on_complex_surfaces_with_C_1_at_least_zero,Tian_On_Calabis_conjecture_for_complex_surfaces_with_positive_first_Chern_class}. 
Smooth Fano surfaces with discrete automorphism groups are even uniformly K-stable, and their delta invariant, see \autoref{sec:delta}, is bounded away from $1$ in an effective way \cite{Park_Won_K-stability_of_smooth_del_Pezzo_surfaces}. Smoothable singular K-stable Del-Pezzo surfaces are classified in \cite{Odaka_Spotti_Sun}.

K-stable proper intersection of two quadrics in an odd dimensional projective space are classified in \cite{Spotti_Sun_Explicit_Gromov-Hausdorff_compactifications_of_moduli_spaces_of_Kahler-Einstein_Fano_manifolds}, see also \cite{Arezzo_Ghigi_Pirola_Symmetries_quotients_and_Kahler-Einstein_metrics}; in particular, smooth varieties of these types are always K-stable. Cubic $3$-folds are studied in \cite{Liu_Xu_K-stability_of_cubic_threefolds}, where again smooth ones are $K$-stable, and so are the ones containing only $A_k$ singularities for $k \leq 4$. Under adequate hypotheses, in \cite{Dervan_Finite}, it is shown that Galois covers of K-semistable Fano varieties are K-stable. This can be applied for instance to double solids. Furthermore, birational superigid Fano varieties are K-stable under some addition mild hypothesis \cite{OO13,Stibitz_Zhuang_K-stability_of_birationally_superrigid_Fano_varieties,Z}. However, according to the best knowledge of the authors, there is not a complete classification of K-stable smooth Fano threefolds.

If one wants to study klt Fano varieties from the point of view of the MMP, it is particularly relevant to see if one can apply \autoref{cor:bounding_volume} to the case of Mori Fibre Spaces with one dimensional bases. In \cite[Corollary 1.11]{CFST}, it is shown that if a smooth Fano surface or a smooth toric variety can appear as a fibre of MFS, then it is K-semistable. We do not know if the analogous  result holds in dimension $3$. However, there are examples of smooth Fano fourfolds with Picard number one, which then can be general fibers of  MFS's, that are not K-semistable \cite{Fujita_Examples}, see also \cite{Codogni2018}.

\end{remark}

\subsection{Byproduct statements}

As a byproduct of our technique for proving \autoref{thm:semi_positive_boundary}, we obtain the following bound on the nef threshold of $-(K_{X/T} + \Delta)$ with respect to $\lambda_{f, \Delta}$ in the uniformly $K$-stable case. 

\begin{theorem}
\label{thm:bounding_nef_threshold}
\label{thm:nef_threshold}
Let $f : X \to T$ be a flat morphism with connected fibers from a normal projective variety of dimension $n+1$ to a smooth curve and let $\Delta$ be an effective $\bQ$-divisor on $X$  such that 
\begin{itemize}
 \item $-(K_{X/T} + \Delta)$ is $\bQ$-Cartier and $f$-ample, and 
 \item $\left(X_{\ot},\Delta_{\ot}\right)$ is uniformly $K$-stable for fibers $X_{\ot}$ over  general geometric points $\ot \in T$.
\end{itemize}
Set
\begin{itemize}
 \item \label{itm:thm:bounding_nef_threshold_Delta} set $\delta:=\delta\left(X_{\ot}, \Delta_{\ot}\right)$ for $\ot$ very general geometric point, and
 \item let $v:=\left( (-K_{X/T} - \Delta)_t \right)^n$ for any $t \in T$.
\end{itemize}
Then, $- K_{X/T} - \Delta + \frac{\delta}{(\delta -1) v (n+1)} f^* \lambda_{f,\Delta} $ is nef. 
\end{theorem}

Recall  that the uniformly $K$-stable assumption in \autoref{thm:nef_threshold} is equivalent to assuming $\delta >1$, see \autoref{def:K_stability} and \autoref{def:K_stable}. In particular, $\delta -1 >0$ in the last line of the statement. 

\begin{remark}
The reason for assuming in \autoref{thm:nef_threshold} that $\left(X_{\ot},\Delta_{\ot}\right)$ is uniformly $K$-stable for general geometric fibers, but setting $\delta$ to be $\delta\left(X_{\ot}, \Delta_{\ot}\right)$ only for \emph{very} general geometric fibers is technical. On one hand, uniform $K$-stability is known to be an open property by \cite[Thm A]{Blum_Liu_Openness_of_uniform_K-stability_in_families_of_Q-Fano___varieties}, and hence one may assume it on the general geometric fiber without imposing any additional assumption. On the other hand, only the function  $\ot \mapsto \min\{1,\delta\left(X_{\ot}, \Delta_{\ot}\right)\}$, but not $\ot \mapsto \delta\left(X_{\ot}, \Delta_{\ot}\right)$ itself, is  known to be constructible \cite[Prop 4.3]{Blum_Liu_Xu_Openness_of_K-semistability_for_Fano_varieties}. For $\ot \mapsto \delta\left(X_{\ot}, \Delta_{\ot}\right)$, it is only known that it is constant on the complement of countably many closed sets by \autoref{prop:delta_general_fiber}.
\end{remark}

\begin{remark}
One cannot have a nef threshold statement, as   in \autoref{thm:nef_threshold}, for $K$-polystable Fano varieties instead of uniformly $K$-stable ones. Indeed, take the family $f : X \to T$ given by  \autoref{ex:not_nef}. It has $K$-polystable fibers, $\deg \lambda_f=0$, but $'K_{X/T}$ is not nef. In particular, for any $a \in \bQ$,  $-K_{X/T}+ a f^*\lambda_f \equiv - K_{X/T}$, and hence for any $a \in \bQ$, $-K_{X/T}+ a f^*\lambda_f$ is not nef.

\end{remark}

We also recover a structure theorem when the CM line bundle $\lambda_f$ is not positive:

\begin{theorem}
\label{thm:lambda_non_big}
Let $f : X \to T$ be a flat morphism of relative dimension $n$ with connected fibers between normal projective varieties and let $\Delta$ be an effective $\bQ$-divisor on $X$  such that $-(K_{X/T} + \Delta)$ is $\bQ$-Cartier and $f$-ample. Assume that $\left(X_{\ot}, \Delta_{\ot}\right)$ is uniformly $K$-stable for fibers $X_{\ot}$ over general geometric points $\ot \in T$. If $H$ is an ample divisor on $T$, such that  $\lambda_{f,\Delta} \cdot H^{\dim T -1} =0$, then  for every integer $q>0$ divisible enough,  $f_* \sO_X(q (-K_{X/T} - \Delta))$ is an $H$-semi-stable vector bundle of slope $0$.
\end{theorem}

\begin{corollary}
\label{cor:lambda_trivial_fiber_bundle}
Assume $k=\bC$, and let $f : X \to T$ be a surjective morphism from a normal projective variety of dimension $n+1$ to a smooth, projective curve  such that $-K_{X/T} $ is $\bQ$-Cartier and $f$-ample, and the general fiber of $f$ is uniformly $K$-stable. Then, $\deg \lambda_{f}=0$ if and only if, $f$ is analytically locally a fiber bundle. 
\end{corollary}

\subsection{Similar results in other contexts}

Roughly, there are three types of statements above: (semi-)positivity results, moduli applications, inequality of volumes of fibrations. Although in the realm of $K$-stability ours are the first general algebraic results, statements of these types were abundant in other, somewhat related, contexts: KSBA stability, GIT stability, and just general algebraic geometry. Our setup and our methods are different from these results, still we briefly list some of them for completeness of background. We note that KSBA stability is related to our framework as it is shown to be exactly the canonically polarized $K$-stable situation \cite{Odaka_The_GIT_stability_of_polarized_varieties_via_discrepancy,Odaka_The_Calabi_conjecture_and_K-stability,Odaka_Xu_Log-canonical_models_of_singular_pairs_and_its_applications}. Also, GIT stability is related, as $K$-stability originates from an infinite dimensional GIT, although it is shown that it cannot be reproduced using GIT, e.g., \cite{Wang_Xu_Nonexistence_of_asymptotic_GIT_compactification}.

\noindent
\begin{tabular}{|l|p{110pt}|p{110pt}|p{77pt}|}
\hline
& general algebraic geometry & KSBA stability & GIT stability \\
\hline
(semi-)positivity & \cite{Griffiths_Periods_of_integrals,Fujita_On_Kahler_fiber_spaces,Kawamata_Characterization_of_abelian_varieties,Viehweg_Weak_positivity,Kollar_Subadditivity_of_the_Kodaira_dimension} & \cite{Kollar_Projectivity_of_complete_moduli,Fujino_Semi_positivity_theorems_for_moduli_problems,Kovacs_Patakfalvi_Projectivity_of_the_moduli_space_of_stable_log_varieties_and_subadditvity_of_log_Kodaira_dimension,Patakfalvi_Xu_Ampleness_of_the_CM_line_bundle_on_the_moduli_space_of_canonically_polarized_varieties} & \cite{CH} \\
\hline
moduli applications & \cite{Viehweg_Quasi_projective_moduli} & \cite{Kollar_Projectivity_of_complete_moduli,Kovacs_Patakfalvi_Projectivity_of_the_moduli_space_of_stable_log_varieties_and_subadditvity_of_log_Kodaira_dimension,Ascher_Bejleri_Moduli_of_weighted_stable_elliptic_surfaces_and_invariance_of_log___plurigenera} & \cite{CH} \\
\hline
volume and slope inequalities & \cite{Xiao_Fibered_algebraic_surfaces_with_low_slope} &  & \cite{Pardini,Stoppino}\\
\hline
\end{tabular}

\subsection{Overview of K-stability for Fano varieties}
\label{sec:K_stability_versions}

In the present article we define $K$-semi-stability and uniform $K$-stability using valuations, see \autoref{def:K_stability}, which is equivalent then to  the $\delta$-invariant definition given in \autoref{def:K_stable}. These definitions were shown to be equivalent  to the more traditional ones that use test configurations \cite[Theorem B]{BJ}. However, this approach has a considerable  disadvantage: there is no known delta invariant type definition of $K$-stability and $K$-polystability. While we do not use these notions in any  of the statements or proof of our results, we believe that they are important notions in the study of Fano varieties. Hence, for completeness we state the classical definitions involving  test configurations for all the four notions of $K$-stability. We refer the reader to \cite{Tian_Test,Don_Calabi} or more recent papers such as \cite{Dervan_Uniform,Boucksom_Uniform} for more details. 

For a Fano variety $X$ we mention the following notions of $K$-stability:
\begin{description}
\item[K-semi-stability] For every normal test configuration of $X$, the Donaldson-Futaki invariant is non-negative.
\item[K-stability]  For every normal test configuration of $X$, the Donaldson-Futaki invariant is non-negative, and it is equal to zero if and only if the test configuration is a trivial test configuration. In particular, there is no $1$-parameter subgroup of  $\Aut(X)$.
\item[K-poly-stability]  For every normal test configuration of $X$ the Donaldson-Futaki invariant is non-negative, and it is equal to zero if and only if the test configuration is a product test configuration, i.e. it comes from a one parameter subgroup of the automorphism group of $X$.
\item[Uniform K-stability] There exists a positive real constant $\delta$ such that for every normal test configuration of $X$ the Donaldson-Futaki invariant is at least $\delta$ times the $L^1$ norm (or, equivalently, the minimum norm) of the test configuration. This notion implies K-stability, and when $X$ is smooth the finiteness of the automorphism group of $X$, too \cite[Cor E]{Boucksom_Hisamoto_Jonsson_Uniform_K-stability_and_asymptotics_of_energy_functionals_in_Kahler_geometry}.
\end{description}
We also note that the Yau-Tian-Donaldson (in short, YTD) conjecture asserts that a klt Fano variety admits a singular K\"{a}hler-Einstein metric if and only if it is K-polystable. This is known to hold for smooth \cite{Chen_Donaldson_Sun_Kahler-Einstein_metrics_on_Fano_manifolds_I_Approximation_of_metrics_with_cone_singularities,Chen_Donaldson_Sun_Kahler-Einstein_metrics_on_Fano_manifolds_II_Limits_with_cone_angle_less_than_2pi,Chen_Donaldson_Sun_Kahler-Einstein_metrics_on_Fano_manifolds_III_Limits_as_cone_angle_approaches_2pi_and_completion_of_the_main_proof,Tian} and smoothable Fano varieties \cite{Li_Wang_Xu_On_the_proper_moduli_spaces_of_smoothable_Kahler-Einstein_Fano_varieties}, and independently  \cite{SpottiYTD} in the finite automorphism case, and for singular ones admitting a crepant resolution \cite{Recent_Tian}. In the literature, there are also many proposed strenghtenings of the notion of K-stability; they should be crucial to extend the YTD conjecture to the case of constant scalar curvature K\"{a}hler metrics. In this paper we are interested in uniform K-stability \cite{Dervan_Uniform,Boucksom_Uniform,Bouck_variational}, which at least for smooth Fano manifold is known to be equivalent to K-stability (we should stress that the proof is via the equivalence with the existence of a  K\"{a}hler-Einstein metric). One can also strenghten the notion of K-stability by possibly looking at  non-finitely generated filtrations of the coordinate ring, see \cite{Nystrom_filtrations,Gabor,Tits}.

\subsection{Outline of the proof}
\label{sec:outline}

Our proof for the semi-postivity and the positivity statements for the CM line bundle are different. Hence, we discuss the corresponding outlines  separately in \autoref{sec:outline_semi_positivity} and in \autoref{sec:outline_positivity}, respectively. Additionally, as it is an indispensable link between semi-positivity and positivity, we present the ideas behind the nefness threshold statement of \autoref{thm:bounding_nef_threshold} in \autoref{sec:outline_nefness_threshold}. For simplicity, we restrict in all cases to the non-logarithmic situation, that is, to statements about $-K_{X/T}$ instead of $-(K_{X/T} + \Delta)$. As all the assumptions and consequences are invariant under base-extension to another algebraically closed field, we may also assume that $k$ is uncountable. In particular, the very general geometric fibers whose existence is assumed in the statements also show up as closed fibers. 

\subsubsection{ Semi-positivity statements.} 
\label{sec:outline_semi_positivity}
As nefness and pseudo-effectivity can be checked via non-negative intersection with effective or moving 1-cycles, respectively, points \autoref{itm:semi_positive_boundary:pseff} and \autoref{itm:semi_positive_boundary:nef} of \autoref{thm:semi_positive_boundary} can be reduced to the case of one dimensional base. Hence, we assume that the base of our fibration $f: X \to T$ is a curve, in which case pseudo-effectivity and nefness are both equal to the degree being at least zero. So, we are supposed to prove that $\deg \lambda_f \geq 0$ or equivalently that $(-K_{X/T})^{n+1} \leq 0$, see \autoref{eq:CM_line_bundle_log}. 

We argue by contradiction, so we assume that $(-K_{X/T})^{n+1}>0$. If we fix a $\bQ$-divisor $H$ on $T$ of small enough positive degree, then by the continuity of the intersection product $(-K_{X/T} - f^*H )^{n+1}>0$ also holds.   As $X$ is normal and fibered over the curve $T$ over which $-K_{X/T}$ is ample, this implies via a Riemann-Roch computation that the $\bQ$-linear system $|-K_{X/T} - f^* H|_{\bQ}$ is non-empty, see \autoref{rem:big}. Our initial idea is to obtain a contradiction from this fact: in fact, \autoref{prop:cont} shows that there are no $\Gamma \in |-K_{X/T} - f^* H|_{\bQ}$ such that $(X_t,\Gamma_t)$ is klt for general $t \in T$. The only problem is that there are examples where $|-K_{X/T} - f^* H|_{\bQ}$ is non-empty such that  for every $\Gamma \in |-K_{X/T} - f^* H|_{\bQ}$, the pair $(X_t,\Gamma_t)$ is not klt for general $t \in T$. Indeed, every family with negative CM line bundle  has to satisfy the conditions stated in the previous sentence,  according to \autoref{prop:cont}. An explicit example is given in \autoref{ex:negative_degree}.

Our second idea is that maybe the $K$-stable assumption leads us to a $\Gamma$ as above that also satisfies the klt condition when restricted to a general fiber. According to the delta invariant description of $K$-semi-stability (\autoref{def:K_stable}), if $X_t$ is $K$-semi-stable, then up to a little perturbation one can obtain klt divisors the following way: for $q \gg 0$, let $D_1,\dots,D_l$ be divisors corresponding to any basis of $H^0\left(X_t, -qK_{X_t} \right)$; then the divisor $D := \displaystyle\sum_{i=1}^l \frac{D_i}{ql} \in |-K_{X_t}|_{\bQ}$ 
is such that $(X_t, D)$ is klt. 

Now, we would like to lift such a divisor to $|-K_{X/T} - f^* H|_{\bQ}$. To this end, it is enough to lift for $q \gg 0$, every element of a basis of $H^0\left(X_t, -qK_{X_t} \right)$ to elements of $H^0(X,q(-K_{X/T} - f^* H))$. Using some perturbation argument, it suffices to show the existence of linearly independent sections $s_1,\dots,s_l \in H^0\left(X_t, -qK_{X_t} \right)$ such that $s_i$ lifts, and $\frac{l}{h^0\left(-qK_{X_t} \right)}$ is close enough to $1$. 

This in turn would be implied by the following: let $\sE_q$ be the subsheaf of $f_* \sO_X(-qK_{X/T})$ spanned by the global sections, then 
\begin{equation}
\label{eq:outline_goal}
\displaystyle\lim_{q \to \infty} \frac{\rk \sE_q}{\rk f_* \sO_X(q(-K_{X/T}-f^*H))}=1.
\end{equation}
For the readers more familiar with the language of volumes and restricted volumes, we note that \autoref{eq:outline_goal} is equivalent to showing that the restricted volume of $-K_{X/T}$ over a general fiber is equal to the anti-canonical volume of the fibers.

Unfortunately, \autoref{eq:outline_goal}   still does not hold. For example, if one takes the isotrivial family 
\begin{equation*}
X:=\bP_{T} (\sO_{T}(-n)\oplus \underbrace{\sO_T(1)\oplus \dots \oplus \sO_T(1)}_{\textrm{$n$ times}}) 
\end{equation*}
of $\bP^n$'s over $T:= \bP^1$ (as in \autoref{ex:negative_degree} for $n=2$), then 
\begin{equation*}
f_* \sO_X(-qK_{X/T}) \cong S^{(n+1)q}(\sO_{T}(-n)\oplus \sO_T(1)\oplus \dots \oplus \sO_T(1)). 
\end{equation*}
 In this situation $\sE_q$ is the direct sum of the factors with  degree greater than $q \deg H \sim q \varepsilon$ (here $1 \gg \varepsilon>0$). Then one can compute that \autoref{eq:outline_goal} does not hold. For example, in the case of $n=1$,
\begin{equation*}
S^{2q}(\sO_T(-1) \oplus \sO_T(1)) = \sO_T(-q) \oplus \sO_T(-q+1) \oplus \dots \oplus \sO_T(q). 
\end{equation*}
 So, we see that the limit of \autoref{eq:outline_goal} is $\frac{1}{2} - \varepsilon$. 

The idea that saves the day at this point is the \emph{product trick}, which was pioneered in the case of semi-posivity questions by Viewheg \cite{Viehweg_Weak_positivity}. The precise idea is to replace $X$ by an $m$-times self fiber product $X^{(m)}$ over $T$. Let $f^{(m)} : X^{(m)} \to T$  be the induced morphism, \autoref{sec:product_notation}. Then, one can replace the initial goal with showing that there exists  $\Gamma \in \left| - K_{X^{(m)}/T} - \left( f^{(m)}\right)^* m H \right|_{\bQ}$ such that $\left(X^{(m)}_t, \Gamma_t \right)$ is klt for $t \in T$ general. Running through the previous arguments for $X^{(m)}$ instead of $X$,  this would boil down to showing that 
\begin{equation}
\label{eq:outline_goal_2}
\displaystyle\lim_{m \to \infty} \frac{\rk \sE_{q,m}}{\rk f_*^{(m)} \sO_{X^{(m)}} \left(q\left(-K_{X^{(m)}/T}- \left(f^{(m)} \right)^* m H\right)\right) }=1,
\end{equation}
where $\sE_{q,m}$ is a  subsheaf given by certain condition specified below of the subsheaf generated by global sections   of 
\begin{equation}
\label{eq:tensor_product_bundle_outline}
f_*^{(m)} \sO_{X^{(m)}} \left(q\left(-K_{X^{(m)}/T}- \left(f^{(m)} \right)^* m H\right)\right) \cong \bigotimes_{\textrm{$m$ times}}f_* \sO_X(q(-K_{X/T}-f^*H)).
\end{equation}
The extra condition in the definition of $\sE_{q,m}$ is due to the need that $\Gamma$ has to be klt on a general fiber. This would be automatic if  the conjecture that products of $K$-semi-stable klt Fano varieties are $K$-semi-stable was known. Unfortunately, this is a surprisingly hard unsolved conjecture in the theory of $K$-stability\footnote{This conjecture has been proved in \cite{Product}, published after the first version of this paper has appeared.}. Hence, we elude it by considering only bases of $H^0\left( X^{(m)}_t, -q K_{X^{(m)}_t} \right) \cong \displaystyle\bigotimes_{m \textrm{ times}} H^0 \left( X_t, -q K_{X_t}\right)$ that are induced from bases of $H^0 \left( X_t, -q K_{X_t}\right)$. As log canonical thresholds are known to behave well under taking products, see \autoref{prop:prod_basis}, 
if the restriction   $\Gamma|_{X_t^{(m)}}$ to a general fiber is  a divisor corresponding to such basis, the $K$-stability of $X_t$ implies that $\left(X_t^{(m)}, \Gamma|_{X_t^{(m)}} \right)$ is klt. Hence, the additional condition in the definition of $\sE_{q,m}$ is that it is the biggest subsheaf as above such that $\left( \sE_{q,m} \right)_t$ is spanned by simple tensors for a basis $t_1,\dots,t_l$ of $\left( f_* \sO_X(q(-K_{X/T}-f^*H)) \right)_t$ to be specified soon.

So, we are left to specify a basis of $\left( f_* \sO_X(q(-K_{X/T}-f^*H)) \right)_t \cong H^0(X_t, -qK_{X_t})$ for which \autoref{eq:outline_goal_2} holds. For that we use the Harder-Narasimhan filtration $0=\sF^0 \subseteq \dots \subseteq \sF^r$ of $f_* \sO_X(q(-K_{X/T}-f^*H))$. Let the basis $v_1,\dots, v_l$ be any basis adapted to the restriction of this filtration over $t$, that is, to $0=\sF^0_t \subseteq \dots \subseteq \sF^r_t$. The lower part of the filtration, until  the graded pieces reach slope $2g$, where $g$ is the genus of $T$,   is globally generated. Furthermore, there is an induced Harder-Narasimhan filtration on the sheaf in \autoref{eq:tensor_product_bundle_outline}. The part of slope at least $2g$ in the last filtration that we defined is globally generated such that its restriction over $t \in T$ is generated by simple tensors in $v_i$, \autoref{prop:HN_Tensor_product}. Hence, if $\sE_{q,m}'$ is this part of the Harder-Narasimhan filtraton, then it is enough to prove that 
\begin{equation}
\label{eq:outline_goal_3}
\displaystyle\lim_{m \to \infty} \frac{\rk \sE_{q,m}'}{\rk f_*^{(m)} \sO_{X^{(m)}} \left(q\left(-K_{X^{(m)}/T}- \left(f^{(m)} \right)^* m H\right)\right) }=1,
\end{equation}
The final trick of the semi-positivity part is then that \autoref{eq:outline_goal_3} can be translated to a probability limit, which then is implied by the central limit theorem of probability theory, see \autoref{prop:vectP2}.

We explain here the probability theory argument via the example of 
\begin{equation*}
\sF_m:=\displaystyle\bigotimes_{m \textrm{ times}} (\sO_{\bP^1}(-1) \oplus \sO_{\bP^1}(2)). 
\end{equation*}
 The claim then is that as $m$ goes to infinity the rank of the non-negative degree part of $\sF_m$ over the rank of $\sF_m$ converges to $1$. It is easy to see that this is the  limit of the left hand side of the following equation as $m$ goes to infinity:
\begin{equation*}
\sum_{0 \leq i \leq m, 2i - (m-i) \geq 0 } {m \choose i} \left(\frac{1}{2} \right)^m =  
\underbrace{\sum_{0 \leq i \leq m, i \geq \frac{m}{3} } {m \choose i} \left(\frac{1}{2} \right)^m \geq \sum_{0 \leq i \leq m, i \geq \frac{m}{2}- A\frac{\sqrt{m}}{4} } {m \choose i} \left(\frac{1}{2} \right)^m }_{\parbox{220pt}{\footnotesize for $m$ big enough, where $A>0$ is an arbitrary fixed real number, independent of $m$}}
\end{equation*}
The last summation appearing in the previous equation is equal to the probability of getting at least $\frac{m}{2} - A\frac{\sqrt{m}}{4}$ heads when flipping a coin $m$ times. Note that for this $m$-times flipping the expected value is $\frac{m}{2}$ and $\sqrt{m}$-times the square deviation is $\frac{\sqrt{m}}{4}$. Hence, the above probability converges to $\int_{-A}^\infty \frac{1}{\sqrt{2 \pi}} e^{\frac{-x^2}{2}} dx$  by the classical De Moivre-Laplace theorem, a special case of the central limit theorem. We obtain \autoref{eq:outline_goal_3} by taking $A \to \infty$ limit, and using that the above integral integrates the density function of the standard Gaussian normal distribution.

\subsubsection{ Nefness threshold, that is, \autoref{thm:nef_threshold}.} 
\label{sec:outline_nefness_threshold}
This part uses the same ideas as the above semi-positivity part, but in a different logical framework. That is, the argument is not a proof by contradiction. Instead, the starting point is that $\left(-K_{X/T} + \left( f^{(m)} \right)^* \left(\frac{\lambda_f}{v (n+1)}  + H \right) \right)^{n+1}>0$. Hence, again up to a little perturbation and by using the ideas of the previous point, there is an integer $m>0$ such that there exists a  $\Gamma \in \left|-\delta K_{X^{(m)}/T} + \left( f^{(m)} \right)^* m \left(\frac{\delta \lambda_f}{v (n+1)}  + H \right) \right|_{\bQ}$ for which  $\left( X^{(m)}_t,\Gamma_t \right)$ is klt for $t \in T$ general. Then standard semi-positivity argument (\autoref{prop:semi_positivity_engine_downstairs}) shows that
\begin{equation*}
K_{X^{(m)}/T} -\delta K_{X^{(m)}/T} + \left( f^{(m)} \right)^* m \left(\frac{\delta \lambda_f}{v (n+1)}  + H \right) = (1- \delta)  K_{X^{(m)}/T} + \left( f^{(m)} \right)^* m \left(\frac{\delta \lambda_f}{v (n+1)}  + H \right)
\end{equation*}
is nef. Lastly, one divides by $\delta -1$, converges to $0$ with $H$, and lastly by a standard lemma (\autoref{lem:product_nef}) removes the $(\_)^{(m)}$.

\subsubsection{ Positivity.}
\label{sec:outline_positivity}
The rough idea here is to use a twisted version of the ampleness lemma, c.f.,  \cite[3.9 Ampleness Lemma]{Kollar_Projectivity_of_complete_moduli} and the slight modification in \cite[Thm 5.1]{Kovacs_Patakfalvi_Projectivity_of_the_moduli_space_of_stable_log_varieties_and_subadditvity_of_log_Kodaira_dimension}. We need a twisted version of the ampleness lemma as the techniques developed until this point in the article  do not work directly over higher dimensional bases. The main idea here is that to get bigness of $\lambda_f$ it is enough to show positivity of $\lambda_f$ over a very general element $C$ of each moving family of curves of $T$ in a bounded way. Below we explain how we do this. 

The main benefit of proving the result on the nefness threshold, \autoref{thm:nef_threshold},  
is the following:  one can prove, again using standard semi-positivity arguments, see \autoref{prop:semi_positivity_engine_downstairs}, that $\sQ:=f_* \sO_X(-r K_{X/T} + \alpha f^* \lambda_f)$ is nef, for some constants $r$ and $\alpha$. Furthermore, these constants $r$ and $\alpha$  can be chosen to be uniform, as $f$ runs through all families obtained by base-changing on a very general element $C$ of a moving family of curves on $T$. Then, the ampleness lemma (\autoref{thm:ampleness}) gives an ample line bundle $B$ on $T$ such that for all curves $C$ as above, $C \cdot B  \leq C \cdot \det \sQ$. Then one can use another  trick from (semi-)positivity theory, already contained in Viehweg's work, which shows that for $q :=  \rk \sQ$ there is an embedding
\begin{equation*}
\det \sQ \to \bigotimes_{q \textrm{ times}} f_* \sO_X(-r K_{X/T} + \alpha f^* \lambda_f) \cong f_*^{(q)} \sO_X\left(-r K_{X^{(q)}/T} + q \alpha \left( f^{(q)} \right)^* \lambda_f \right),
\end{equation*}
Using the adjunction of $f^{(q)}_*$ and $\left( f^{(q)} \right)^*$, we obtain the inequality of divisors
\begin{equation*}
\left(f^{(q)}\right)^* B \leq  \left(f^{(q)}\right)^* \det \sQ \leq - r K_{X^{(q)}/T} + q \alpha \left( f^{(q)} \right)^* \lambda_f,
\end{equation*}
which survives the restriction over $C$ by the genericity assumption in the choice of $C$. 
From here, a simple intersection computation shows that $C \cdot B$ bounds $\deg \lambda_f|_C$ from below up to some uniform constants, not depending on the choice of $C$, see the end of the proof of  point \autoref{itm:semi_positive_no_boundary:big} of \autoref{thm:semi_positive_no_boundary}. 

\subsection{Organization of the paper}

See \autoref{sec:outline} for a thorough explanation on which part of the argument can be found where. Here we only note that the actual argument, so what is explained in  \autoref{sec:outline}, starts in \autoref{sec:growth}, and lasts until \autoref{sec:examples}, where we construct some examples which show that the statements of the main results are sharp. After \autoref{sec:examples}, we only have \autoref{sec:appendix}, with some computations related to the definition of the CM line bundle. 

 Before the argument starts, in \autoref{sec:notation}, \autoref{sec:CM_definition} and \autoref{sec:delta} we present notation and background, as well as, simpler statements. The division of this part between the above 3 sections is based on topics. \autoref{sec:notation} contains general topics, \autoref{sec:CM_definition} contains the definition of the CM line bundle and the related statements, and \autoref{sec:delta} contains the definition and the basics about the $\delta$-invariant and $K$-stability. 

We also include a table on the location of the proofs of the theorems stated in the introduction.\\[5pt]
\def\arraystretch{1.3}
\begin{tabular}{|l|l|}
\hline
\textbf{Statements of the introduction} & \textbf{their proofs} \\
\hline
\autoref{thm:main} & \autoref{sec:main_theorem} \\
\hline
\autoref{thm:semi_positive_boundary} & \autoref{sec:semi_pos_proof} \\
\hline
\autoref{thm:semi_positive_no_boundary}
& \autoref{sec:arbitrary_base_pos} \\
\hline 
\autoref{cor:proper_base} \& \autoref{thm:lambda_non_big} & \autoref{sec:arbitrary_base_pos} \\
\hline 
 \autoref{cor:bounding_volume} & \autoref{sec:applications} \\
\hline 
\autoref{thm:nef_threshold} & \autoref{sec:nefness_threshold} \\
\hline
\end{tabular}

\subsection{Acknowledgements}

We would like to thank the referee for the thorough and careful reading of the article, as well as for the many useful suggestions. This project gained significant momentum during the INdAM workshop  "Moduli of K-stable varieties"; we thank INdAM for the organizational and financial support. We thank Harold Blum,  Ruadha\'i Dervan, Mattias Jonsson, Quentin Posva, Chenyang Xu and Maciej Zdanowicz for the many useful conversations and comments. The work of the second author was partially supported by the Swiss National Science Foundation grant \#200021/169639.

\section{Notation}
\label{sec:notation}

\subsection{Base-changes}
\label{sec:base_changes}

All base-changes are denoted by lower index. For example, if $f : X \to T$ is a family, $\sF$ is a coherent sheaf on $X$ and $S \to T$ is a base-change, then $\sF_S:= h^* \sF$, where $h : S \times_T X \to X$ is the projection morphism. 

\subsection{Fiber product notation}
\label{sec:product_notation}

The most important particular notation used in the article is that of fiber products. That is, for a family $f : X \to T$ of varieties we denote the $m$-times fiber product of $X$ with itself over $T$ by $X^{(m)}$. As in our situation the base is always clear, we omit it from the notation. Hence, $X^{(m)}$ denotes the fiber product over $T$ of $m$ copies of $X$, and for a point $t \in T$, $X_t^{(m)}$ denotes the fiber product over $t$ of  $m$ copies of $X_t$. In this situation, $p_i : X^{(m)} \to X$ denotes the projection onto the $i$-th factor, and we set for any divisor $D$ or line bundle $\sL$:
\begin{equation*}
D^{(m)}:= \sum_{i=1}^m p_i^* D \textrm{, and } \sL^{(m)}:= \bigotimes_{i=1}^m p_i^* \sL.
\end{equation*}

\subsection{General further notation}
\label{sec:general_notation}

A \emph{variety} is an integral, separated scheme of finite type over $k$. We call $(X, \Delta)$ a \emph{pair}, if $X$ is a normal variety, and $\Delta$ is an  effective $\bQ$-divisor, called the \emph{boundary}. 
A projective pair $(X, \Delta)$ over $k$  is a \emph{normal Fano} pair, if   $-(K_X + \Delta)$ is an ample $\bQ$-Cartier divisor. A \emph{normal Fano} pair $(X, \Delta)$ is
a \emph{Fano} pair if $(X, \Delta)$ has  klt singularities. To avoid confusion, many times we say \emph{klt Fano} instead of \emph{Fano}, nevertheless we mean the same by the two. If there is no boundary, we mean taking the boundary $\Delta=0$. 

A \emph{big open set} $U$ of a variety $ X$ is an open set for which $\codim_X (X \setminus U) \geq 2$. 

A \emph{vector bundle} is a locally free sheaf of finite rank. 

The \emph{$\bQ$-linear system}  of a $\bQ$-divisor $D$ on a normal variety is $|D|_{\bQ}:=\{ \ L \textrm{ is an effective $\bQ$-divisor}\ | \ \exists m \in \bZ, m>0 : m L \sim m D\ \}$.  

A \emph{geometric fiber} of a morphism $f : X \to T$ is a fiber over a geometric point, that is over a morphism $\Spec K \to T$, where $K$ is an algebraically closed field extension of the base field $k$. We say that a condition holds for a \emph{very general geometric point/fiber}, if there are countably many proper closed sets, outside of which it holds for all geometric points/fibers. \emph{General point/fiber} is defined the same way but excluding only finitely many proper closed subsets. The \emph{(geometric) generic point/fiber} on the other hand denotes the scheme theoretic (geometric) generic point/generic fiber.

\subsection{Relative canonical divisor}
\label{sec:relative_canonical}

For a flat family $f : X \to T$ the relative dualizing complex  is defined by $\omega_{X/T}^\bullet:=f^! \sO_T$, where $f^!$ is Grothendieck upper shriek functor as defined in \cite{Hartshorne_Residues_and_duality}. If $f$ is also a family of pure dimension $n$, then the relative canonical sheaf is the lowest non-zero cohomology sheaf $\omega_{X/T}:= h^{-n}(\omega_{X/T}^\bullet)$ of the relative dualizing complex. To obtain the absolute versions of these notions one uses the above definition for $T = \Spec k$. The important facts regarding the relative dualizing sheaf that we use in the present section are the following: 
\begin{enumerate}
\item The sheaf $\omega_{X/T}$ is reflexive if the fibers are normal \cite[Prop A.10]{Patakfalvi_Schwede_Zhang_F_singularities_in_families}.
\item If $T$ is Gorenstein and $X$ is normal, then $\omega_{X/T} \cong \omega_X \otimes f^* \omega_T^{-1}$ \cite[Lemma 2.4]{Patakfalvi_Semi_negativity_of_Hodge_bundles_associated_to_Du_Bois_families}, and then as $\omega_X$ is $S_2$ \cite[Cor 5.69]{Kollar_Mori_Birational_geometry_of_algebraic_varieties}, $\omega_{X/T}$ is also reflexive in this case \cite{Hartshorne_Stable_reflexive_sheaves}. 
\item By the previous two points, if $f$ is flat, $X$ is normal and either $T$ is smooth or the fibers are normal, then $\omega_{X/T}$ is reflexive, and hence it corresponds to a linear equivalence class of Weil divisors which we denote by $K_{X/T}$. 
\item \label{itm:Cohen_Macaulay_base_change} On the relative Cohen-Macaulay locus $U \subseteq X$ (that is, on the open set where the fibers are Cohen-Macaulay),  $\omega_{U/T} \cong \omega_{X/T}|_U$ is compatible with base-change \cite[Thm 3.6.1]{Conrad_Grothendieck_duality_and_base_change}. 
\end{enumerate}
In particular, by the above we always have the following assumptions on our families: $f : X \to T$ is flat with fibers being of pure dimension $n$, and either $T$ is smooth, or the fibers of $f$ are normal. In both cases  we discuss base-change properties of the relative canonical divisor below. 

\subsubsection{Base-change of the relative log-canonical divisor when the fibers are normal}
\label{sec:base_change_relative_canonical_normal_fibers}

Let us assume that $f : X \to T$ is a projective, flat morphism to a normal projective variety with normal, connected fibers. In particular then $X$ is also normal. Assume additionally that there is an effective $\bQ$-divisor  $\Delta$ given on $X$, such that $\Delta$ does not contain any fiber, and  $K_{X/T} + \Delta$ is a $\bQ$-Cartier divisor. Let $U \subseteq X$ be the smooth locus of $f$, which is an open set, and by the normality assumption on the fibers, $U \cap X_t$ is a big open set on each fiber $X_t$, see \autoref{sec:general_notation} for the definition of a big open set. 

Let $S \to T$ be a morphism from another normal projective variety. Then, we may define a pullback $\Delta_S$ as the unique extension of the pullback of $\Delta|_U$ to $U_S$; the key here is that $\Delta|_U$ is $\bQ$-Cartier.  Moreover, if $\sigma: X_S \to X$ is the induced morphism, then as $\bQ$-Cartier divisors
\begin{equation}
\label{eq:pullback_log_canonical_divisor_normal_fibers}
 K_{X_S/S} +  \Delta_S \sim_{\bQ} \sigma^* (K_{X/T} + \Delta).
\end{equation}
Indeed, it is enough to verify this isomorphism on $U$, as $U$ is big in $X$ and $U_S$ is big in $X_S$. However, over $U$ the linear equivalence \autoref{eq:pullback_log_canonical_divisor_normal_fibers} holds by the definition of $\Delta_S$ and by the base-change property of point \autoref{itm:Cohen_Macaulay_base_change} above. 
In particular, $f_S : X_S \to S$ and $\Delta_S$   satisfies all the assumptions we had for $f : X \to T$ and $\Delta$.

\subsubsection{Base-change of the relative log-canonical divisor when the base is smooth}
\label{sec:base_change_relative_canonical_smooth_base}

Let $f : X \to T$ be a flat morphism from a normal projective variety to a smooth, projective variety with connected fibers. Let $\Delta$ be an effective $\bQ$-divisor on $X$ such $K_{X/T} + \Delta$ is $\bQ$-Cartier. Let $T_{\norm} \subseteq T$ be the open set over which the fibers of $X$ are normal. 

Note that by the smoothness assumption on $T$, at a point $x \in X$, the fiber $X_{f(x)}$ is Gorenstein if and only if $X$ is relatively Gorenstein if and only if $X$ is Gorenstein. 
Let $U \subseteq X$ be the open set of relatively Gorenstein points over $T$. Let $\iota : C \to T$ be a finite morphism from a smooth, projective curve such that $\iota(C) \cap T_{\norm} \neq \emptyset$, and denote by $\sigma: X_C \to X$ the natural morphism. 

\emph{We claim that $\sigma^{-1} U$ is big in $X_C$.} This is equivalent  to showing that for each $c \in C$, $X_c$ is Gorenstein at some point, and that for general $c \in C$, there is a big open set of $X_c$ where $X_c$ is Gorenstein. The former is true for all schemes of finite type over $k$, hence also for $X_c$. The latter is true by the $\iota(C) \cap T_{\norm} \neq \emptyset$ assumption. This concludes our claim.

Now, let $\pi : Z \to X_C$ be the normalization of $X_C$, $\rho: Z \to X$ and $g : Z \to C$ the induced morphisms and set $W:= \rho^{-1} U$. The notations are summarized in the following diagram:
\begin{equation*}
\xymatrix{
&   W \ar@{^(->}[ld] \ar[r] & \sigma^{-1} U \ar@{^(->}[ld] \ar[r]  & U \ar@{^(->}[ld] & \\
Z \ar[dr]_g 
\ar@/^1.3pc/[rr]|!{[rru];[r]}\hole^\rho
\ar[r]_{\pi} & X_C \ar[r]_{\sigma} \ar[d]^{f_C} & X \ar[d]^f \\
& C \ar[r]_{\iota} & T & T_{\norm} \ar@{_(->}[l]
} 
\end{equation*}
Then, \cite[Lem 9.13]{Kovacs_Patakfalvi_Projectivity_of_the_moduli_space_of_stable_log_varieties_and_subadditvity_of_log_Kodaira_dimension} tells us that there is a natural injection $\omega_{W/C} \to \left( \pi|_W \right)^* \omega_{\sigma^{-1} U/C}$. To be precise, \cite[Lem 9.13]{Kovacs_Patakfalvi_Projectivity_of_the_moduli_space_of_stable_log_varieties_and_subadditvity_of_log_Kodaira_dimension} assumes $\sigma^{-1} U$ to be normal, but as the proof does not use it, this is an unnecessary assumption.  Combining this injection with the isomorphism
$\left( \sigma|_{\sigma^{-1}U} \right)^* \omega_{U/T} \cong \omega_{\sigma^{-1} U/C}$ given by point \autoref{itm:Cohen_Macaulay_base_change} above we obtain
\begin{equation}
\label{eq:relative_canonical_base_change}
\omega_{W/C}  \hookrightarrow \left( \pi|_W \right)^* \omega_{\sigma^{-1} U/C} \cong  \left( \pi|_W \right)^* \left( \sigma|_{\sigma^{-1}U} \right)^* \omega_{U/T} \cong \left( \rho|_W \right)^* \omega_{U/T} ,   
\end{equation}
which is an isomorphism over the locus $T_{\red}$ over which the fibers of $f$ are reduced. Indeed, over $T_{\red}$ the fibers of $X_C \to C$ are all reduced, and by the $\iota(C) \cap T_{\norm} \neq \emptyset$ assumption the general fiber of $X_C \to C$ is normal. In particular, over $T_{\red}$, $X_C$ is $R_1$ and $S_2$,  and hence normal. So, $\pi$ is the identity over $T_{\red}$.

Let $m>0$ be then an integer such that $m(K_{X/T}+ \Delta)$ is Cartier. That is, $\sL:=\sO_X(m(K_{X/T}+ \Delta))$ is a line bundle, and furthermore, $m\Delta$ yields an embedding $\omega_{U/T}^{\otimes m} \hookrightarrow \sL|_U$. Composing this with the $m$-th power of the homomorphism of \autoref{eq:relative_canonical_base_change} we obtain:
\begin{equation}
\label{eq:relative_canonical_base_change_final}
 \omega_{W/C}^{\otimes m} \to (\rho|_W)^* \sL \cong \sO_W(m \rho^*(K_{X/T} + \Delta)|_W),
\end{equation}
which map over $T_{\red}$ is given by ``multiplying with $\left(\rho|_{g^{-1}\iota^{-1} T_{\red}}\right)^* m\Delta$''. Indeed, for the last remark, the main thing to note is that the regular locus of $X$, over which   $m \Delta$ is necessarily Cartier, pulls back to a big open set of  $g^{-1}\iota^{-1}T_{\red}$, as general fiber of $f_C$ is normal and special fiber of $f_C$ over $T_{\red}$ are reduced. Hence $\pi$ is an isomorphism over $g^{-1}\iota^{-1}T_{\red}$ and also the pullback $\left(\rho|_{g^{-1}\iota^{-1} T_{\red}}\right)^* m\Delta$ is sensible the usual way: restricting to the regular locus, performing the pullback there, and then taking divisorial extension using bigness of the open set. 

Lastly, the map \autoref{eq:relative_canonical_base_change_final} is given by an effective divisor $D$. If we set $\Delta_Z:= \frac{D}{m}$, using that $W$ is big in $Z$, we obtain:

\begin{proposition}
\label{rem:reduced_fibers_no_boundary}
\label{prop:relative_canonical_base_change_normal}
Consider the following situation:
\begin{itemize}
\item let $f : X \to T$ be a flat morphism from a normal projective variety to a smooth, projective variety with connected fibers,
\item  let $\Delta$ be an effective $\bQ$-divisor on $X$ such $K_{X/T} + \Delta$ is $\bQ$-Cartier,
\item let $T_{\norm} \subseteq T$ and $T_{\red} \subseteq T$  be the open set over which the fibers of $X$ are normal or  reduced, respectively,
\item let $\iota : C \to T$ be a finite morphism from a smooth, projective curve such that $\iota(C) \cap T_{\norm} \neq \emptyset$, and
\item  let $\pi : Z \to X_C$ be the normalization, and $\rho: Z \to X$ and $g : Z \to C$ be the induced morphisms.
\end{itemize}
Then, there is an effective $\bQ$-divisor $\Delta_Z$ on $Z$ such that:
\begin{enumerate}
\item \label{itm:relative_canonical_base_change:base_change} $K_{Z/C} + \Delta_Z \sim_{\bQ} \rho^*( K_{X/T} + \Delta)$, 
\item \label{itm:relative_canonical_base_change:red_fibers} $X_C$ is normal over $T_{\red}$ and $\Delta_Z|_{g^{-1}\iota^{-1}T_{\red}} = \left(\rho|_{g^{-1}\iota^{-1} T_{\red}}\right)^* \Delta$, and 
\item $\Delta_Z|_{g^{-1}\iota^{-1}T_{\norm}}$ agrees with the pullback of $\Delta|_{f^{-1}T_{\norm}}$ in the sense of \autoref{sec:base_change_relative_canonical_normal_fibers}.
\end{enumerate}
\end{proposition}

\section{The definition of the CM line bundle}
\label{sec:CM_definition}

Here we present the definition of the CM line bundle in two cases:
\begin{enumerate}
 \item in the non logarithmic case for arbitrary polarizations, and
 \item in the logarithmic case for the anti-log-canonical polarization.
\end{enumerate}
In the first case, we also connect it to the other existing definitions in the literature. In the second case, we are not able to present such connections, because the lack of literature  would force us to work out many details about the Paul-Tian type definition \cite{Paul_tian1,Paul_tian2}, and then prove the equivalence with that: this would be beyond the scope of the present article. 

In any case, it is important to stress that the definitions are different in the two cases: \emph{One does not obtain the logarithmic version by simply plugging in the logarithmic relative anti-canonical divisor into the polarization of the non-logarithmic case}. The reason for the difference is that that in the logarithmic case the CM line bundle has to take into account also the variation of the boundary, see the paragraph before \autoref{thm:semi_positive_no_boundary}.

\begin{definition}
\label{def:CM}
{\scshape CM line bundle in the non-logarithmic setting.}
Let $f : X \to T$ be a flat morphism of normal projective varieties of relative dimension $n$, and $L$ an $f$-ample $\bQ$-Cartier divisor on $X$. For every integer $q$ divisible enough, the Hilbert polynomial of a (equivalently any) fiber $X_t$ is 
\begin{equation}
\label{eq:a_0_a_1}
\chi(X_t, qL_t) = a_0 q^{n} + a_1 q^{n-1} + O(q^{n-2}).
\end{equation}
Set $\mu_L:=\frac{2a_1}{a_0}$.
We define the Chow-Mumford line bundle as the pushforward cycle
$$
\lambda_{f,L}:=f_*\left(\mu_L L^{n+1}+(n+1)L^n \cdot K_{X/T}\right),
$$
which is an abuse of language as it is not a line bundle but rather a $\bQ$-Cartier divisor class, according to \autoref{prop:2_defs_CM_same}. We would also like to stress that $\lambda_{f,L}$ is a divisor class (in the Weil group, or equivalently the first Chow group), as opposed to a fixed divisor. 

If $L$ is not indicated, then we take $L= -K_{X/T}$, which we assume to be an $f$-ample $\bQ$-Cartier divisor, and we use the notation $\lambda_f:=\lambda_{f,L}$.

\end{definition}

\begin{remark}
\label{rem:anti_canonically_polarized_CM}
\label{rem:anti_canonically_polarized_CM_2}
Note that in the $L= - K_{X/T}$ case:
\begin{equation*}
\lambda_f=f_*\left(\mu_L (-K_{X/T})^{n+1}+(n+1)(-K_{X/T})^n \cdot K_{X/T}\right) = f_* (\mu_L - (n+1)) (-K_{X/T})^{n+1}
\end{equation*}
As $X$ in \autoref{def:CM} is assumed to be normal, so is $X_t$ for $t$ a general closed point.  In particular, \autoref{lem:Hilbert_poly_coeffs} implies that
\begin{equation*}
\mu_L=\frac{2a_1}{a_0}= \frac{2 \left(- \frac{K_X \cdot L_t^{n-1}}{2 (n-1)!}\right)}{\frac{L_t^n}{n!}} = 
n \frac{-K_{X_t} \cdot L_t^{n-1}}{L_t^n}. 
\end{equation*}
In particular if $L= -K_{X/T}$ we obtain that $\mu_L = n$.
Hence, we obtain the definition we used in \autoref{eq:CM_line_bundle_log}:
\begin{equation*}
\lambda_f= f_* (\mu_L - (n+1)) (-K_{X/T})^{n+1} = - f_*  (-K_{X/T})^{n+1}.
\end{equation*}

\end{remark}

We only define the logarithmic version of the CM line bundle in the anti-log-canonically polarized case. If $\Delta=0$, this definition agrees with  the case of $L=-K_{X/T}$ of the non-logarithmic definition, according to the  final formula of \autoref{rem:anti_canonically_polarized_CM}.

\begin{definition}
\label{def:CM_log} 
{\scshape CM line bundle in the logarithmic setting.}
If $f : (X, \Delta) \to T$ is a flat morphism of relative dimension $n$ from a projective normal pair to a normal projective variety such that $-(K_{X/T} + \Delta)$ is $\bQ$-Cartier and $f$-ample. Then we define the CM line bundle by 
\begin{equation*}
\lambda_{f, \Delta}:= -f_* ((-(K_{X/T} + \Delta))^{n+1} ).
\end{equation*}

\end{definition}

\begin{notation}
\label{notation:Knudsen_Mumford}
 In the set-up of \autoref{def:CM} (resp. of \autoref{def:CM_log}, in which case we set also $L:=-(K_{X/T} + \Delta)$), fix an integer $s$ such that $sL$ is an $f$-very ample Cartier divisor. Following \cite[Appendix to Chapter 5, Section D]{MD_FJ_KF_GIT} and \cite[Theorem 4]{Knudsen_Mumford_The_projectivity_of_the_moduli_space_of_stable_curves_I}, consider  the Mumford-Knudsen expansion of $\sO_X(sL)$:
\begin{equation}\label{eq:Mum-Ku}
\det f_* \sO_X(qsL) \cong \bigotimes_{i=0}^{n+1} \sM_i^{{q \choose i}},
\end{equation}
where  $\sM_i$ are uniquely determined line bundles on $T$. 
\end{notation}

For future reference, we note that as the left side of \autoref{eq:Mum-Ku} is invariant under base-change for $q \gg 0$, the above unicity of $\sM_i$ implies that:

\begin{lemma}
\label{lem:Knudsen_mumford_base_change}
In the situation of \autoref{notation:Knudsen_Mumford}, the formation of $\sM_i$ is compatible with base-change. That is, if $ S \to T$ is a base-change, and $\sM^S_i$ are the coefficients of the Knudsen-Mumford expansion of $sL_S$, then $\sM^S_i \cong \left( \sM_i\right)_S$. 
\end{lemma}

\begin{notation}
\label{notation:Paul_Tian}
In the case of \autoref{def:CM}, according to \cite[Definition 1]{Paul_tian2} (see also \cite[Section 2.4, page 11]{Paul_tian1} and \cite[Theorem 4]{Knudsen_Mumford_The_projectivity_of_the_moduli_space_of_stable_curves_I} for the role of $\sM_{n+1}$), the CM line bundle is defined as 
\begin{equation*}
L_{CM,f,sL}:=\sM_{n+1}^{n(n+1)+ \mu_{sL}} \otimes \sM_n^{-2 (n+1)},
\end{equation*}
where $\mu_{sL}$ is the number defined in \autoref{def:CM}.
For simplicity we regard $L_{CM,f,sL}$ as a Cartier divisor. As we explained earlier in the case of \autoref{def:CM_log} a definition as above is not worked out in the literature to such an extent, and hence we do not consider it here. 
\end{notation}

The proof of the following proposition will be given in \autoref{sec:appendix}. 

\begin{proposition} 
\label{prop:2_defs_CM_same}
\begin{enumerate}
\item \label{itm:2_defs_CM_same:Tian_Paul} {\scshape Connection with the Paul-Tian definition.} 
In the situation of \autoref{notation:Paul_Tian}, if $T$ is smooth or the fibers of $f$ are normal, then 
$s^n \lambda_{f,L} = c_1(L_{CM,f,sL})$. 
In particular, $\lambda_{f,L}$ is $\bQ$-Cartier.
\item \label{itm:2_defs_CM_same:leading_term} {\scshape Connection with the leading term of the Knudsen-Mumford expansion.} 
In the situation of \autoref{notation:Knudsen_Mumford},  consider the case of \autoref{def:CM_log}, which includes the case of  \autoref{def:CM} with $L=-sK_{X/T}$ as well. Additionally, assume that 
either $T$ is smooth or the fibers of $f$ are normal, and $\Delta$ does not contain any fiber. Then,  
$-s^{n+1} \lambda_{f, \Delta}= c_1(\sM_{n+1})$.
In particular, $\lambda_{f,\Delta}$ is $\bQ$-Cartier. 
\end{enumerate}
\end{proposition}

\begin{proposition} {\scshape Base-change for the CM-line bundle.}
\label{prop:CM_base_change}
Let $f : X \to T$ be a flat morphism between projective normal varieties, let $\Delta$ be an effective $\bQ$-divisor such that $-(K_{X/T} + \Delta)$ is  an $f$-ample $\bQ$-Cartier divisor, and let $\tau : S \to T$ be a morphism from a normal projective variety. Assume either:
\begin{enumerate}
 \item \label{itm:CM_base_change:normal_fiber} the fibers of $f$ are normal and $\Delta$ does not contain any fiber, in which case set $g:=f_S, Z:= X_S$, and let $\Delta_Z$ be the pullback of $\Delta$ as explained in \autoref{sec:base_change_relative_canonical_normal_fibers}. 
 \item \label{itm:CM_base_change:smooth_base} $T$ is smooth and $\tau$ is a finite morphism from a curve, such that some of the fibers of $f$ over $\tau(S)$ are normal and not contained in $\Delta$. In this case, set  $Z$ to be the normalization of $X_S$, $\rho:Z \to X$ and $g :Z \to S$ the induced morphisms and $\Delta_Z$ the effective $\bQ$-divisor on $Z$ given by \autoref{prop:relative_canonical_base_change_normal}.\end{enumerate}
Then, the CM line bundle satisfies the  base-changes 
 $\tau^*\lambda_{f, \Delta} = \lambda_{g,\Delta_Z}$. 

\end{proposition}

\begin{proof}
Set $V:=X_S$, $L:=-(K_{X/T} + \Delta)$ and let $h: V \to S$ and $\sigma: V \to X$ be the induced morphisms. Fix an integer  $s>0$ be such that $sL$ and $s\rho^* L$ are relatively very ample over $T$ and $S$, respectively.  Note that according, to point \autoref{itm:relative_canonical_base_change:base_change} of \autoref{prop:relative_canonical_base_change_normal}, $s\rho^* L \cong -s(K_{Z/T} +\Delta_Z )$. Furthermore, set $\sM_{n+1}^f$, $\sM_{n+1}^g$ and $\sM_{n+1}^h$ be the leading terms of the Knudsen-Mumford expansions of $sL$, $s \rho^* L$ and $s \sigma^* L$, respectively. Then, 
\begin{equation*}
\tau^* \lambda_{f,\Delta} 
= 
\underbrace{\frac{\tau^* c_1\left(\sM_{n+1}^f\right)}{-s^{n+1}} }_{\textrm{point \autoref{itm:2_defs_CM_same:leading_term} of \autoref{prop:2_defs_CM_same}}}
=
\underbrace{\frac{c_1\left(\sM_{n+1}^h\right)}{-s^{n+1}} }_{\textrm{\autoref{lem:Knudsen_mumford_base_change}}} 
=
\underbrace{\frac{c_1\left(\sM_{n+1}^g\right)}{-s^{n+1}}}_{\parbox{73pt}{\tiny vacuous statement in the case of point \autoref{itm:CM_base_change:normal_fiber}, and  \autoref{lem:highest_Knudsen_Mumford_coefficient_integral} in the case of point \autoref{itm:CM_base_change:smooth_base}}}
=
\underbrace{  \lambda_{g,\Delta_Z} }_{\textrm{point \autoref{itm:2_defs_CM_same:leading_term} of \autoref{prop:2_defs_CM_same}}}.
\end{equation*}

\end{proof}

\section{The delta invariant and $K$-stability}
\label{sec:delta}

Here we give the definitions and the properties used in the present article of $\delta$-invariants, as well as we present the definition of $K$-semi-stability and uniform $K$-stability in \autoref{def:K_stability}. In the rest of the article we will use the characterizations of $K$-semi-stability and $K$-stability via $\delta$-invariants given in \autoref{def:K_stable}. We also prove in the present section that the $\delta$-invariant is constant at the very general fibers of a log-Fano family, see \autoref{prop:delta_general_fiber}.

\subsection{Definitions}

Basis-type divisors and the delta invariant have been introduced by K. Fujita and Y. Odaka in \cite{Fujita_Odaka_On_the_K-stability_of_Fano_varieties_and_anticanonical_divisors}, see also \cite{BJ}; in this section we recall their definitions. 

\begin{definition}\label{def:basis}
Assume we are in the following situation:
\begin{itemize}
 \item $Z$ is  a variety over $k$,
 \item  $L$ is a $\bQ$-Cartier divisor on $Z$, and
 \item $q>0$ is an integer for which $qL$ is Cartier.
\end{itemize}
A divisor $D \in |L|_{\bQ}$ is of \emph{$q$-basis type} if there are $D_i \in |qL| \quad (1 \leq i \leq h^0(X,qL))$, for which the corresponding $s_i \in H^0(Z,qL)$ form a $k$-basis of $H^0(Z, qL)$, and $D$ can be expressed as
$$
D=
\frac{1}{qh^0(Z,qL)}\sum_{i=1}^{h^0(Z,qL)} D_i. 
$$
$D$ is of \emph{basis type} if it is of $q$-basis type for some integer $q>0$. 
\end{definition}

Let $\Delta$ be a fixed effective $\bQ$-divisor on $Z$ such that $(Z,\Delta)$ is a klt pair. Given a $\bQ$-Cartier effective divisor $D$ on $Z$, we define its log canonical threeshold as
\begin{equation*}
\lct(Z,\Delta; D):=\sup \{t | (Z,\Delta+tD) \textrm{ is klt } \}.
\end{equation*}
Remark that since $(Z,\Delta)$ is klt, the above threshold is a positive number. Let us recall the definition of the $\alpha$ invariant.

\begin{definition}\label{def:alpha}
Let  $(Z,\Delta)$ be a klt pair and let $L$ be an effective $\bQ$-Cartier divisor on $Z$. The alpha invariant of $(Z,\Delta;L)$ is
\begin{equation*}
\alpha(Z,\Delta; L):=\inf_{D\in |L|_{\bQ}}\lct(Z,\Delta;D).
\end{equation*}
We write $\alpha(Z,\Delta)$ for $\alpha(Z,\Delta;-K_Z-\Delta)$.																																																															
\end{definition}

The $\alpha$ invariant has been introduced by Tian in relation with the existence problem for K\"{a}hler-Einstein metrics. The delta invariant is a variation on the alpha invariant. The main difference is that in the case of $\alpha$ invariant one considers the log canonical threshold of all divisors in the $\bQ$-linear system, while in the $\delta$ invariant is defined using only basis type divisors. In particular, while $\alpha (X) \geq \frac{\dim X}{\dim X + 1}$ only implies $K$-semi-stability \cite{Tian_alpha,Odaka_Sano}, $\delta(X) \geq 1$ happens to be equivalent to it \cite[Theorem B]{BJ}, see also \autoref{def:K_stable}. The delta invariant was introduced in \cite[Definition 0.2]{Fujita_Odaka_On_the_K-stability_of_Fano_varieties_and_anticanonical_divisors}. In \cite{BJ}, although it was also denoted by $\delta$, it is called the \emph{stability threshold}. 

\begin{definition}\label{def:delta}
Let  $(Z,\Delta)$ be a klt pair and let $L$ be a $\bQ$-Cartier divisor on $Z$. 
\begin{enumerate}
 \item \label{itm:delta:delta_q} For every positive integer $q$ for which  $qL$ is Cartier and $h^0(Z,qL)>0$, the \emph{$q$-th delta invariant} of $L$ with respect to the pair  $(Z,\Delta)$ is
$$
\delta_q(Z,\Delta;L):=\inf_{D \in |L|_{\bQ}  \textrm{ is of $q$-basis type}} \lct(Z,\Delta;D).
$$
\item \label{itm:delta:delta} Assume that $L$ is big, and fix an integer $s>0$ such that $sL$ is Cartier and $h^0(Z, sL)>0$, which conditions then also hold for  every positive multiple of $s$. 
The \emph{delta invariant} of   $L$ with respect to   $(Z,\Delta)$ is
$$
\delta(Z,\Delta;L):=\limsup_{q\to \infty}\delta_{sq}(Z,\Delta;L).
$$
\item 
If $(Z,\Delta)$ is a klt Fano pair, we let $\delta_q(Z,\Delta):=\delta_q(Z,\Delta;-K_Z-\Delta)$ and 
$\delta(Z,\Delta):=\delta(Z,\Delta;-K_Z-\Delta)$.
\end{enumerate}
\end{definition}

\begin{remark}
We note the following subtleties of \autoref{def:delta}:
\begin{itemize}
\item According to \cite[Lem 8.8]{Kovacs_Patakfalvi_Projectivity_of_the_moduli_space_of_stable_log_varieties_and_subadditvity_of_log_Kodaira_dimension}, the infimum of point \autoref{itm:delta:delta_q} is in fact a minimum.
\item According to \autoref{lem:scaling_delta_invariant}, the  definition of point \autoref{itm:delta:delta} does not depend on the choice of $s$, and the limsup in point \autoref{itm:delta:delta} is in fact a limit. 

\end{itemize}
\end{remark}

\subsection{Relation to K-stability}

In this section we follow closely \cite{BJ}, as we want to adapt some of their result from Fano varieties over $\bC$ to Fano pairs over $k$. Similar adaptation was done also in \cite{Blum_Thesis}. Consider the situation:

\begin{notation}
$(Z,\Delta)$ is a klt pair, $L$ is a $\bQ$-Cartier divisor on $Z$, and $s>0$ is an integer such that $sL$ is Cartier and $h^0(Z,sL)\neq 0$. 
\end{notation}
 Let $v$ be a non-trivial divisorial valuation on $Z$ associated to a prime divisor $E$ over $Z$, we consider the filtration
\begin{equation*}
F_iH^0(Z,qsL):=\{  t \in H^0(Z,qsL) |\textrm{ such that } v(t)\geq i\}=\underbrace{H^0(V,qs\pi^*L-iE)}_{\parbox{80pt}{\tiny $\pi : V \to Z$ is a normal model where $E$ lives}},
\end{equation*}
and the invariant
\begin{multline*}
S_{q}(v):=\frac{1}{qsh^0(Z,qsL)}\sum_i i \dim_k\left(F_iH^0(Z,qsL)/F_{i+1}H^0(Z,qsL)\right)
\\ =\frac{1}{qsh^0(Z,qsL)}\sum_{i \geq 1}  \dim_k F_iH^0(Z,qsL).
\end{multline*}
Denote by $B_q$ the set of $qs$-basis type divisors with respect to $qsL$. As observed for instance in \cite[proof of Lemma 2.2]{Fujita_Odaka_On_the_K-stability_of_Fano_varieties_and_anticanonical_divisors},
\begin{equation}
\label{com_basis}
S_q(v)=\max_{D\in B_q}v(D),
\end{equation}
and the maximum is attained exactly for bases adapted to the filtration $F_i$. When $L$ is big, the asymptotic of $S_q$ is well-understood, see for instance \cite[proof of Theorem 1.3]{Fujita_Odaka_On_the_K-stability_of_Fano_varieties_and_anticanonical_divisors}, \cite[Corollary 2.12]{BJ} and \cite[Corollary 3.2]{Boucksom_Uniform}:
\begin{equation}
\label{asymp}
S(v):=\lim_{q\to \infty}S_q(v)=\frac{1}{\Vol(L)}\int_0^{+\infty}\Vol(\pi^* L-xE)dx
\end{equation}
The next statement is a logarithmic version of \cite[Theorem 4.4]{BJ}, following very closely the arguments given there. 

\begin{theorem}
\label{thm:BJ}
\begin{enumerate}
 \item \label{itm:BJ_limit}

If $L$ is a big $\bQ$-Cartier divisor, such that $sL$ is a Cartier divisor and $h^0(Z, sL) \neq 0$, then 
the sequence $\delta_{qs}(Z,\Delta;L)$ converges to $\delta(Z,\Delta;L)$, i.e. the delta invariant is a limit and not only a limsup; moreover
\begin{equation*}
\delta(Z,\Delta;L)=\inf_v\frac{A(v)}{S(v)},
\end{equation*}
where $A(v)$ is the log-discrepancy of $v$ with respect to the klt pair $(Z,\Delta)$, and the inf is taken over all non-trivial divisorial valuations. In particular, $\delta(Z, \Delta;L)$  is independent of the choice of $s$.

\item \label{itm:BJ_alpha}
Assuming furthermore that $L$ is ample, the following bounds hold
\begin{equation*}
\frac{\dim Z+1}{\dim Z }\alpha(Z,\Delta;L)\leq \delta(Z,\Delta;L)\leq \left(\dim Z+1\right)\alpha(Z,\Delta;L).
\end{equation*}
\end{enumerate}
\end{theorem}

\begin{proof}
{\scshape Point \autoref{itm:BJ_limit}.}
Set $\delta_q:=\delta_{qs}(Z,\Delta;L)$ and $\delta:=\delta(Z,\Delta;L)$. We first prove the inequality
\begin{equation}
\label{sup}
\limsup_{q\to \infty} \delta_q\leq \inf_v\frac{A(v)}{S(v)}
\end{equation}
Thanks to Equations \autoref{com_basis} and \autoref{asymp}, we can write 
\begin{equation*}
\inf_v\frac{A(v)}{S(v)}=
\underbrace{\inf_v\lim_{q\to \infty}\inf_{D\in B_q}\frac{A(v)}{v(D)}\geq \limsup_{q\to \infty}\left(\inf_{D\in B_q}\inf_v\frac{A(v)}{v(D)}\right)}_{\parbox{200pt}{\tiny $\forall v' : \displaystyle\inf_{D\in B_q} \frac{A(v')}{v'(D)} \geq \displaystyle\inf_{D\in B_q}\displaystyle\inf_v\frac{A(v)}{v(D)}$ \\ $\Rightarrow \forall v' \displaystyle\lim_{q\to \infty}\displaystyle\inf_{D\in B_q}\frac{A(v')}{v'(D)}\geq \displaystyle\limsup_{q\to \infty}\left(\displaystyle\inf_{D\in B_q}\displaystyle\inf_v\frac{A(v)}{v(D)}\right)$, and then take $\displaystyle\inf_{v'}( \_ )$ on the left side}} 
=\underbrace{\limsup_{q\to \infty} \delta_q.}_{\inf_v\frac{A(v)}{v(D)}=\lct(Z,\Delta; D)}
\end{equation*}
We now prove the inequality
\begin{equation}
\label{inf}
\liminf_{q\to \infty} \delta_q\geq \inf_v\frac{A(v)}{S(v)}
\end{equation}
This inequality follows from the key uniform convergence result \cite[Corollary 3.6]{BJ}: for every $\varepsilon>0$ there exists a $q_0=q_0(\varepsilon)$ such that for all $q>q_0$ and all divisorial valuations $v$ we have
\begin{equation*}
(1+\varepsilon)S(v)\geq S_q(v)
\end{equation*}
\cite[Corollary 3.6]{BJ} is stated over the complex numbers, however its proof works verbatim over $k$, let us explain why. The core part of the argument is \cite[Lemma 2.2]{BJ}, which is about convergence of integrals of concave functions over convex bodies in an Euclidean spaces, and this has nothing to do with the base field of $Z$. Another key ingredient is \cite[Lemma 2.6]{BJ}, which relies just on the concavity of the volume function. The rest of the proof uses filtrations of the coordinate ring and the Okunkov body of $Z$ to reduce the claimed approximation result to \cite[Lemma 2.2]{BJ}.

Let us now finish the proof. For $q$ big enough we have
\begin{equation*}
\frac{1}{1+ \varepsilon}\inf_v\frac{A(v)}{S(v)}\leq \inf_v\frac{A(v)}{S_q(v)}=\inf_v\inf_{D\in B_q}\frac{A(v)}{v(D)}\underbrace{=\delta_q}_{\inf_v\frac{A(v)}{v(D)}=\lct(Z,\Delta; D)}
\end{equation*}
taking the liminf on $q$ on the right hand side, and then letting $\varepsilon$ go to zero, we get the requested inequality. We obtain point \autoref{itm:BJ_limit} combining Equations \ref{sup} and \ref{inf}.

{\scshape Point \autoref{itm:BJ_alpha}.}
 Given a divisorial valuation $v$, we define its $q$-th pseudo-effective threshold as
\begin{equation*}
T_q(v):=\max \left\{ \left. \frac{v(D)}{qs} \; \right| \;  D\in |qsL|\right\}
\end{equation*}
and we have
\begin{equation*}
\alpha(Z,\Delta;L)=
\inf_q \inf_v \frac{A(v)}{T_q(v)}.
\end{equation*}

When $L$ is ample, \cite[Prop. 3.11]{BJ} gives the following bounds
\begin{equation*}
\frac{\dim(Z)}{\dim(Z)+1}\inf_q T_q(v)\geq S(v)\geq \left(\frac{1}{\dim(Z)+1}\right)\inf_{q}T_q(v),
\end{equation*}
which imply point \autoref{itm:BJ_alpha} (again, the proof in \cite{BJ} is over the complex numbers, but it works also over $k$). 
\end{proof}

\begin{corollary}[{\scshape Invariance of the delta invariant by scaling}]
\label{lem:scaling_delta_invariant} In the situation of \autoref{def:delta}.\autoref{itm:delta:delta}, for every positive integer $r>0$,  $ \delta(Z,\Delta;L)=r \delta(Z,\Delta;rL)$. Equivalently,
\begin{equation}
\label{eq:delta_invariant_scaling:statement}
\limsup_{q\to \infty}\delta_{rsq}(Z,\Delta;L)= \limsup_{q\to \infty}\delta_{sq}(Z,\Delta;L).
\end{equation}\end{corollary}
\begin{proof}
By \autoref{thm:BJ}, the limsup appearing in Equation \autoref{eq:delta_invariant_scaling:statement} is a limit, so the claim.
\end{proof}

We give the following definition of K-stability, which is equivalent to the more classical one by \cite[Theorem 6.1 (ii)]{Odaka_Sun} and \cite[Theorem 1.5]{Fujita_Uniform_K-stability_and_plt_blowups_of_log_Fano_pairs}.

\begin{definition}
\label{def:K_stability}
A normal Fano pair $(Z,\Delta)$ is
\begin{enumerate}
\item \emph{$K$-semi-stable} if  it is klt and for every divisorial valuation $v$, one has $A(v)\geq S(v)$; 
\item \emph{uniformly $K$-stable} if  it is klt  and there exists a positive constant $\varepsilon$ such that for every divisorial valuation $v$, one has $A(v)\geq(1 +\varepsilon) S(v)$. 
\end{enumerate}
Here $A(v)$ denotes the log-discrepancy of $\nu$ with respect to the pair $(Z, \Delta)$.

\end{definition}

The following corollary is now an immediate consequence of the above definition and \autoref{thm:BJ}

\begin{corollary}[{\scshape Characterization of K-stability}]
\label{def:K_stable}
Let $(Z,\Delta)$ be a normal Fano pair. 
Then, $(Z,\Delta)$ is
\begin{enumerate}
 \item $K$-semi-stable if and only if $(Z,\Delta)$ is klt and $\delta(Z,\Delta) \geq 1$,
 \item uniformly $K$-stable if and only if $(Z,\Delta)$ is klt and $\delta(Z,\Delta) > 1$.
\end{enumerate}

Moreover, if $(Z,\Delta)$ is klt and $\alpha(Z,\Delta)\geq\frac{\dim(Z)}{\dim(Z)+1} $ (resp. $>\frac{\dim(Z)}{\dim(Z)+1}$), then $(Z,\Delta)$ is K-semi-stable (resp. uniformly K-stable); if $(Z,\Delta)$ is klt and $\alpha(Z,\Delta)\leq \frac{1}{\dim(Z)+1}$ (resp. $< \frac{1}{\dim(Z)+1}$), then $(Z,\Delta)$ is not uniformly K-stable (resp. not K-semi-stable).

\end{corollary}

\subsection{Products}

The following conjecture is motivated by the equivalence between K-stability and K\"{a}hler-Einstein metrics in the Fano setting, it has been already proposed in \cite[Conjecture 1.11]{Park_Won_K-stability_of_smooth_del_Pezzo_surfaces}.

\begin{conjecture}\label{conj:delta}
Given two klt Fano pairs $(W,\Delta_W)$ and $(Z,\Delta_Z)$, one has
$$
\delta(W\times Z,\Delta_W\boxtimes \Delta_Z)=\min\{\delta(W,\Delta_W) \, , \, \delta(Z,\Delta_Z)\}
$$
\end{conjecture}

The analogue result for the alpha invariant and any polarization appeared for example in \cite[Proposition 8.11]{Kovacs_Patakfalvi_Projectivity_of_the_moduli_space_of_stable_log_varieties_and_subadditvity_of_log_Kodaira_dimension}, but used to be present much earlier in a smaller generality, i..e, in the smooth non-log case, for example in Viehweg's works. See also \cite[Thm. 1.10]{Park_Won_K-stability_of_smooth_del_Pezzo_surfaces} and \cite[Lemma 2.29]{Chelstov_log} for the Fano case. We can prove a weaker result for the delta invariant in \autoref{prop:prod_basis}.

\begin{definition}[Product basis type divisor]\label{def:prod_basis}
Let $(W,\Delta_W)$ and $(Z,\Delta_Z)$ be two klt pairs, let $L_W$ and $L_Z$ $\bQ$-Cartier divisors on $W$ and $Z$, respectively, and let $q>0$ be an integer such that both $qL_W$ and $q L_Z$ are Cartier and both $h^0(W, qL_W)$ and $h^0(Z, qL_Z)$ are non-zero. A divisor $D$ on $W\times Z$ is of $q$-\emph{product basis type} if there exist $q$-basis type divisors $D_W$ on $W$ and $D_Z$ on $Z$ such that
$$
D=p_W^*D_W+p_Z^*D_Z
$$
where $p_W$ and $p_Z$ are the projections.
 \end{definition}
\begin{remark}
\label{rem:prod_basis}
In \autoref{def:prod_basis}, if $D_W$ is associated to a basis $s_i$ and $D_Z$ to a basis $t_i$, then $D$ is associated to the basis $s_i\boxtimes t_j$.
\end{remark}

\begin{lemma}
\label{lem:product}
Let $(W,\Delta_W)$ and $(Z,\Delta_Z)$ be two klt (resp. lc) pairs, then also $(W\times Z,\Delta_W \boxtimes \Delta_Z)$ is klt (resp. lc).
\end{lemma}
\begin{proof}
As we work in characteristic zero, we may take  the product of a log resolution of $(W,\Delta_W )$ and of $(Z, \Delta_Z )$. This will be a log-resolution for $(W\times Z,\Delta_W\boxtimes \Delta_Z)$, with the union of the discrepancies of the original two log-resolutions, so the claim.
\end{proof}

\begin{proposition}\label{prop:prod_basis}
With the notations of \autoref{def:prod_basis}, let $D$ be a $q$-product basis type divisor. Then, 
$$
\lct(W\times Z, \Delta_W  \boxtimes \Delta_Z,D)\geq\min\{ \delta_q(W,\Delta_W;L_W)\; , \; \delta_q(Z,\Delta_Z; L_Z)\}
$$
\end{proposition}
\begin{proof}
 Take $t<\min\{ \delta_q(W,\Delta_W; L_W),\delta_q(Z,\Delta_Z;L_Z)\}$. We have to show that $(W\times Z,\Delta_W\boxtimes \Delta_Z+tD)$ is log canonical. Recall that 
\begin{equation*} 
 (W\times Z,\Delta_W\boxtimes \Delta_Z+tD)=(W\times Z,\left(\Delta_W+tD_W\right)\boxtimes\left( \Delta_Z+tD_Z\right))
 \end{equation*}
 and both $(W,\Delta_W+tD_W)$ and $(Z,\Delta_Z+tD_Z)$ are log canonical because of the hypothesis on $t$, so the claim follows from \autoref{lem:product}
 
\end{proof}

The full \autoref{conj:delta} has been proved in the preprint \cite{Product}, published after the first version of this paper has appeared.

\subsection{Behavior in families}

Here we prove that the $\delta$-invariant is constant on very general geometric points.
Recall that a \emph{geometric point} of $T$ is a map from the spectrum of an algebraically closed field to $T$. Key examples are the closed points and the geometric  generic point (i.e. the algebraic closure of the function fields) of $T$.

\begin{proposition}
\label{prop:delta_general_fiber}
Let $f : (X,\Delta) \to T$ be a flat, projective family of normal pairs over a normal variety, that is, we assume that $K_{X/T} + \Delta$ is $\bQ$-Cartier, and $\Supp \Delta$ does not contain any fiber. Additionally, let $L$ be an $f$-ample $\bQ$-Cartier divisor on $X$. Then there is a very general value of $\delta\left(X_{\ot}, \Delta_{\ot} ; L_{\ot}\right)$. More precisely, there is a real number $d\geq 0$ and there are countably many Zariski closed subsets $T_i \subseteq T$ such that for any geometric point $\ot \in T \setminus \left( \bigcup_i T_i \right)$, $\delta\left(X_{\ot}, \Delta_{\ot}; L_{\ot} \right)=d$.
\end{proposition}

\begin{proof}

We may fix an integer $s>0$ such that $sL$ is Cartier and $f_* \sO_X(qsL)$ is non-empty and commutes with base-change for any integer $q>0$. In particular, then for all $t \in T$, $sL_t$ is Cartier and $h^0(X_t, qsL_t)$ is positive and independent of $t$ for any integer $q>0$. 

We claim that \emph{for each integer $q>0$ there is a real number $d >0$ and a non-empty Zariski open set $U_q \subseteq T$ such that for each geometric point $\ot \in U$, $\delta_{qs}\left(X_{\ot},\Delta_{\ot}; L_{\ot} \right)=d$.} 
Assuming this claim, by setting $T_q:= T \setminus U_q$ we obtain  the statement of the proposition. 

So, we fix an integer $q>0$, and in the rest of the proof we show the above claim. We also set  $r:=h^0(X_t, qsL_t)$ and $l:=qsr$, where the former is independent of $t \in T$ by the above choice of $s$.

Set $W:=\bP((f_* \sO_X(qsL))^*)$. Then,  for any geometric point $\ot \in T$ we have natural bijections:
\begin{equation}
\label{eq:delta_general_fiber:one_div}
\fbox{\parbox{125pt}{$k\big(\ot\big)$-rational points of $
W_{\ot}$}} 
\leftrightarrow
\fbox{\parbox{130pt}{lines through the origin in $H^0\left(X_K,\sO_X(qsL)|_{X_{\ot}} \right)$}}  
\leftrightarrow
\fbox{\parbox{100pt}{$D \in \left|qsL|_{X_{\ot}}\right|$ } },
\end{equation}
where $k\big(\ot\big)$ is the residue field of $\ot$, and $X_K$ and $X_{\ot}$ are the corresponding base-changes, as explained in \autoref{sec:base_changes}.
We consider the open subset
\begin{equation*}
Y \subseteq \underbrace{  W \times_T W \times_T \cdots \times_T W}_{r \textrm{ times}}
\end{equation*}
corresponding to linearly independent lines. That is, for  any geometric point $\ot \in T$, using \autoref{eq:delta_general_fiber:one_div},  we have a natural bijection
\begin{equation}
\label{eq:delta_general_fiber:basis_type_div}
\fbox{\parbox{120pt}{$k\big(\ot\big)$-rational points of $
Y_{\ot}$}} 
\leftrightarrow
\fbox{\parbox{230pt}{$(D_i)=(D_1,\dots,D_r)$ is a  basis of  $\left|qsL|_{X_{\ot}}\right|$ } }
\end{equation}
Denote by $\oy_{(D_i)}$ the geometric point of $Y$ corresponding to $(D_i)$ via the correspondence \autoref{eq:delta_general_fiber:basis_type_div}, where $D_i\in \left|\left.qsL\right|_{X_{\ot}} \right|$. 

Consider the universal family of $q$-basis type divisors, where $\Delta_Y$ is the base-change divisor as defined in \autoref{sec:base_change_relative_canonical_normal_fibers},
\begin{equation*}
g : (Z := X \times_T Y, \Delta':=\Delta_Y ; \Gamma  ) \to  Y 
\end{equation*}
such that for any geometric point $\oy:=\oy_{(D_i)} \in Y$, 
$\Gamma_{\oy} = \sum_{i=1}^{r} \frac{D_i}{l}$.
 Denote by $\pi : Y \to T$ the natural projection.

According to \cite[Lem 8.8]{Kovacs_Patakfalvi_Projectivity_of_the_moduli_space_of_stable_log_varieties_and_subadditvity_of_log_Kodaira_dimension} , the log canonical threshold function $\oy \mapsto \lct \left( \Gamma_{\oy};Z_{\oy}, \Delta'_{\oy} \right)$, which takes values on the geometric points of $Y$,  is lower semi-continuous. Furthermore, the second paragraph of \cite[Lem 8.8]{Kovacs_Patakfalvi_Projectivity_of_the_moduli_space_of_stable_log_varieties_and_subadditvity_of_log_Kodaira_dimension} shows that there is a dense open set $Y_0 \subseteq Y$ such that $\lct \left( \Gamma_{\oy};Z_{\oy}, \Delta'_{\oy} \right)$ is the same for every $\oy \in Y_0$. Applying this iteratedly to the complement of $Y_0$, we obtain that $\oy \mapsto \lct \left( \Gamma_{\oy};Z_{\oy}, \Delta'_{\oy} \right)$ takes only finitely many values on $Y$, say $r_1>r_2>\dots>r_l$, and the level sets are constructible subsets of $Y$. Hence,
\begin{equation*}
L_i:=\left\{\oy \in Y \left| \lct \left( \Gamma_{\oy};Z_{\oy}, \Delta'_{\oy} \right) \geq r_i  \right. \right\}
\end{equation*}
are open sets, and 
  for any geometric point $\oy:=\oy_{(D_j)}$ of $Y$,
\begin{equation*}
 \lct \left(X_{\oy},\Delta_{\oy}; \Gamma_{\oy} \right) =  \lct \left(X_{\oy},\Delta_{\oy}; \sum_{j=1}^r \frac{D_j}{l} \right) = \max \{r_i | (D_j) \in L_i \}.
\end{equation*}
It follows that for any geometric point $\ot \in T$, 
\begin{equation}
\label{eq:delta_general_fiber:delt_inv}
\delta_{qs}\left(X_{\ot}, \Delta_{\ot}\right)= \max\{ r_i | Y_{\ot} \subseteq (L_i)_{\ot}  \}. 
\end{equation}
After the above discussion, our claim follows immediately. Indeed, we just need to choose $a$ to be the smallest integer such that $L_a$ contains the  generic fiber of $\pi$. Then there is a non-empty open set $U \subseteq T$ contained in 
\begin{equation*}
\left( T \setminus \pi(Y  \setminus L_a) \right) \cap \pi(L_a \setminus L_{a-1}).
\end{equation*}
In particular, for any geometric point $\ot \in U$:
\begin{enumerate}
\item $Y_{\ot} \subseteq (L_a)_{\ot}$, and
\item $ (L_a \setminus L_{a-1})_{\ot} \neq \emptyset$ and hence $Y_{\ot} \not\subseteq (L_{a-1})_{\ot}$. 
\end{enumerate}
Therefore, by setting $d:=r_a$, \autoref{eq:delta_general_fiber:delt_inv} implies that $\delta_{qs}\left( X_{\ot}, \Delta_{\ot}\right) = d$ for all geometric points $\ot \in U$.

\end{proof}

\begin{remark}
\label{rem:delta_invariant_non_albraically_closed_field}
We note that one could define the $\delta$-invariant also over over non algebraically closed base fields, with verbatim the same definition as \autoref{def:delta}. If $(Y_K,\Delta_K)$ is a projective klt pair and $N_K$ is a $\bQ$-Cartier divisor  defined over a non-closed field $K$,  and furthermore we choose a basis type divisor $D= \sum_{i=1}^{h^0(Y_K, qN_K)} \frac{D_i}{q}$ (that is, $D_i$ form a $K$-basis of $H^0(Y_K, qN_K)$),  then $\lct (Y_K,\Delta_K, D) = \lct\left(Y_{\oK}, \Delta_{\oK}, D_{\oK} \right)$, where $D_{\oK}$ is a basis type divisor for $N_{\oK}$. Hence, $\delta_q (Y_K,\Delta_K, N_K) \geq \delta_q \left(Y_{\oK}, \Delta_{\oK}, N_{\oK} \right)$. However, $\delta_q (Y_K,\Delta_K, N_K) > \delta_q \left(Y_{\oK}, \Delta_{\oK}, N_{\oK} \right)$ could happen as not all basis type divisors of $N_{\oK}$ come from basis type divisors  of $N_K$. A simple example is if $Y_K$ is a conic not isomorphic to $\bP^1_K$, $\Delta_K=0$, and $N_K = K_{Y_K}^{-1} $. Then, $\delta_q(Y_K,\Delta_K;N_K)=3$, but $\delta_q\left(Y_{\oK},\Delta_{\oK};N_{\oK}\right)=1$. 

In particular, if one takes a conic bundle $f : X \to T$ without a section, and $\eta$ is the generic point of $T$, then for the generic fiber we have $\delta\left(X_{\eta}\right)=2$, but for all geometric fiber (including the geometric generic fiber) outside of the discriminant locus we have $\delta\left(X_{\ot}\right)=1$. So, the $\delta$-invariant is not the same for a general and for the generic point (in general). In particular, one cannot replace "any geometric point $\ot \in T$" in \autoref{prop:delta_general_fiber} with just "any point $t \in T$".
\end{remark}

\begin{remark}
The special case of \autoref{prop:delta_general_fiber} when $d=1$ and $\Delta=0$ (so for $K$-semi-stability via \cite{BJ}) was shown in \cite[Thm 3]{Blum_Liu_The_normalized_volume_of_a_singularity_is_lower_semicontinuous} with other methods.

\end{remark}

\begin{remark}

 \autoref{prop:delta_general_fiber} is very weak version of what is expected to hold. It is conjectured, cf., \cite{Blum_Liu_The_normalized_volume_of_a_singularity_is_lower_semicontinuous}, that $\delta$ is lower semi-continuous, and furthermore the $\delta  \geq 1$ set is also open. Some of this has been proven in \cite[Thm 1.1(i)]{Li_Wang_Xu_On_the_proper_moduli_spaces_of_smoothable_Kahler-Einstein_Fano_varieties} and \cite{Blum_Liu_The_normalized_volume_of_a_singularity_is_lower_semicontinuous}. The full conjecture has been proven recently in \cite{Blum_Liu_Xu_Openness_of_K-semistability_for_Fano_varieties}, after the first version of this paper was published.
 \end{remark}

\section{Growth of sections of vector bundles over curves}
\label{sec:growth}

In this section, we present results about the growth of the number of sections of vector bundles over curves. We apply these in \autoref{sec:semi_positivity} and \autoref{sec:positivity} to vector bundles of the form $f_* \sO_X(q (-K_{X/T}- \Delta - f^* H))$ to obtain many sections of divisors of type $q (-K_{X/T}- \Delta - f^* H)^{(m)}$, where $f : (X, \Delta) \to T$ is a log-Fano family, $H$ is an auxiliary ample divisor on $T$, and $(\_)^{(m)}$ is the fiber product notation of \autoref{sec:product_notation}. 
The precise statement is given in \autoref{prop:vectP2}.

\begin{notation}
\label{notation:semi_stable_vector_bundle}
Let $T$  be a smooth projective curve of genus $g$ over $k$, let $\sE$ be a vector bundle on $T$. Let $\mu(\sE)$ be the slope of $\sE$, namely $\mu(\sE):= \deg \sE/\rk \sE$.

\end{notation}

First we recall well known statements in \autoref{prop:tensor_product_semi_stable}, \autoref{prop:semi-stable_coho_vanishing} and \autoref{prop:semi_stable_globally_generated} concerning semi-stable bundles.

\begin{proposition}
\label{prop:tensor_product_semi_stable}
In the situation of \autoref{notation:semi_stable_vector_bundle}, given two vector bundles $\sE$ and $\sE'$, we have $\mu(\sE \otimes \sE') = \mu(\sE) + \mu(\sE')$; moreover, if both $\sE$ and $\sE'$ are semi-stable, then so is $\sE \otimes \sE'$. 
\end{proposition}

\begin{proof} For the first statement, just remark that $\det(\sE \otimes \sE')=\det(\sE)^{\otimes \rk(\sE') } \otimes \det(\sE')^{\otimes \rk(\sE)}$. The second statement is \cite[Corollary 6.4.14]{Lazarsfeld_Positivity_in_algebraic_geometry_II}.

\end{proof}

\begin{proposition}
\label{prop:semi-stable_coho_vanishing}
In the situation of \autoref{notation:semi_stable_vector_bundle},  if $\sE$ is semi-stable  with $\mu(\sE)> 2g -2$, then $h^1(\sE) = h^0(\omega_T \otimes \sE^*)=0$. 
\end{proposition}

\begin{proof}
We prove the $h^0$ vanishing, and then the $h^1$ vanishing follows by Serre-duality. The bundle $\omega_T \otimes \sE^*$ is also semi-stable and $\mu(\omega_T \otimes \sE^*) = \mu(\omega_T) - \mu(\sE) < 0$. Hence, $h^0(\omega_T \otimes \sE^*)=0$, as a section would give a subbundle of $\sE$ of slope $0$. 
\end{proof}

\begin{proposition}
\label{prop:semi_stable_globally_generated}
In the situation of \autoref{notation:semi_stable_vector_bundle}, if $\sE$ is  semi-stable with $\mu(\sE) \geq 2g$, then $\sE$ is globally generated. 
\end{proposition}

\begin{proof}
Fix a closed point $t \in T$, and let $\sG$ be either $\sE$ or $\sE(-t)$.
Riemann-Roch tells us that
\begin{equation*}
h^0(\sG)
\expl{=}{\autoref{prop:semi-stable_coho_vanishing}}
h^0(\sG) - h^1(\sG)
= \deg \sG + \rk(\sG) (1 - g) = \rk(\sG)(\mu(\sG) + 1-g).
\end{equation*}
In particular, $h^0( \sE) = h^0(\sE(-t)) + \rk(\sG)$. So, by looking at the usual exact sequence:
\begin{equation*}
\xymatrix{
0 \ar[r] & H^0(T, \sE(-t)) \ar[r] & H^0(T, \sE) \ar[r] & H^0(k(t), \sE \otimes k(t)) 
}
\end{equation*}
we see that $\sE$ is in fact generated at $t$. As $t$ was chosen arbitrarily,  $\sE$ is globally generated.
\end{proof}

\begin{notation}
\label{notation:HN_basic}
In the situation of \autoref{notation:semi_stable_vector_bundle},
let $0=\sF^0  \subsetneq \sF^1 \subsetneq \dots \subsetneq \sF^{\ell-1} \subsetneq \sF^{\ell} = \sE$ be the Harder-Narasimhan filtration \cite[Lem 1.3.7 \& 1.3.8]{Harder_Narasimhan_On_the_cohomology_groups_of_moduli_spaces_of_vector_bundles_on_curves} of $\sE$. Set $\mu_i:= \mu\left(\sG^i\right)$ and $r_i:=\rk\left(\sG^i \right)$, where $\sG^i:= \sF^i /\sF^{i-1}$. In particular, we have $\mu_1 > \mu_2 >\dots > \mu_{\ell}$ \cite[Lem 1.3.8]{Harder_Narasimhan_On_the_cohomology_groups_of_moduli_spaces_of_vector_bundles_on_curves}.
\end{notation}

\begin{remark}
\label{weight_versus_HN}
When $T=\mathbb{P}^1$, we have a canonical decomposition 
$$\sE=\bigoplus_{1\leq j\leq l} \sO_T(a_j)\otimes \sO_T^{\oplus n_j},$$ with $a_j< a_{j+1}$. In this case, the Harder-Narasimhan filtration turns out to be $$\sF^i=\bigoplus_{1\leq j\leq i} \sO_T(a_j)\otimes \sO_T^{\oplus n_j},$$ and the slope $\mu_i$ is just $a_i$. 

In the study of K-stability, a key situation is when $T=\bP^1$ is the base of a test configuration $f\colon (\sX,\sL) \to T$ trivially compactified at infinity, and $\sE=f_*(q\sL)$ for some $q>0$. The classical localization formula, see for instance \cite[Example 1]{Wang_GIT}, implies that the Harder-Narasimhan filtration of $\sE$ is equal to the weight filtration with respect to the $\bG_m$-action induced by the test configuration. We can thus think at the Harder-Narasimhan filtration as a generalization of the weight filtration. 

On the other hand, we also note that the Harder-Harasimhan filtration is much more general than the weight filtration as it exists for any family not only for test configuartions, in particular for non-isotrivial families over arbitrary curve bases. This is a crucial point for our argument.
\end{remark}

\begin{proposition}
\label{prop:non_semi_stable_global_generation}
In the situation of \autoref{notation:HN_basic}, if $\mu_i\geq 2g$ for every $i$, then:
\begin{enumerate}
 \item  $H^1(T,\sE)= 0$
 \item $\sE$ is globally generated. 
\end{enumerate}

\end{proposition}

\begin{proof}
We prove both statements at once, by induction on the length $\ell$ of the Harder-Narasimhan filtration of $\sE$. If $\ell=1$, both statements were shown in \autoref{prop:semi-stable_coho_vanishing}. So, we may assume that $\ell>1$. However, then we may include $\sE$ in an exact sequence
\begin{equation}
\label{eq:non_semi_stable_global_generation:induction_seq}
\xymatrix{
0 \ar[r] & \sG \ar[r] & \sE \ar[r] & \sE' \ar[r] & 0
},
\end{equation}
where $\sG$ is semi-stable of rank at least $2g$ and $\sE'$ also satisfies the assumption of the proposition, but with $\ell$ replaced by $\ell-1$. Hence we know both statements for $\sE$ replaced by $\sE'$. Applying now long exact sequence of cohomology to \autoref{eq:non_semi_stable_global_generation:induction_seq} yields:
\begin{equation}
\label{eq:non_semi_stable_global_generation:long_ex}
\xymatrix@C=15pt{
0 \ar[r] &  H^0(T, \sG) \ar[r] & H^0(T, \sE) \ar[r] & H^0(T, \sE') \ar[r] &
0 = H^1(T, \sG) \ar[r] & H^1(T, \sE) \ar[r] & H^1(T, \sE')=0,
}
\end{equation}
where the two vanishings are due to \autoref{prop:semi-stable_coho_vanishing} and induction, respectively. 
This proves our cohomology vanishing statement. 

For the global generation statement, we just use that both $\sG$ and $\sE'$ are globally generated again by \autoref{prop:non_semi_stable_global_generation} and induction respectively. Hence, according to \autoref{eq:non_semi_stable_global_generation:long_ex}, the sections generating these two bundles at a given $t$ give a section generating  $\sE$ at $t$. 
 
\end{proof}

After the above basic statements, we work towards \autoref{prop:vectP2}. This is a statement about tensor powers of vector bundles of positive degree. In particular, we need to understand the Harder-Narasimhan filtration of a tensor power, in terms of the Harder-Narasimham filtration of the original vector bundle. The necessary notation is introduced in \autoref{notation:HN}.

\begin{notation}
\label{notation:HN}
In the situation of \autoref{notation:HN_basic}, fix also a closed point $t \in T$, which will be the point at which the global sections we are interested in would need to become simple tensors. Then the Harder Narasimhan filtration induces a filtration $0=\sF^0_t  \subsetneq \sF^1_t \subsetneq \dots \subsetneq \sF^{\ell-1}_t \subsetneq \sF^{\ell}_t = \sE_t$. Let $\{e_i\}$ be a basis of $\sE_t$ adapted to this filtration. By this, we mean that the intersection of $\sF_t^j$ and $\{e_i\}$ gives a basis of $\sF_t^j$ for every $j$.

Fix an integer $m>0$, which will be the power of the tensor-power of $\sE$ that we are examining. We will parametrize subsheaves of $\sE^{\otimes m}$ that are tensor products of the $\sF^i$'s by elements of $\{1,\dots,\ell\}^m$. Because of \autoref{prop:non_semi_stable_global_generation}, we will be particularly interested in subsheaves with slope at least $2g$. So, consider the subset of $\{1, \dots, \ell\}^m$ defined by:
\begin{equation*}
 S_m:= \left\{(s_1, \dots, s_m)\in \{1, \dots, \ell\}^m \ \left| \quad \displaystyle\sum_{j=1}^m \mu_{s_j} \geq 2g \right.\right\}
\end{equation*}
As we are interested in a filtration of $\sE^{\otimes m}$, we will need an ordering on $\{1, \dots, \ell\}^m$. First we introduce a partial ordering:  for any two $s, s' \in \{1, \dots, \ell\}^m $, we say that
\begin{itemize}
 \item  $s \geq s'$ if $s_j \geq s_j'$ for all $1 \leq j \leq m$, and
 \item  $s > s'$ if $s \geq s'$ and there is a $1 \leq j \leq m$ such that $s_j > s_j'$.
\end{itemize}
Note that $S_m$ is closed in the downwards direction, that is, whenever $s \in S_m$, and $s' <s$, then $s' \in S_m$. 
  
We also assign a \emph{minimal slope} $\nu(s)$ to $s \in \{1, \dots, \ell\}^m $  given by $\sum_{j=1}^m \mu_{s_j}$, which will be the actual minimal slope of the corresponding sheaf in the Harder-Narasimham filtration of $\sE^{\otimes m}$. Note that $\nu(s) \geq 2g$ if and only if $s \in S_m$. 

After the above, we arrange the elements of $\{1, \dots, \ell\}^m$ in a  decreasing order with respect to $\nu( \_ )$, such that $S_m$ consists of the first $d$ elements
\begin{equation*}
S_m= \left\{s^1, \dots, s^d\right\}\textrm{, and }\{1, \dots, \ell\}^m = S_m \cup \left\{s^{d+1}, \dots, s^e\right\},
\end{equation*}
where  $\nu(s^c)$ is a (not necessarily strictly) decreasing function of $c$. In particular, 
whenever $s^{c'} < s^{c}$, then $c'<c$. As expected, $s^c_j$ denotes the coordinates of $s^c$, that is,  $s^c = ( s^c_1,\dots, s^c_m)$.

For any integer $1 \leq c \leq e$, we define then the following subbundles of $\sE^{\otimes m}$
\begin{equation*}
 \widetilde{\sF}^c:=\bigotimes_{j=1}^m \sF^{s^c_j} \textrm{, and } \sH^c:= \sum_{i=1}^c \widetilde{\sF}^i .
\end{equation*}
In fact, it is not clear from the definition that  $\sH^c$ is a subbundle as oppoed to just a coherent subsheaf. We prove in \autoref{prop:HN_Tensor_product} that it is indeed a subbundle. For simplicity we also define 
\begin{equation*}
\widetilde{\sG}^c:= \bigotimes_{j=1}^m \sG^{s^c_j}
\end{equation*}
Recall that
\begin{equation}
\label{eq:HN:G} 
\rk \left( \widetilde{\sG}^c \right)  = \prod_{j=1}^m r_{s^c_j}
\textrm{, and }
\mu\left(\widetilde{\sG}^c\right)=\sum_{j=1}^m \mu\left( \sG^{s^c_j}\right) = \sum_{j=1}^m \mu_{s_j^c}.
\end{equation}

After the above notational preparation, it is quite straight-forward to state and prove the description of the Harder-Narasimham filtration of $\sE^{\otimes m}$ that we need:

\end{notation}

\begin{proposition}
\label{prop:HN_Tensor_product}
In the situation of \autoref{notation:HN}:
\begin{enumerate}
\item \label{itm:HN_Tensor_product:subbundle} For each integer $1 \leq c \leq e$,  $\sH^c$ is a subbundle of $\sE^{\otimes m}$.
\item \label{itm:HN_Tensor_product:filtration} The filtration $0 \subsetneq \sH^1 \subsetneq \sH^2 \subsetneq \dots \subsetneq \sH^d$ is a refinement of the Harder-Narsimhan filtration of $\sH^d$. More precisely, the quotients are semi-stable with (not necessarily strictly) decreasing slopes. Furthermore, all these slopes are at least  $2g$. Even more precisely, for each integer $1 \leq c \leq e$,
\begin{equation*}
\sH^c/\sH^{c-1} \cong \widetilde{\sG}^c.
\end{equation*}

\item \label{itm:HN_Tensor_product:simple_tensor} For each integer $1 \leq c \leq e$,  $\sH^c_t \subseteq \sE_t^{\otimes m}$ is spanned by simple tensors of $e_i$.
\end{enumerate}

\begin{proof}
For each integer $1 \leq c \leq e$ we have a surjective homomorphism:
\begin{equation}
\label{eq:HN_Tensor_product:isomorphism_theorem}
\sH^c / \sH^{c-1} 
\expl{=}{by definition}
\left. \left(\sH^{c-1} + \widetilde{\sF}^c\right)\right/ \sH^{c-1}  
\explshift{-40pt}{\cong}{isomorphism theorem} 
\widetilde{\sF}^c\left/ \left(\widetilde{\sF}^c \cap \sH^{c-1} \right) \right.
\expl{\twoheadleftarrow}{$\sum_{c'< c} \widetilde{\sF}^{c'} \subseteq \widetilde{\sF}^c \cap \sH^{c-1} $}
\widetilde{\sF}^c \left/  \left( \sum_{s_{c'}< s_{c}} \widetilde{\sF}^{c'} \right) \right.
\cong 
\widetilde{\sG}^c
\end{equation}
So, 
\begin{multline}
\label{eq:non_semi_stable_global_generation:inequality}
\rk \left( \sE^{\otimes m} \right) = \sum_{c=1}^e \rk \left( \sH^c/ \sH^{c-1}\right) 
\expl{\leq}{\autoref{eq:HN_Tensor_product:isomorphism_theorem}}
\sum_{c=1}^e \rk \left( \widetilde{\sG}^c\right)
= \sum_{(i_1, \dots, i_m) \in \{1,\dots,\ell\}^m} \left( \prod_{j=1}^m r_{i_j} \right) 
\\ = \left( \sum_{i=1}^{\ell} r_i \right)^m 
 = \rk \left( \sE^{\otimes m} \right).
\end{multline}
Hence, we have equality in the middle of \autoref{eq:non_semi_stable_global_generation:inequality}, and hence the last homomorphism in \autoref{eq:HN_Tensor_product:isomorphism_theorem} is an isomorphism for all $c$. 

In particular, for all $1 \leq c \leq e$, there is an exact sequence:
\begin{equation}
\label{eq:HN_Tensor_product:extension}
\xymatrix{
0 \ar[r] &  \sH^{c-1} \ar[r] & \sH^c  \ar[r] & \widetilde{\sG}^c \ar[r] & 0
}
\end{equation}
This concludes \autoref{itm:HN_Tensor_product:subbundle}, as both $\sH^c$ and $\sE^{\otimes m} / \sH^c$ are iterated extensions of vector bundles, hence they are both vector bundles. 

Point \autoref{itm:HN_Tensor_product:simple_tensor} also follows immediately from the definition of $\sH^c$ as it is a sum of product type subbundles.  Lastly,  \autoref{itm:HN_Tensor_product:filtration} also follows directly from \autoref{eq:HN_Tensor_product:extension}.

\end{proof}

\end{proposition}

\begin{notation}
\label{notation:globally_generated_part}
In the situation of \autoref{notation:HN} (in fact, for introducing the following notation we only need the first two paragraphs of \autoref{notation:HN}),  let $G_{m,t}$ be the $k$-linear subspace of 
\begin{equation*}
\im \big( \underbrace{H^0(T,\sE^{\otimes m}) \to \left(\sE^{\otimes m}\right)_t}_{\textrm{evaluation map}} \big)
\end{equation*}
spanned by pure tensor in the $e_i$. 
 
\end{notation}

\begin{theorem}\label{prop:vectP2}
In the situation of \autoref{notation:globally_generated_part},
if $\deg \sE>0$, then 
$$
\lim_{m\to \infty}\frac{\dim_k G_{m,t}}{\dim_k \sE_t^{\otimes m}}=1.
$$
\end{theorem}

\begin{proof}
Combining \autoref{prop:non_semi_stable_global_generation} and \autoref{prop:HN_Tensor_product} yields
$
\sH^d_t\subseteq G_{m,t}.
$
So, it is enough to prove that
$$
\lim_{m\to \infty}\frac{\rk \sH^d}{\rk \sE^{\otimes m}}=1.
$$
By Proposition \autoref{prop:HN_Tensor_product}, item \autoref{itm:HN_Tensor_product:filtration}, and by equation \autoref{eq:HN:G} if we set $r:= \rk \sE$ and $p_i:=\frac{r_i}{r}$, then 
\begin{equation}
\label{eq:vectP2:to_probability}
\frac{\rk \sH^d}{\rk \sE^{\otimes m}}  = \sum_{c =1}^d \frac{\rk \widetilde{\sG}^c}{\rk \sE^{\otimes m}}  = \sum_{c=1}^d \prod_{j=1}^m \frac{r_{s^c_j}}{r} = \sum_{s \in S_m} \prod_{j=1}^m \frac{r_{s_j}}{r} = \sum_{s \in S_m} \prod_{j=1}^m p_{s_j}. 
\end{equation}
As $\sum_{i=1}^{\ell} p_i=1$, we may define a discrete probability space $\sX$ on $\ell$ elements $\{1,\dots, l\}$ with measures $p_1,\dots, p_l$ respectively. Let $X_j$ be a sequence of independent identically distributed random variables of $\sX$ that take value $\mu_i$ on $i$, and let $Z_m:=\sum_{i=j}^m X_j$. With this language, \autoref{eq:vectP2:to_probability} tells us that 
\begin{equation*}
\frac{\rk \sH^d}{\rk \sE^{\otimes m}} = P\left(\sum_{i=j}^m X_j \geq 2g \right)= P\left(Z_m \geq 2g \right),
\end{equation*}
where $P(\dots)$ denotes the probability of the given condition. Hence we are left to show that 
\begin{equation}
\label{eq:vectP2:goal}
\lim_{m \to \infty} P\left(Z_m \geq 2g \right) =1 .
\end{equation}
Consider now, the Central Limit Theorem of probability theory as for example in \cite[Thm 3.4.1]{Durrett_Probability_theory_and_examples}. Note that as $\sX$ is a finite metric space both the expected value $\mu$ and the variance $\sigma^2$ of $X_j$ are finite. Then the central limit theorem states that the random variable $\frac{Z_m - m \mu}{\sqrt{m}}$ weakly converges to a normal distribution $\Phi$ with expected value $0$ and covariance $\sigma^2$. In particular, this induces a convergence on the level of distribution functions, or more precisely we would like to use the following convergence, which holds for each real number $A$ and it is shown for example in \cite[Thm 3.2.5.iv]{Durrett_Probability_theory_and_examples}:
\begin{equation}
\label{eq:vectP2:clt}
\lim_{m \to \infty} P \left( \frac{Z_m - m \mu}{\sqrt{m}} \geq A  \right) = P(\Phi \geq A )
\end{equation}
\begin{inlineclaim} For each fixed real number $A$ there is an integer $m_A>0$ such that for all integers $m \geq m_A$:
\begin{equation}
\label{eq:vectP2:containment}
\frac{Z_m - m \mu}{\sqrt{m}} \geq A \quad  \Rightarrow \quad Z_m \geq 2g
\end{equation}
\end{inlineclaim}

\begin{proof}[Proof of the claim]
For this, note first that 
\begin{equation*}
\mu= \sum_{i=1}^{\ell} \mu_i p_i = \frac{\sum_{i=1}^{\ell} \mu_i r_i}{r} = \frac{\deg \sE}{\rk \sE} > 0. 
\end{equation*}
Hence, if we assume that $ \frac{Z_m - m \mu}{\sqrt{m}} \geq A $ then there is an integer $m_A$ such that 
\begin{equation*}
2g 
\expl{\leq}{for $m \geq m_A$ }
 A \sqrt{m} + m \mu
\expl{\leq}{$ \frac{Z_m - m \mu}{\sqrt{m}} \geq A $} 
Z_m
.
\end{equation*}
\end{proof}
We continue the proof of \autoref{prop:vectP2}: Combining our claim and \autoref{eq:vectP2:clt} we obtain that 
\begin{equation*}
\liminf_{m \to \infty} P \left(Z_m \geq 2g \right) \geq P(\Phi \geq A)
\end{equation*}
As this is true for all real numbers $A$, and $\lim_{A \to - \infty} P( \Phi \geq A)=1$,  we obtain that 
\begin{equation*}
\liminf_{m \to \infty} P \left(Z_m \geq 2g \right) =1 
\expl{\Rightarrow}{$ \forall m \  : \  P \left(Z_m \geq 2g \right) \leq 1$}
\lim_{m \to \infty} P \left(Z_m \geq 2g \right) =1
\end{equation*}
This is exactly the statement of \autoref{eq:vectP2:goal}, which was our goal to prove. 
      
\end{proof}

\begin{remark}
We note that in the proof of \autoref{prop:vectP2}, one can replace the Central Limit Theorem by the weaker statement of Chebyshev's inequality. Indeed, using the notation of the proof, as the variance of $Z_m$ is $m \sigma^2$:
\begin{equation*}
1- P(Z_m - m \mu \geq \sqrt{m}A ) \leq 
P\left(|Z_m - m \mu |\geq \frac{-A}{\sigma}\sqrt{m}\sigma \right) 
\expl{\leq}{Chebyshev's inequalty} 
\frac{\sigma^2}{A^2}  \to 0 \quad (\textrm{as }A \to -\infty).
\end{equation*}
\end{remark}

\section{Ancillary statements}

Here we gather smaller statements that are used multiple times in \autoref{sec:semi_positivity}, \autoref{sec:nefness_threshold} and \autoref{sec:positivity}.

\subsection{Normality of total spaces}

In  the next sections we work mostly in the following setup:

\begin{notation}
\label{notation:CM_semi_positivity}
 Let  $f \colon (X,\Delta) \to T$ satisfy the following assumptions:
 \begin{enumerate}
  \item $T$ is a smooth, projective curve, 
  \item $X$ is a normal, projective variety of dimension $n+1$,
  \item $f$ is a projective and surjective morphism with connected fibers, 
  \item $\Delta$ is an effective $\bQ$-divisor on $X$, 
  \item $-(K_X + \Delta)$ is an  $f$-ample $\bQ$-Cartier divisor. 
  \item $(X_t,\Delta_t)$ is klt for general $t \in T$.  \end{enumerate}
\end{notation}

\begin{lemma}
\label{lem:reducing_reduced_fibers}
In the situation of \autoref{notation:CM_semi_positivity}, there exists a finite morphism from  a smooth projective curve $\tau :S \to T$ such that if $g : Y \to S$ is the normalization of the pullback of $f$, and $\pi : Y \to X$ the induced morphism, then $g$ has reduced fibers, and there is an effective $\bQ$-divisor $\Gamma$ on $Y$ such that 
\begin{enumerate}
 \item \label{itm:reducing_reduced_fibers:canonical_bundle_formula} $\pi^* (K_{X/T} + \Delta) = K_{Y/S} + \Gamma$,
 \item \label{itm:reducing_reduced_fibers:CM} $\lambda_g = \sigma^* \lambda$, where $\lambda_g$ is the CM line bundle for $g$.
\end{enumerate}
\end{lemma}

\begin{proof}
Let $\tau$ be any finite cover such that at the closed points $t \in T$ over which the fiber $X_t$ is non reduced, the ramification order of $\tau$ is divisible by all the multiplicities of all the components of $\tau$. Then, $g$ will have reduced fibers, and \autoref{sec:base_change_relative_canonical_smooth_base}
 implies the existence of $\Gamma$, denoted by $\Delta_Z$ there. Finally, \autoref{prop:relative_canonical_base_change_normal}.\autoref{itm:relative_canonical_base_change:base_change} yields point \autoref{itm:reducing_reduced_fibers:canonical_bundle_formula}, and \autoref{prop:CM_base_change}.\autoref{itm:CM_base_change:smooth_base} yields point \autoref{itm:reducing_reduced_fibers:CM}.

\end{proof}

\begin{lemma}
\label{lem:normal_stable_pullback_fiber_product}
If $f : X \to T$ is a surjective morphism from a normal variety to a smooth projective curves with reduced fibers, $m>0$ is an integer and $ \tau : S \to T$ is a finite morphism from another smooth curve, then
\begin{enumerate}
\item \label{itm:normal_stable_pullback_fiber_product:pullback} $X \times_T S$ is normal, and
\item \label{itm:normal_stable_pullback_fiber_product:fiber_product} $X^{(m)}$ is normal, see \autoref{sec:product_notation} for the product notation.
\end{enumerate}

\end{lemma}

\begin{proof}
First we note that $f$ is flat and hence so is $f^{(m)} : X^{(m)} \to T$ by induction on $m$ and the stability of flatness under base-change.

We know that a variety $Z$ is normal if and only if  it is $S_2$ and $R_1$. In the particular case, when $Z$ maps to a smooth curve $U$ via  a flat morphism $g$, then $Z$ is $S_2$ if and only if the general fibers of $g$ are $S_2$ and the special ones are $S_1$ (so without embedded points) \cite[6.3.1]{Grothendieck_Elements_de_geometrie_algebrique_IV_II} \cite[12.2.4.i]{Grothendieck_Elements_de_geometrie_algebrique_IV_III}, and it is $R_1$ if the general fibers are $R_1$ and the special ones are $R_0$ (so reduced) \cite[12.2.4.ii]{Grothendieck_Elements_de_geometrie_algebrique_IV_III}. It is immediate then that this characterization of $S_2$ and $R_1$ propagates both to fiber powers and to base-changes. 
\end{proof}

\subsection{Semi-positivity engine}

\begin{proposition}
\label{prop:semi_positivity_engine_downstairs}
Let $f : (X, \Delta) \to T$  be a surjective morphism from a normal, projective pair to a smooth curve such that $(X_t,\Delta_t)$ is klt for general $t \in T$ (recall that $X_t$ is normal for $t \in T$ general, $\Delta$ is $\bQ$-Cartier at the codimension $1$ points of $X_t$, and hence $\Delta_t$ makes sense), and let $L$ be a Cartier divisor on $X$ such that $L- K_{X/T} - \Delta$ is an $f$-ample and nef $\bQ$-Cartier divisor. Then, $f_* \sO_X(L)$ is a nef vector bundle. 
\end{proposition}

\begin{proof}
According to \autoref{lem:reducing_reduced_fibers} we may assume that the fibers of $f$ are reduced. According to \cite[Lem 3.4]{Patakfalvi_Semi_positivity_in_positive_characteristics}, it is enough to prove that for all integers $m>0$, the following vector bundle is generated at a general $t \in T$ by global sections: 
\begin{equation*}
\omega_T (2t) \otimes \bigotimes_{i=1}^m f_* \sO_X(L) 
\explparshift{170pt}{-70pt}{\cong}{$m-1$-times use of \cite[Lem 3.6]{Kovacs_Patakfalvi_Projectivity_of_the_moduli_space_of_stable_log_varieties_and_subadditvity_of_log_Kodaira_dimension}, and see \autoref{sec:product_notation} for the fiber product notation}
\omega_T (2t) \otimes f^{(m)}_*  \sO_{X^{(m)}} \left( L^{(m)} \right)
\expl{\cong}{projection formula}
f^{(m)}_* \sO_{X^{(m)}} \left( L^{(m)} + \left( f^{(m)} \right)^* K_T + 2 X^{(m)}_t \right)
.
\end{equation*}
For that it is enough to prove that the natural restriction homomorphism $
H^0\left( X^{(m)}, N \right) \to H^0\left( X_t^{(m)}, N_t \right)$
is surjective, where 
\begin{equation*}
 N:=L^{(m)} + \left( f^{(m)} \right)^* K_T + 2 X^{(m)}_t  = K_{X^{(m)}} +  \Delta^{(m)} +  (L - K_{X/T} - \Delta)^{(m)} + 2 X_t^{(m)}.
\end{equation*}
We note here that  according to \autoref{lem:normal_stable_pullback_fiber_product}, $X^{(m)}$ is normal. Furthermore, $K_{X^{(m)}} +  \Delta^{(m)} = (K_{X/T} + \Delta)^{(m)} +  \left( f^{(m)} \right)^* K_T$ is $\bQ$-Cartier. We also note that the only generality property of $t$ that we use below is that $X_t$ is normal, $X_t \subsetneq \Supp \Delta_t$ and  $(X_t, \Delta_t)$ is klt. Hence, at this point, we fix a $t$ with such properties.

Set $\sI:=\sJ_{\left( X^{(m)}, \Delta^{(m)}\right)}$, where $\sJ$ denotes the multiplier ideal of the corresponding pair. Then for the above surjectivity the next diagram, the top row of which is exact, shows that it is enough to prove the vanishing of $H^1\left( X^{(m)}, \sI \otimes \sO_{X^{(m)}} \left(  N - X_t^{(m)}\right)\right)$.
\begin{equation*}
\xymatrix@C=10pt@R=10pt{
 H^0\left( X^{(m)}, \sI \otimes   \sO_{X^{(m)}}\left( N \right)\right) \ar[r] \ar@{^(->}[d]
&
H^0\left( X^{(m)}_t,  N|_{X^{(m)}_t}\right) \ar[r] 
& H^1\left( X^{(m)}, \sI \otimes \sO_{X^{(m)}}\left(  N - X^{(m)}_t\right)\right)  \\
H^0\left( X^{(m)},   N \right) \ar[ur] 
}
\end{equation*}
We note that here we used that $\left(X_t^{(m)}, \Delta_t^{(m)} \right)$ is klt by \autoref{lem:product}, and hence by inversion of adjunction \cite[Thm 5.50]{Kollar_Mori_Birational_geometry_of_algebraic_varieties} so does $\left(X^{(m)}, \Delta^{(m)} \right)$ in a neighborhood of $X_t^{(m)}$. This then implies that $\sI$ is trivial in a neighborhood of $X_t^{(m)}$.

We conclude by noting that the above cohomology vanishing is given by Nadel-vanishing as
\begin{equation*}
 N -  X_t^{(m)} = K_{X^{(m)}} +  \Delta^{(m)} +  \underbrace{(\underbrace{L - K_{X/T} - \Delta}_{\textrm{nef and $f^{(m)}$-ample}})^{(m)} +  \underbrace{X_t^{(m)}}_{\left(f^{(m)}\right)^*\textrm{ ample}}}_{\textrm{  ample}}.
\end{equation*}

\end{proof}

\begin{corollary}
\label{cor:semi_positivity_engine_upstairs}
Let $f : (X, \Delta) \to T$  be a surjective morphism from a normal, projective pair to a smooth curve such that $(X_t,\Delta_t)$ is klt for some (or equivalently general) $t \in T$, and let $L$ be an $f$-nef $\bQ$-Cartier $\bQ$-divisor on $X$  such that 
\begin{enumerate}
 \item there is  a $\bQ$-Cartier $\bQ$-divisor $N$ on $T$ such that $L + f^* N $ is Cartier,
\item \label{itm:semi_positivity_engine_upstairs:sections_non_zero}  $L_t = (L+ f^*N)_t$ is globally generated for $t \in T$ general, and
 \item $L  - K_{X/T} - \Delta$ is an $f$-ample and nef $\bQ$-Cartier $\bQ$-divisor.

\end{enumerate}
Then $L $ is nef.
\end{corollary}

\begin{proof}
 According to \autoref{lem:reducing_reduced_fibers} we may assume that the fibers of $f$ are reduced, and by further pullback using \autoref{lem:normal_stable_pullback_fiber_product}.\autoref{itm:normal_stable_pullback_fiber_product:pullback}, we may also assume that $N$ is Cartier, whence $L$ is also Cartier. Then, we may apply \autoref{prop:semi_positivity_engine_downstairs} yielding that $f_* \sO_X(L)$ is nef. 
 
Note now  that cohomology and base change always holds over a dense open set. So, for general $t \in T$ we have a commutative diagram as follows:
\begin{equation}
\label{eq:semi_positivity_engine_upstairs}
\xymatrix{
\ar@{->}[d]^{\cong} f^* f_* \sO_X(L) |_{X_t} \ar[r] & \sO_X(L)  |_{X_t} \ar@{->}[d]^{\cong} \\
\sO_{X_t} \otimes H^0(X_t, L|_{X_t}) \ar[r] & \sO_{X_t}( L|_{X_t})
}
\end{equation}
Assumption \autoref{itm:semi_positivity_engine_upstairs:sections_non_zero} tells us that the bottom arrow of diagram \autoref{eq:semi_positivity_engine_upstairs} is surjective. Hence, so is the top arrow, and then  the natural homomorphism $f^* f_* \sO_X(L) \to \sO_X(L)$ is surjective over a dense open set of $T$. 

As $L$ is $f$-nef, we only have to show that if $C$ is a horizontal curve, then $C \cdot L \geq 0$. However, by the previous paragraph, $f^* f_* \sO_X(L)|_C \to \sO_X(L)|_C = \sO_C( L |_C)$ is generically surjective. Hence, $\sO_C(L|_C)$ is a generically surjective image of a nef vector bundle. So, we obtain that $0 \leq \deg \sO_C(L|_C) = C \cdot L$. 
 \end{proof}

\section{Semi-positivity  }
\label{sec:semi_positivity}

In this section we prove our semi-positivity results. Here, and also in \autoref{sec:positivity} we use extensively the fiber product notation explained in \autoref{sec:product_notation}.

\subsection{Framework and results}

The main result of the section is the following, from which the statements of the introduction will follow in a quite straightforward manner.

\begin{theorem}
\label{thm:semi_positivity_curve}
In the situation of \autoref{notation:CM_semi_positivity}, if $\delta\left(X_{\ot},\Delta_{\ot}\right) \geq 1$ for a very general geometric point $\ot \in T$, then $\deg \lambda_{f,\Delta} \geq 0$. 
\end{theorem}

\subsection{Proofs}
\label{sec:semi_pos_proof}

The proof of \autoref{thm:semi_positivity_curve} will be by contradiction with the next proposition.  

\begin{proposition}\label{prop:cont}
In the situation of \autoref{notation:CM_semi_positivity}, let $H$ be an ample $\bQ$-divisor  on $T$. Then, there do not exist $\bQ$-Cartier divisors $\Gamma$ and $\widetilde{\Gamma}$ on $X$ such that:
\begin{enumerate}
 \item\label{class} $\Gamma + \widetilde{\Gamma} \sim_{\bQ} -K_{X/T}-\Delta-f^* H$,
  \item \label{ample} $\widetilde{\Gamma}$ is nef, and 
 \item \label{klt} $(X_t,\Delta_t +\Gamma_t)$ is klt for  $t \in T$ general. 
\end{enumerate}
\end{proposition}
\begin{proof}
Assume that there exist $\Gamma$ and $\widetilde{\Gamma}$ as above. 
Let $a>0$ be a rational number such that $-K_{X/T} - \Delta + af^* H$ is ample. Fix a rational number $\varepsilon > 0$ such that $\varepsilon a - (1-\varepsilon)< 0$. Apply then \autoref{cor:semi_positivity_engine_upstairs} by  setting the $L$, $N$ and $\Delta$ of \autoref{cor:semi_positivity_engine_upstairs} to be  respectively $(\varepsilon a - (1-\varepsilon)) f^*H$, $-(\varepsilon a - (1-\varepsilon)) H$, and $\Delta + (1-\varepsilon) \Gamma$. These choices satisfy the assumptions of \autoref{cor:semi_positivity_engine_upstairs} by the right hand side of equation \autoref{eq:cont}, and it  yields a contradiction as indicated under the left hand side term of \autoref{eq:cont}.
\begin{equation}
\label{eq:cont}
\underbrace{  (\varepsilon a - (1-\varepsilon)) f^*H}_{\varepsilon a - (1-\varepsilon)<0 \Rightarrow \textrm{ this is not nef}}  
 \sim_{\bQ}
 \underbrace{K_{X/T} + \Delta + (1-\varepsilon) \Gamma }_{(X_t,\Delta_t +(1-\varepsilon) \Gamma_t ) \textrm{ is klt }} 
+ \underbrace{(1-\varepsilon)\widetilde{\Gamma} + \varepsilon\left(- K_{X/T} -\Delta + af^* H \right)}_{\textrm{ample}}  
\end{equation}
\end{proof}

\begin{proof}[Proof of \autoref{thm:semi_positivity_curve}]
As both the consequences and the conditions of the theorem are invariant under base-extension to another algebraically closed field, we may assume that $k$ is uncountable. In particular, whenever a property is true for  very general geometric fibers, it is also true for some closed fibers. That is, by removing countably many proper closed sets from a variety over $k$, there are some closed points left. The reason is the following:  by cutting down with hyperplanes, this statement can be reduced to curves, where it is true because removing countably many closed points from a curve over an uncountable field leaves uncountably many points of the curve intact. 

First, according to \autoref{lem:reducing_reduced_fibers} we may assume that all fibers of $f$ are reduced. This is to guarantee that the $m$-times iterated fiber product $X^{(m)}$ is normal for any integer $m>0$, according to \autoref{lem:normal_stable_pullback_fiber_product}. 

We argue by contradiction, so assume  that $\deg \lambda_{f,\Delta}<0$.  For $m$ big enough, we are going to produce divisors $\Gamma$ and $\widetilde{\Gamma}$ on $X^{(m)} $ whose existence contradicts \autoref{prop:cont}. 

Fix a closed point $t$ in $T$ such that $X_t$ is normal, $X_t \not\subseteq \Supp \Delta_t$, $(X_t,\Delta_t)$ is klt and $\delta(X_t, \Delta_t) \geq 1$, using \autoref{prop:delta_general_fiber}.  Let $H$ be an ample line bundle on $T$. Fix rational numbers $a,\varepsilon>0$ and $0<c<1$ and an integer $q>0$, such that:
\begin{enumerate}[itemsep=4pt]
 \item the intersection product inequality $(-K_{X/T} -\Delta- \varepsilon f^* H)^{n+1}>0$ holds. This is possible because \autoref{def:CM_log}
and the assumption $\deg \lambda_{f,\Delta}<0$ imply that $ (-K_{X/T} - \Delta)^{n+1}>0$. Set $M:= -K_{X/T} - \Delta - \varepsilon f^* H$.
 \item  $D:=-K_{X/T}-\Delta+af^*H$ is ample.
 \item \label{negative}
$c<\frac{\varepsilon}{a+\varepsilon}$.
\item \label{Cartier} $qM$ is Cartier, which is possible, as $M$ is $\bQ$-Cartier.
\item \label{higher_cohomology} $R^if_*\sO_X(qM)=0$ for all $i>0$, which is possible, as $M$ is $f$-ample.
\item \label{degree} $\deg\left(f_*\sO_X(qM)\right)>0$, using \autoref{lem:Mumford_line_bundles_over_curve}.
\item \label{delta} $\delta_q(X_t,\Delta_t)>1-c$, using \autoref{thm:BJ}.
\end{enumerate}
From now on, let $\sE:=f_*\sO_X(qM)$. Remark that according to \cite[Lemma 3.6]{Kovacs_Patakfalvi_Projectivity_of_the_moduli_space_of_stable_log_varieties_and_subadditvity_of_log_Kodaira_dimension} for every integer $m>0$,
$$
\sE^{\otimes m}=f_*^{(m)} \sO_{X^{(m)}} \left( q M^{(m)} \right) \cong f_*^{(m)} \sO_{X^{(m)}}\left( q\left( -  K_{X^{(m)}/T} -\Delta^{(m)}-  m\varepsilon \left(f^{(m)}\right)^* H\right)\right),
$$
and by item \autoref{higher_cohomology}, the following base change holds
$$
\sE^{\otimes m}_t=H^0\left(X_t^{(m)}, q\left( -  K_{X_t^{(m)}} -\Delta^{(m)}_t \right)  \right).
$$
In general, it is not possible to lift a basis of $\sE_t$ to sections of $\sE$. However, thanks to \autoref{prop:vectP2}, we can choose a basis $e_i$ of $\sE_t$, an integer $m>0$, and $\ell$ global sections $s_i$ of $\sE^{\otimes m}$ so that the sections $s_i$, when restricted over $t$, are linearly independent pure tensor in the $e_i$, and furthermore
\begin{equation}\label{bound}
\frac{\ell}{h^0\left(X_t^{(m)},-q\left(K_{X_t^{(m)}}+\Delta_t^{(m)}\right)\right)}
>
\underbrace{\frac{1-c}{\delta_q(X_t,\Delta_t)}}_{\parbox{57pt}{\tiny $<1$ according to assumption \autoref{delta}}} \,.
\end{equation}
We are now ready to construct $\Gamma$ and $\widetilde{\Gamma}$ on $X^{(m)}$ as in \autoref{prop:cont}. We let
$$ \Gamma:=(1-c)\frac{1}{q\ell}\sum_{i=1}^{\ell}\{s_i=0\}\,,$$
and
$$ \widetilde{\Gamma}:=cD^{(m)}  \,.$$ 
To complete the proof of \autoref{thm:semi_positivity_curve}, we have to prove that $\Gamma$ and $\widetilde{\Gamma}$ are as in \autoref{prop:cont}, with $f$ replaced by $f^{(m)}$. To check item \autoref{class}, remark that
$$\Gamma+\widetilde{\Gamma} \sim_{\bQ} -K_{X^{(m)}/T}-\Delta^{(m)}+ m\left( c a - (1 - c) \varepsilon \right) \left(f^{(m)}\right)^*H.$$
Furthermore,  because of assumption \autoref{negative}, $c a - (1 - c) \varepsilon <0$ holds; so, we may apply \autoref{prop:cont} replacing $H$ by $-m\left( c a - (1 - c) \varepsilon \right) H$. Item \autoref{ample} of \autoref{prop:cont} follows from the ampleness of $D$.

To prove of item \autoref{klt} of \autoref{prop:cont}, we compute the log canonical threshold. We first remark that, since the sections $s_i$ restricted to $X_t^{(m)}$ are linearly independent pure tensors in the $e_i$, we have that
$$
 \frac{\ell}{h^0\left(X_t^{(m)},-q\left(K_{X_t^{(m)}}+\Delta^{(m)}_t\right)\right)} \Gamma_t \leq (1-c)P
$$
for the $q$-product basis type divisor  $P$  on $X_t^{(m)}$ associated to $\{e_i\}$, as in \autoref{def:prod_basis} and \autoref{rem:prod_basis}. Using \autoref{prop:prod_basis}, we obtain that $\lct\left(X^{(m)}_t,\Delta^{(m)}_t;P_t\right)\geq \delta_q(X_t,\Delta_t)$. This yields 
$$
\lct\left(X_t^{(m)},\Delta_t^{(m)};\Gamma_t\right) 
\geq 
\frac{\delta_q(X_t,\Delta_t)\ell}{(1-c)h^0\left(X_t^{(m)},-q\left(K_{X_t^{(m)}}+\Delta^{(m)}_t\right)\right)} 
\expl{>}{rearranging inequlaity \autoref{bound}} 
1.
$$
Hence, all assumptions of \autoref{prop:cont} are verified, implying that $\Gamma$ and $\widetilde{\Gamma}$ cannot exist. Therefore, we obtained a contradiction with our initial assumption that $\deg \lambda_{f,\Delta} <0$. 
\end{proof}

\begin{proof}[Proof of \autoref{thm:semi_positive_boundary}]
\emph{The proof of point \autoref{itm:semi_positive_boundary:pseff}:}
As at the beginning of the proof of \autoref{thm:semi_positivity_curve}, we may assume that $k$ is uncountable. 
 According to \cite[Thm 0.2]{Boucksom_Demailly_Paun_Peternell_The_Pseudo_effective_cone_of_a_compact_Kahler_manifold_and_varieties_of_negative_Kodaira_dimension}, it is enough to show that $\lambda_{f,\Delta} \cdot C \geq 0$ for every morphism $\iota: C\to X$ from a smooth projective curve  such that $C \to \iota(C)$ is the normalization and  $\iota(C)$ is a very general curve in a  family covering $T$.   In particular,  for a very general closed point $t \in \iota(C)$, $X_t$ is normal, $(X_t, \Delta_t)$ is klt and $\delta\left(X_t, \Delta_t\right) \geq 1$. Let $Z \to X_C$ be the normalization, $g : Z \to C$ the induced morphism and $\Delta_Z$ the boundary induced by $\Delta$ on $Z$ as explained in \autoref{sec:base_change_relative_canonical_smooth_base}. According to \autoref{prop:relative_canonical_base_change_normal}.\autoref{itm:relative_canonical_base_change:base_change}, $g : (Z, \Delta_Z) \to C$ satisfies the assumptions of \autoref{thm:semi_positivity_curve}. Hence the following computation concludes the proof of point \autoref{itm:semi_positive_boundary:pseff}:
\begin{equation*}
0 \leq  
\explshift{-40pt}{\deg}{\autoref{thm:semi_positivity_curve}}
 \lambda_{g, \Delta_Z}
\explshift{60pt}{=}{\autoref{prop:CM_base_change}.\autoref{itm:CM_base_change:smooth_base}}
C \cdot \lambda_{f, \Delta}.
\end{equation*}

\emph{The proof of point \autoref{itm:semi_positive_boundary:nef}:}
In this case for each finite morphism $C \to T$ from a smooth projective curve, according to \autoref{sec:base_change_relative_canonical_normal_fibers}, $f_C : (X_C, \Delta_C) \to C$ satisfy the assumptions of \autoref{thm:semi_positivity_curve}. So:
\begin{equation*}
0 
\expl{\leq}{\textrm{\autoref{thm:semi_positivity_curve}}}  
\deg \lambda_{f_C, \Delta_C}
\explshift{60pt}{=}{\autoref{prop:CM_base_change}.\autoref{itm:CM_base_change:normal_fiber}}
C \cdot \lambda_{f, \Delta}.
\end{equation*}
\end{proof}

\begin{proof}[Proof of points \autoref{itm:semi_positive_boundary:pseff} and \autoref{itm:semi_positive_boundary:nef} of \autoref{thm:semi_positive_boundary}]
These are special cases of \autoref{thm:semi_positive_boundary}. 
\end{proof}

\section{Bounding the nef threshold}
\label{sec:nefness_threshold}

\begin{lemma}
\label{lem:product_nef}
If $f : X \to T$ is a morphism between projective varieties, $m>0$ is an integer and $M$ is a $\bQ$-Cartier divisor on $X$, then $M$ is nef if and only if $M^{(m)}$ is nef.
\end{lemma}

\begin{proof}
If $M$ is nef, $M^{(m)}$ is nef by definition. For the other direction, assume that $M^{(m)}$ is nef.  Let $\iota : C \to X$ be a morphism from a smooth, projective curve. Take then the diagonal morphism $\Delta :C \to X^{(m)}$, which is defined by the equality $p_i \circ \Delta = \iota$ for each $i$. Then:
\begin{equation*}
0 \leq M^{(m)} \cdot \Delta(C) = \left( \sum_{i=1}^m p_i^* M \right) \cdot \Delta(C) = \sum_{i=1}^m ((p_i^* M) \cdot \Delta(C)) = \sum_{i=1}^m (M \cdot C) = m (M \cdot C).
\end{equation*}
Hence, $M \cdot C \geq 0$. As this works for any curve $C$ in $X$ we see that $M$ is nef.
\end{proof}

\begin{proof}[Proof of \autoref{thm:bounding_nef_threshold}]
As both the consequences and the conditions of the theorem are invariant under base-extension to another algebraically closed field, we may assume that $k$ is uncountable. In particular, as at the beginning of the proof of \autoref{thm:semi_positivity_curve}, whenever a property is true for  very general geometric fibers, it is also true for some closed fibers.

According to \autoref{lem:reducing_reduced_fibers} we may assume that all fibers of $f$ are reduced. In particular then for all integers $m>0$, $X^{(m)}$ is normal according to \autoref{lem:normal_stable_pullback_fiber_product}. Set $\lambda:=\lambda_{f, \Delta}$.

Fix the following:
\begin{enumerate}
 \item let $H$ be an ample divisor on $T$ of degree $1$, 
 \item let $t \in T$ be a closed point such that $X_t$ is normal, $X_t \subsetneq \Supp \Delta$, $(X_t, \Delta_t)$ is klt and $\delta(X_t, \Delta_t)=\delta$, 
 \item \label{itm:bounding_nef_threshold:epsilon} let $0< \varepsilon< \delta-1$ be an arbitrary rational number, and 
 \item let $0<\varepsilon' \ll \varepsilon$ be another rational number.
\end{enumerate}
 It is enough to prove that 
\begin{equation}
\label{eq:bounding_nef_threshold:goal}
 N:=\varepsilon (-K_{X/T} - \Delta) + \left(\frac{(1 + \varepsilon)\deg \lambda}{v(n+1)} + \varepsilon'\right) f^* H
\end{equation}
is nef, as we may converge with $\varepsilon$ and $\varepsilon'$ to $\delta -1 $ and to $0$, respectively. 
Set 
\begin{equation}
\label{eq:bounding_nef_threshold:M}
M:=(1 + \varepsilon) (-K_{X/T} - \Delta) + \left(\frac{(1 + \varepsilon)\deg \lambda}{ v (n+1)} + \varepsilon' \right) f^* H  = N - K_{X/T} - \Delta. 
\end{equation}
Note that
\begin{multline}
\label{eq:bounding_nef_threshold:self_intersection}
 M^{n+1} = \left((1 + \varepsilon )(-K_{X/T} - \Delta) + \left( \frac{(1 + \varepsilon)\deg \lambda}{v(n+1)} + \varepsilon'\right) f^* H \right)^{n+1} 
\\ = (1 + \varepsilon )^n  (-K_{X/T} - \Delta )^{n} \left( (1 + \varepsilon )(-K_{X/T} - \Delta ) + (n+1)\left(\frac{(1 + \varepsilon) \deg \lambda}{v(n+1)} + \varepsilon' \right) f^* H \right)
\\ = (1 + \varepsilon )^n ( \deg \lambda ( -(1 + \varepsilon ) + (1 + \varepsilon))  + (n+1)v \varepsilon' ) = 
(1 + \varepsilon )^n (n+1)\varepsilon' v   
>0.
\end{multline}
We now fix a positive integer $q$ so that the following hold:
\begin{enumerate}[resume]
\item \label{itm:Cartier} $qM$ is Cartier,
\item $q \varepsilon \in \bN$,
\item \label{itm:higher_cohomology} $R^if_*\sO_X(qM)=0$ for all $i>0$,  which is doable as $M$ is $f$-ample,
\item \label{itm:degree} $\deg\left(f_*\sO_X(qM)\right)>0$, which is doable according to \autoref{lem:Mumford_line_bundles_over_curve} and \autoref{eq:bounding_nef_threshold:self_intersection}, and 
\item \label{itm:delta} $\delta_{q'}(X_t,\Delta_t)>1+\varepsilon$, where $q':= q(1+ \varepsilon )$. This is doable according to \autoref{lem:scaling_delta_invariant} and assumption \autoref{itm:bounding_nef_threshold:epsilon}.
\end{enumerate}
From now on, let $\sE:=f_*\sO_X(qM)$. Remark that according to \cite[Lemma 3.6]{Kovacs_Patakfalvi_Projectivity_of_the_moduli_space_of_stable_log_varieties_and_subadditvity_of_log_Kodaira_dimension} for every integer $m>0$,
\begin{multline*}
\sE^{\otimes m} \cong f_*^{(m)} \sO_{X^{(m)}}\left( qM^{(m)}\right) 
\\ \cong f_*^{(m)} \sO_{X^{(m)}}\left(q' \left(-K_{X^{(m)}/T} - \Delta^{(m)} \right) +  qm\left( \frac{(1 + \varepsilon)\deg \lambda}{ v(n+1)} + \varepsilon' \right)\left(f^{(m)}\right)^* H \right)
\end{multline*}
and, by item \autoref{itm:higher_cohomology}, the following base change holds
$$
\sE^{\otimes m}_t=H^0\left(X_t^{(m)}, q'\left( -  K_{X_t^{(m)}} -\Delta^{(m)}_t \right)  \right).
$$
According to \autoref{prop:vectP2}, we may find a basis $\{e_i\}$ of $\sE_t$, an integer $m>0$, and $\ell$ global sections $s_1,\dots,s_{\ell}$ of $\sE^{\otimes m}$ so that the sections $s_j$, when restricted over $t$, are linearly independent pure tensor in the $e_i$, and furthermore
\begin{equation}\label{eq:bound}
\frac{\ell}{h^0\left(X_t^{(m)},-q'\left(K_{X_t^{(m)}}+\Delta_t^{(m)}\right)\right)}>
\underbrace{\frac{1 + \varepsilon  }{\delta_{q'}(X_t,\Delta_t)}}_{\textrm{$<1$ by assumption \autoref{itm:delta}}} \,.
\end{equation}
Define $\Gamma$ as 
$$ \Gamma:=\frac{1}{q\ell}\sum_{i=1}^{\ell}\{s_i=0\} \sim_{\bQ} M^{(m)}.$$
Note that according to \autoref{eq:bounding_nef_threshold:M},
$$K_{X^{(m)}/T} + \Delta^{(m)} + \Gamma \sim_{\bQ} N^{(m)}.$$
So, to show \autoref{eq:bounding_nef_threshold:goal}, according to \autoref{lem:product_nef} 
 it is enough to prove that $K_{X^{(m)}/T} + \Delta^{(m)} + \Gamma$ is nef, and for that according to \cite[Thm 1.13]{Fujino_Semi_positivity_theorems_for_moduli_problems} it is enough to show that $\left( X^{(m)}_t, \Delta^{(m)}_t + \Gamma_t \right)$ is klt for some (equivalently, a general) $t \in T$. For this we compute the log canonical threshold. We first remark that, since the sections $s_i$ restricted to $X_t^{(m)}$ are linearly independent pure tensors in the $e_i$, we have that
$$
 \frac{q\ell}{q' h^0\left(X_t^{(m)},-q'\left(K_{X_t^{(m)}}+\Delta^{(m)}_t\right)\right)} \Gamma_t \leq P
$$
for the $q'$-product basis type divisor  $P$  on $X_t^{(m)}$ associated to $\{e_i\}$, as in \autoref{def:prod_basis} and \autoref{rem:prod_basis}. Using \autoref{prop:prod_basis}, we obtain that $\lct \left(X^{(m)}_t,\Delta^{(m)}_t;P\right)\geq \delta_{q'}(X_t,\Delta_t)$; this yields 
\begin{multline*}
\lct\left(X_t^{(m)},\Delta_t^{(m)};\Gamma_t\right) \geq
\frac{\delta_{q'}(X_t,\Delta_t)\ell q}{h^0\left(X_t^{(m)},-q'\left(K_{X_t^{(m)}}+\Delta^{(m)}_t\right)\right)q'}
\\ \expl{=}{by the definition of $q'$ in \autoref{itm:delta}}
\frac{\ell \delta_{q'}(X_t,\Delta_t)}{h^0\left(X_t^{(m)},-q'\left(K_{X_t^{(m)}}+\Delta^{(m)}_t\right)\right)(1+ \varepsilon )}
\expl{>}{by \autoref{eq:bound}}
1
\end{multline*}

\end{proof}

\section{Positivity}
\label{sec:positivity}

\subsection{Variation}
\label{sec:variation}

\begin{definition}
\label{def:var}
Let $f : X \to T$ be a flat morphism between normal projective varieties, with $-K_{X/T}$ $\bQ$-Cartier and $f$-ample. Let $q_0$ be an integer such that $q_0K_{X/T}$ is Cartier, and for all positive integers $q_0|q$, set $\sL_q:= \sO_X(-qK_{X/T})$.   As $\sL_q$ provides a relatively ample polarization, the $\Isom$ scheme $I:=\Isom_{T \times T} ( p_1^* f , p_2^* f)$ exists together with the two natural projections $q_i : I  \to T$ \cite[1.10.2]{Kollar_Rational_curves_on_algebraic_varieties}. Let $I'$ be the image of $(q_1, q_2) : I \to T \times T$. Then, there is a non-empty open set $U \subseteq T$ where the fibers of $p_1|_{I'}: I' \to T$ have the same dimension, say $d$. This dimension is the dimension of a general isomorphism equivalence class of the fibers of $f$. As these isomorphism equivalence classes (at least general ones) would be exactly the fibers of any reasonable moduli map, one defines the \emph{variation} of $f$ as 
\begin{equation}
\label{eq:var}
 \var (f):= \dim T - d.
\end{equation}
$f$ has \emph{maximal variation}, if $\var (f)=\dim T$.  
\end{definition}

\subsection{Curve base}

\begin{notation}
\label{notation:CM_positivity}
In the situation of \autoref{notation:CM_semi_positivity}, assume that
\begin{enumerate}
 \item $\delta>1$, where $\delta= \delta\left(X_{\ot}, \Delta_{\ot}\right)$ for very general geometric points $\ot \in T$, and
 \item $\deg \lambda_{f, \Delta} =0$.
\end{enumerate}

\end{notation}

\begin{theorem}
\label{thm:no_divisor}
In the situation of \autoref{notation:CM_positivity}, for each ample $\bQ$-divisor $L$ on $T$, $|-K_{X/T} -\Delta -f^*L|_{\bQ}=\emptyset$. 
\end{theorem}

\begin{proof}
 Assume that $\Gamma \in |-K_{X/T} -\Delta -f^*L|$. Using \cite[Thm 1.2]{Fujita_Openness_results_for_uniform_K-stability}, \autoref{def:K_stable} and \autoref{prop:delta_general_fiber}, choose a small rational number $\varepsilon>0$ such that for very general geometric points $\ot \in T$ we have $\delta\left(X_{\ot}, \Delta_{\ot} + \varepsilon \Gamma_{\ot}\right)>1$. Then, 
 \begin{multline*}
 0 
\explshift{40pt}{ \geq }{by \autoref{thm:semi_positivity_curve}}
 (-K_{X/T} - \Delta - \varepsilon \Gamma)^{n+1} 
 = (-K_{X/T} - \Delta + \varepsilon (K_{X/T} + \Delta + f^*L))^{n+1} 
 = 
 (-(1-\varepsilon) (K_{X/T} + \Delta) + \varepsilon f^* L )^{n+1} 
 \\ = 
 (1-\varepsilon)^n \left( (1-\varepsilon)(- K_{X/T} - \Delta)^{n+1} + (n+1)\varepsilon(-K_{X/T} - \Delta)^n f^* L \right)
\\ 
\expl{ = }{$(- K_{X/T} - \Delta)^{n+1} =0$}
 (n+1)\varepsilon(1-\varepsilon)^n   (-K_{X_t} - \Delta_t)^n \deg L   >0.
 \end{multline*}
This is a contradiction. 
\end{proof}

\begin{notation}
\label{notation:CM_positivity_HN}
 In the situation of \autoref{notation:CM_positivity}, 
 \begin{enumerate}
  \item let $q_0>0$ be an integer such that $q_0 (-K_{X/T} - \Delta)$ is Cartier,
  \item for each integer $q_0 |q$, define $\sE_q:= f_* \sO_X(q(-K_{X/T} - \Delta))$, and set $0=\sF^0_q \subseteq \sF^1_q \subseteq \dots \subseteq \sF^{s_q-1}_q \subseteq \sF^{s_q}_q$ be the Harder-Narasimhan filtration of $\sE_q$. Set $\sG^i_q:= \sF^i_q/ \sF^{i-1}_q$,
  \item  let $g$ be the genus of $T$.
 \end{enumerate}

\end{notation}

\begin{lemma}
\label{lem:2g_slope_bound}
In the situation of \autoref{notation:CM_positivity_HN}, for every positive integer $q_0 | q$,  $\mu(\sF_q^1) \leq 2g$. 
\end{lemma}

\begin{proof}
Assume the contrary, that is, $\mu(\sF_q^1)>2g$, and let $t \in T$ be an arbitrary closed point. According to \autoref{prop:semi_stable_globally_generated}, $\sF_q^1(-t)$ is globally generated. So, there is a $\Gamma'\in |q (-K_{X/T} - \Delta) - f^* L'|$, where $L'$ is the divisor determined by $t$ on $T$. Hence, for $\Gamma := \frac{\Gamma}{q}$ and $L:= \frac{L'}{q}$ we have $\Gamma \in |-K_{X/T} - \Delta - f^* L |_{\bQ}$. This contradicts \autoref{thm:no_divisor}. 
 
\end{proof}

\begin{proposition}
\label{prop:no_positive_slopes}
In the situation of \autoref{notation:CM_positivity_HN}, for every positive integer $q_0 | q$, $\mu(\sF_q^1) \leq 0$.
\end{proposition}

\begin{proof}
Assume that $\mu\left( \sF_q^1\right) >0$, and let $\sH$ be the image of 
\begin{equation*}
\xi: \left( \sF_q^1 \right)^{\otimes m }  \to \sE_{qm} 
\end{equation*}
for some $m \gg 0$. \emph{We claim that  $\sH$ is not zero} because of the following: Let $\eta$ be the generic point of $T$. Then any $x \in \left( \sF_q^1\right)_\eta$ can be identified with some $\tx \in H^0\left( X_\eta,q \left( - K_{X_\eta} - \Delta_\eta \right) \right)$, in which case $\xi \left( x^{\otimes m}\right)$ gets identified with $\tx^m \in H^0\left( X_\eta, mq \left(- K_{X_\eta} - \Delta_\eta \right) \right)$. In particular, the following implications conclude our claim:
$
x \neq 0 \Rightarrow \tx \neq 0 \Rightarrow \tx^m \neq 0 \Rightarrow \xi\left( x^{\otimes m} \right) \neq 0$.

Let then $j$ be the smallest integer such that $\sF_{qm}^j$ contains $\sH$, and let $\sH'$ be the image of $\sH$ in $\sG_{qm}^j$. By the choice of $j$, $\sH' \neq 0$, and as $\sH'$ is a surjective image of  $ \left( \sF_q^1 \right)^{\otimes m }$:
\begin{equation*}
\mu\left(\sF_{mq}^1\right) 
\explparshift{90pt}{-85pt}{\geq }{ by the definition of the Harder-Narasimhan filtration}
\mu \left(\sG_{mq}^j\right) 
\explparshift{25pt}{-30pt}{\geq }{ $\sG_{mq}^j$ is semi- stable}
\mu(\sH') 
\explparshift{74pt}{12pt}{>}{ $ \left( \sF_q^1 \right)^{\otimes m }$ is semi-stable according to \autoref{prop:tensor_product_semi_stable}}
\mu \left( \left( \sF_q^1 \right)^{\otimes m } \right)
\explshift{25pt}{= }{\autoref{prop:tensor_product_semi_stable}}
m \mu \left(\sF_q^1\right)
\explparshift{100pt}{100pt}{>}{ $m \gg 0$, and we assumed that $\mu\left( \sF_q^1 \right)>0$}
2g . 
\end{equation*}
This contradicts \autoref{lem:2g_slope_bound}.
\end{proof}

\begin{theorem}
\label{thm:positivity_curve}
In the situation of \autoref{notation:CM_positivity}, if $q>0$ is an integer such that $-q(K_{X/T} + \Delta)$ is Cartier, then 
$f_* \sO_X(-q (K_{X/T} + \Delta))$ is a semi-stable vector bundle of slope $0$. 
\end{theorem}

\begin{proof}
First, 
\autoref{thm:bounding_nef_threshold}  yields that $-K_{X/T} - \Delta$ is nef. Then, $f_* \sO_X(  q(-K_{X/T}- \Delta))$ is also nef, by \autoref{prop:semi_positivity_engine_downstairs} taking into account  the $\bQ$-linear equivalence
\begin{equation*}
q(-K_{X/T}- \Delta) \sim_{\bQ} K_{X/T} + \Delta + (q+1)(-K_{X/T}- \Delta).
\end{equation*}
Finally,  \autoref{prop:no_positive_slopes}, concludes our proof.

\end{proof}

\subsection{Ampleness lemma}

\autoref{thm:ampleness} is an extract of the argument of the ampleness lemma of \cite{Kollar_Projectivity_of_complete_moduli} (one assumption removed in \cite{Kovacs_Patakfalvi_Projectivity_of_the_moduli_space_of_stable_log_varieties_and_subadditvity_of_log_Kodaira_dimension}). It will be one of the main technical ingredients for the proof of items \autoref{itm:semi_positive_no_boundary:big} and \autoref{itm:semi_positive_no_boundary:ample} of \autoref{thm:semi_positive_no_boundary} given in \autoref{sec:arbitrary_base_pos}. We denote by  $\Gr(w,w-q,k)$ the Grassmanian parameterizing linear subspaces of dimension $w -q$ in the $w$ dimensional $k$-vectorspace $k^{\oplus w}$. 

\begin{theorem}
\label{thm:ampleness}
Let $V$ be a vector bundle of rank $v$ on a normal  projective variety $T$ over $k$, and let $\phi : W:=\Sym^d (V) \twoheadrightarrow Q$ be a surjective homomorphism onto another vector bundle, where the ranks are $w$ and $q$, respectively. Assume that there is an open set, where the map of sets $T(k) \to \Gr(w,w-q,k)/\GL(v, k)$ induced by $\phi$ is finite to one. Then, for each ample Cartier divisor $B$ on $T$ there is an integer $m>0$ and a non-zero homomorphism
\begin{equation*}
\Sym^{qm}\left( \bigoplus_{i=1}^w W  \right) \to \sO_T(-B) \otimes ( \det Q)^m.
\end{equation*}

\end{theorem}

\begin{proof}
We explain how to turn the proof of \cite[Thm 5.5]{Kovacs_Patakfalvi_Projectivity_of_the_moduli_space_of_stable_log_varieties_and_subadditvity_of_log_Kodaira_dimension} into a proof of the above statement. 

First, specialize \cite[Thm 5.5]{Kovacs_Patakfalvi_Projectivity_of_the_moduli_space_of_stable_log_varieties_and_subadditvity_of_log_Kodaira_dimension} to the case of projective base and, thanks to \cite[Rem 5.3]{Kovacs_Patakfalvi_Projectivity_of_the_moduli_space_of_stable_log_varieties_and_subadditvity_of_log_Kodaira_dimension},  to the special choices of $W = \Sym^d (V)$ and  $G=\GL(w,w-q,k)$.  At this point the assumptions of \cite[Thm 5.5]{Kovacs_Patakfalvi_Projectivity_of_the_moduli_space_of_stable_log_varieties_and_subadditvity_of_log_Kodaira_dimension} become identical to ours, except that in \cite[Thm 5.5]{Kovacs_Patakfalvi_Projectivity_of_the_moduli_space_of_stable_log_varieties_and_subadditvity_of_log_Kodaira_dimension} was assumed to be weakly positive. 

The first, and actually main, step of the proof of \cite[Thm 5.5]{Kovacs_Patakfalvi_Projectivity_of_the_moduli_space_of_stable_log_varieties_and_subadditvity_of_log_Kodaira_dimension}  is to  construct a non-zero homomorphism as the one whose existence we have claimed. This homomorphism is displayed in \cite[Equation (5.5.5)]{Kovacs_Patakfalvi_Projectivity_of_the_moduli_space_of_stable_log_varieties_and_subadditvity_of_log_Kodaira_dimension}. In the remaining part of the proof, which is actually just the the few lines after Equation (5.5.5),  the authors use the weakly positivity assumption to deduce weakly positivity of the domain and bigness of the codomain of the morphism, and this is the only place where weakly positivity is used.

We conclude that the argument given in the proof of \cite[Thm 5.5]{Kovacs_Patakfalvi_Projectivity_of_the_moduli_space_of_stable_log_varieties_and_subadditvity_of_log_Kodaira_dimension}  to show the existence of the homomorphism displayed in \cite[Equation (5.5.5)]{Kovacs_Patakfalvi_Projectivity_of_the_moduli_space_of_stable_log_varieties_and_subadditvity_of_log_Kodaira_dimension} also proves our claim.

\end{proof}

\subsection{Arbitrary base}
\label{sec:arbitrary_base_pos}

\begin{proof}[Proof of point \autoref{itm:semi_positive_no_boundary:big} of \autoref{thm:semi_positive_no_boundary}]
As in the proofs of \autoref{thm:semi_positivity_curve} and \autoref{thm:nef_threshold}, we may assume that $k$ is uncountable. 
Let $\eta$ be the generic point of $T$. 
\begin{enumerate}
\item Set  $n:= \dim X - \dim T$, $v:=K_{X_\eta}^n$, $\delta:= \delta\left(X_{\overline{\eta}}\right)$.
 \item Fix a rational number $\alpha$ such that  $\alpha > \max\left\{ 1, \frac{\delta}{(\delta -1)v(n+1)} \right\}$. 
\end{enumerate}
Throughout the proof $\iota : C \to T$ denotes the normalization of a very general member of an arbitrary covering curve family of $T$. Very general here  means that it is not contained in  countably many divisors $S_i$, which we will specify during the proof explicitly. Set:
\begin{itemize}
 \item $\eta_C$ to be the generic point of $C$, 
 \item $Z:=X_C$ (note that as the fibers of $f$ are reduced, and the general ones are normal, $Z$ is normal),
 \item $\sigma : Z \to X$ and $g : Z \to C$ be the induced morphisms, 
\item $\lambda:=\lambda_g$.
\end{itemize}
Then the following holds:
\begin{itemize}
\item $\sigma^* K_{X/T} \cong K_{Z/C}$ by \autoref{rem:reduced_fibers_no_boundary}.\autoref{itm:relative_canonical_base_change:base_change}, and 
  $\lambda=\lambda_f|_C$ by \autoref{prop:CM_base_change}.\autoref{itm:CM_base_change:smooth_base}.
 \item a $\bQ$-Cartier divisor $L$ is pseudo-effective if and only if $L \cdot C \geq 0$ (for any such $C$), 
 \item according to \autoref{prop:delta_general_fiber}, $\delta\left(X_{\overline{\eta_C}} \right) = \delta$ (assuming we add the countably many divisors to $S_i$, over which $\delta(X_t)<\delta$, which are given by \autoref{prop:delta_general_fiber}). In particular, as $\delta>1$ the very general fibers of $g$ are uniformly $K$-stable, and hence klt, see \cite[Theorem 1.3]{Odaka_The_GIT_stability_of_polarized_varieties_via_discrepancy}.
 \item in particular, by \autoref{thm:semi_positivity_curve}, $\deg \lambda \geq 0$,
 \item by \autoref{thm:bounding_nef_threshold}, $-K_{Z/C}  + \alpha g^* \lambda$ is nef and $g$-ample.
\end{itemize}
It is important that throughout the proof all constants, so all rational numbers, will be fixed independently of the particular choice of  $C$ (for which there are two choices, first one choses a covering family, and then a very general member of that). For this reason, whenever such a constant is fixed, we do it in a numbered list item, see points above and below. 

Choose integers $r \geq 2$ and $d>0$ such that 
\begin{enumerate}[resume]
 \item $rK_{X/T}$ and $r\alpha \lambda_f$ are Cartier,
 \item $h^i(X_t, -jrK_{X_t})=0$ for all integers $i,j>0$ and all $t \in T$,
 \item $-rK_{X_t}$ is very ample for all $t \in T$, 
 \item the multiplication maps  $W:=\Sym^d f_* \sO_X (-r K_{X/T}) \to f_* \sO_X (-dr K_{X/T})=:Q$ are surjective, and
 \item \label{itm:positivity:classifying_map} for all $t \in T$, $K_{t}:=\Ker \left( \Sym^d H^0\left(X_t, -r K_{X_t}\right) \to H^0\left(X_t,-d r K_{X_t}\right) \right)$ generates $\sI(d)$, where $\sI$ is the ideal of $X_t$ via the embedding $\varphi_{\left|-rK_{X_t} \right|} : X_t \to \bP^{v-1}$, where $v:= \rk f_* \sO_X (-r K_{X/T})$ and $\varphi_{\left|-rK_{X_t} \right|}$ is defined only up to the action of $\GL(v,k)$ on the target. Note that this is achievable because $\sI$ form a flat family as $t$ varies.  
 
 In particular, if we set  $w:= \rk W$ and $q:= \rk Q$, then for every $t \in T(k)$, $K_{t} \subseteq W_t$ determines $X_t \hookrightarrow \bP^{v-1}$ up to the action of $\GL(v, k)$, which then means that the orbit of $K_{t}$ in $\Gr(w,q)/\GL(v,k)$ determines $X_t$ up to isomorphism. Therefore if we apply \autoref{thm:ampleness} for $W \to Q$, then the fibers of the classifying map $T(k) \to \Gr(w,q)/\GL(v,k)$ are contained in the isomorphism classes of the fibers of $f$ and hence, by the maximal variation assumption, there is an open set where these fibers are finite. 
\end{enumerate}
As,
\begin{equation*}
r (- K_{Z/C}  +  2\alpha g^* \lambda)  = K_{Z/C}  +  \underbrace{ (r+1)( -K_{Z/C} +  \alpha g^* \lambda) +  \underbrace{(r-1) \alpha}_{>1} g^* \lambda}_{\textrm{nef and $g$-ample}},
\end{equation*}
by \autoref{prop:semi_positivity_engine_downstairs}, $g_* \sO_Z(r (- K_{Z/C}  +  2\alpha g^* \lambda)) $ is a nef vector bundle. Set 
\begin{equation*}
M:= r( -K_{X/T} + 2 \alpha f^* \lambda_f). 
\end{equation*}
Note that the conclusions of point \autoref{itm:positivity:classifying_map} about the finiteness of the classifying map hold also for $-rK_X$ replaced by  $M$, as $f_* \sO_X(M)$ and $f_* \sO_X(dM)$ differs from $f_* \sO_X(-rK_{X/T})$ and $f_* \sO_X(-rdK_{X/T})$ only by a twist with $r2 \alpha \lambda_f$ and $dr2 \alpha \lambda_f$, respectively. So,
\autoref{thm:ampleness} yields an ample divisor $B$ on $T$,  an integer $m>0$ and  a non-zero homomorphism as follows (see point \autoref{itm:positivity:classifying_map} above for the definition of $w$ and $q$):
\begin{equation*}
\xi : \Sym^{qm}\left( \bigoplus_{i=1}^w \Sym^d (f_* \sO_X(M)) \right) \to \sO_X(-B) \otimes \left(\det f_* \sO_X(dM) \right)^m.
\end{equation*}
As the target of $\xi$ is a line bundle, there exists a divisor, on the complement of which $\xi$ is surjective. Let us add this divisor to $S_i$. Then $\xi|_C$ is a non-zero homomorphism as follows:
\begin{equation*}
\xi_C : \Sym^{qm}\left( \bigoplus_{i=1}^w \Sym^d( g_* \sO_Z(M_C)) \right) \to \sO_C(-B_C) \otimes \left( \det g_* \sO_Z(dM_C) \right)^m. 
\end{equation*}
Define 
\begin{equation*}
\sA:=\det f_* \sO_X(dr ( -K_{X/T}  + 2 \alpha f^* \lambda_f))= \det f_* \sO_X(dM), 
\end{equation*}
and let $A$ be a divisor corresponding to $\sA$.
As $g_* \sO_Z(M)$ is nef and hence so is every bundle that admits a generically surjective map from the left side of $\xi_C$, we obtain that 
\begin{equation}
\label{eq:positivity:B_A}
\deg \sA|_C = \deg  \det g_* \sO_Z(dM_C)) \geq \frac{B \cdot C}{m} . 
\end{equation}
Consider now, the natural  embedding:
\begin{equation*}
\alpha : \det f_* \sO_X(dM) \hookrightarrow \bigotimes_{i=1}^q f_* \sO_X(dM)  \cong f^{(q)}_* \sO_{X^{(q)}}\left(dM^{(q)}\right),
\end{equation*}
given by the embedding of representations $\det \to \bigotimes_{i=1}^q  $ of $\GL(q,k)$. 
Hence, by adjunction of $f_*^{(q)}(\_)$ and $\left(f^{(q)}\right)^*(\_)$ one can write $\left(f^{(q)} \right)^* A + D = dM^{(q)}$, where $D$ is an effective divisior on $X^{(q)}$. We claim that, since $\mathcal{O}_X(dM)$ is compatible with base change, $D$ does not contain any fiber. Indeed, assuming by contradiction that $D$ contains a fiber $X_t$, we obtain a basis $s_i$ of $H^0(X_t,\mathcal{O}_X(dM)|_{X_t})$ whose determinant vanishes in $\bigotimes^qH^0(X_t,\mathcal{O}_X(dM)|_{X_t})$, and this is a contradiction as the $s_i$ are linearly independent.

By the continuity of log canonical threshold, there is a $0<\varepsilon < \frac{1}{rd}$ such that $\left(X_t^{(q)}, \varepsilon D_t\right)$ is klt for general closed points $t \in T$. In particular by the genericity of $C$ the same holds also for general $t \in C$.	 Then, if we define $N:= d r( -K_{X/T} + 3 \alpha f^* \lambda_f)$, 
 according to \autoref{cor:semi_positivity_engine_upstairs}, the following divisor is nef ($Z^{(q)}$ is normal by \autoref{lem:normal_stable_pullback_fiber_product}.\autoref{itm:normal_stable_pullback_fiber_product:fiber_product}). 
\begin{multline*}
 \underbrace{K_{Z^{(q)}/C}  + \varepsilon D_C }_{\parbox{80pt}{\tiny $\left(Z^{(q)}_t,  \varepsilon \left(D_C\right)_t \right)$ is klt for $t \in C$ general}}
+
 \underbrace{\left(dr + 1 - \varepsilon rd \right)\left(   -K_{Z^{(q)}/C}  + 2 \alpha q \left(g^{(q)}\right)^* \lambda \right) +  (dr-2) \alpha q \left(g^{(q)}\right)^* \lambda}_{\textrm{nef and $f$-ample  $(r \geq 2, d>0)$}} 
 \\ \sim \left(  N^{(q)}_C - \varepsilon \left(g^{(q)} \right)^* A_C \right) = \left(N_C - g^* \frac{\varepsilon }{q} A_C \right)^{(q)}
\end{multline*}
Set $\varepsilon':= \frac{\varepsilon}{qr d}$. Then we have that $\frac{N_C}{dr} - \varepsilon' g^*A_C$ is nef according to \autoref{lem:product_nef}.
So,
\begin{multline*}
0 \leq  (-K_{Z/C}  + 3\alpha g^* \lambda - \varepsilon'  g^* A_C)^{n+1} 
 =
 \left(\underbrace{-K_{Z/C}  +  \frac{g^*\lambda}{v(n+1)}}_{\parbox{90pt}{\tiny top self-intersection is $0$ by the definition of $\lambda$}}  + \left(3\alpha - \frac{1}{v(n+1)} \right) g^* \lambda -  \varepsilon' g^*A_C \right)^{n+1} 
\\ =
(n+1) v \deg \left( \left(3\alpha - \frac{1}{v(n+1)} \right) \lambda -  \varepsilon' A_C \right)
\expl{\leq}{equation \autoref{eq:positivity:B_A}}
(n+1) v \deg \left( \left(3\alpha - \frac{1}{v(n+1)} \right) \lambda -  \varepsilon' \frac{B_C}{m} \right)
\end{multline*}
Hence, $ \left(3\alpha - \frac{1}{v(n+1)} \right) \lambda_f -  \varepsilon' \frac{ B}{m}$ is pseudo-effective (as it dots to at least zero with each movable class). Therefore, $\lambda_f$ is the sum of an ample and a pseudo-effective $\bQ$-divisor, so  $\lambda_f$ is big.
 
\end{proof}

\begin{proof}[Proof of point \autoref{itm:semi_positive_no_boundary:ample} of \autoref{thm:semi_positive_no_boundary}]
By Nakai-Moishezon it is enough to prove that for all normal varieties $V$ mapping finitely to $X$, $\left( \lambda_f|_V \right)^{\dim V} >0$. However, using \autoref{prop:CM_base_change}, this we may obtain by replacing $f : X \to T$ with $f_V : X \times_T V \to V$, and applying point  \autoref{itm:semi_positive_no_boundary:big} to $f_V$.
\end{proof}

\begin{proof}[Proof of point \autoref{itm:semi_positive_no_boundary:q_proj} of \autoref{thm:semi_positive_no_boundary}]

Let $\lambda$ be the CM line bundle on $T$. According to \cite[Thm 6.1]{Li_Wang_Xu_Qasi_projectivity_of_the_moduli_space_of_smooth_Kahler_Einstein_Fano_manifolds}, it is enough to prove that for every irreducible closed subspace $Z$ of $T$ we have $(\lambda|_Z)^{\dim Z}\geq 0$, with strict inequality if $Z$ intersects $U$.

The algebraic space $Z$ has a finite cover $\pi\colon V\to Z$ by a scheme \cite[Tag 04V1]{stacks-project}, and by Nagata's theorem and resolution of singularities we may also assume that $V$ is projective and smooth. To prove $(\lambda|_Z)^{\dim Z}\geq 0$, it is enough to show $(\pi^* \lambda|_Z)^{\dim V}\geq 0$. As $V$ is smooth and projective, this is equivalent to show that $\pi^* \lambda|_Z$ is nef. Since $\lambda$ is compatible with base-change (\autoref{prop:CM_base_change} \autoref{itm:CM_base_change:normal_fiber}), this follows from our semipositivity result \autoref{thm:semi_positive_boundary} \autoref{itm:semi_positive_boundary:nef}.

To prove the strict inequality, as $\pi^*\lambda|_Z$ is nef, we have to show that $\pi^*\lambda|_Z$ is big. This follows from our positivity result \autoref{thm:semi_positive_no_boundary} \autoref{itm:semi_positive_no_boundary:big}, remarking that, as the isomorphism class of the family $f$ are finite, the family $f_Z$ is of maximal variation, and a finite cover does not affect this maximality.

\end{proof}

\begin{proof}[Proof of \autoref{thm:lambda_non_big}]
Choose $q$ big enough such that $-q(K_{X/T} + \Delta)$ is Cartier and without higher cohomology on the fibers. 
Let $H_i \in |H|$ be general for $i=1,\dots, \dim T -1$, and set $C:= \bigcap_{i=1}^{\dim T -1} H_i$. By the above generic choices, $Z:=X_C$ is normal. Furthermore, $C$ lies in the smooth locus of $T$, hence for base-change properties along $C \to T$ we may assume that $T$ is smooth. In particular,  there is an induced boundary $\Delta_Z$ on $Z$ (\autoref{sec:base_change_relative_canonical_smooth_base}), for which $K_{X/T} + \Delta|_Z = K_{Z/C} + \Delta_Z$  (\autoref{prop:relative_canonical_base_change_normal}), and consequently
\begin{equation}
\label{eq:lambda_non_big:base_change}
f_* \sO_X(-q(K_{X/T} + \Delta))|_C \cong \left(f_C\right)_* \sO_{Z}\left(-q\left(K_{Z/C} + \Delta_Z \right) \right).
\end{equation}
Furthermore,
\begin{equation*}
0 
\expl{= }{assumption}
\lambda_{f,\Delta} \cdot H^{\dim T -1}
= \deg \lambda_{f,\Delta}|_C 
\expl{= }{\autoref{prop:CM_base_change}.\autoref{itm:CM_base_change:smooth_base}}
\deg \lambda_{f_C, \Delta_Z} .
\end{equation*}
Therefore, according to \autoref{thm:positivity_curve},  $\left(f_C\right)_* \sO_{Z}\left(-q\left(K_{Z/C} + \Delta_Z \right) \right)$ is a semi-stable vector bundle of slope $0$. However, then the isomorphism \autoref{eq:lambda_non_big:base_change} implies that $f_* \sO_X (- q(K_{X/T} + \Delta))$ is $H$-semi-stable of slope $0$: if it had a subsheaf $\sF$ of $H$-slope bigger than $0$, then for the saturation $\sF'$ of $\sF$, $\sF'|_C$ would be a subbundle of positive degree of $\left(f_C\right)_* \sO_{Z}\left(-q\left(K_{Z/C} + \Delta_Z \right) \right)$, which is a contradiction.  
 \end{proof}
 
\begin{proof}[Proof of \autoref{cor:lambda_trivial_fiber_bundle}]
First, assume that $f$ is analytically locally a fiber bundle. Then, all fibers are isomorphic, and hence uniformly $K$-stable. In particular, there is an induced moduli map $T \to \sM_{n,v}^{\uKs}$. As $\sM_{n,v}^{\uKs}$ is a D-M stack, it has a finite cover $S \to \sM_{n,v}^{\uKs}$ bya  scheme. Let $T'$ be the normalization of a component of $T \times_{\sM_{n,v}^{\uKs}} S$ that dominates $T$. As $T' \to \sM_{n,v}^{\uKs}$ factors through $S$, the family $f \times_{T} T'$ corresponding to $T' \to \sM_{n,v}^{\uKs}$ is a trivial family. In particular, $\deg \lambda_{f \times_{T} T'}=0$. However, by \autoref{prop:CM_base_change}.\autoref{itm:CM_base_change:smooth_base}, $\deg \lambda_{f \times_{T} T'}= (\deg f)\cdot (\deg \lambda_f)$. In particular, we conclude that $\deg \lambda_f=0$.

Second, assume that $\deg \lambda_f=0$.  By \autoref{thm:lambda_non_big}, the vector bundles $\mathcal{E}_q:=f_*\mathcal{O}_X(-qK_{X/T})$ are semistable of slope zero for all $q$ divisible enough. Now, we have to use the language of Higgs bundle to use a result from \cite{Simpson_Higgs_bundles_and_local_systems}. Consider the functor from the category of semistable bundles of slopes zero on $T$ to the category of semistable  Higgs bundle on $T$ with $c_1=0$ associating to $\mathcal{E}$ the pair $(\mathcal{E},0)$, where  $0$ be the zero homomorphism $\mathcal{E} \to \mathcal{E} \otimes \Omega^1_X$. This functor is fully faithfull. By \cite[Corollary 3.10]{Simpson_Higgs_bundles_and_local_systems}, its codomain is equivalent to the category of local systems. Moreover, by the remark at the end of subsection ``Examples' 'of  \cite[Section 3]{Simpson_Higgs_bundles_and_local_systems}, the local system associated to a semistable Higgs bundle $\mathcal{E}$ is isomorphic to $\mathcal{E}$ as holomorphic vector bundle.

We conclude that the multiplication map
$$
m\colon \Sym^r\mathcal{E}_q\to \mathcal{E}_{qr}
$$
is a actually a morphism of local systems, and its kernel $\mathcal{K}_{q,r}$ is a local system too. Taking $r=2$, and $q$ big enough such $-qK_{X/T}$ is very ample and the ideals defining the fibers  in $\mathbb{P}\mathcal{E}_q$ are defined by quadrics, we conclude that the fibration is locally trivial in the analytic topology.
\end{proof}

\begin{proof}[Proof of \autoref{cor:proper_base}]
The proof is very similar to that of  point \autoref{itm:semi_positive_no_boundary:q_proj} of \autoref{thm:semi_positive_no_boundary} above.  
As in the above proof, $T$ has a generically finite cover by a smooth, projective scheme. By base-changing over this cover one may assume that the base is smooth and projective. By \autoref{prop:2_defs_CM_same}, we may replace $N$ by the CM-line bundle notion used in the present article, see \autoref{def:CM}, and then point \autoref{itm:semi_positive_boundary:nef} of \autoref{thm:semi_positive_boundary} and point  \autoref{itm:semi_positive_no_boundary:big} of \autoref{thm:semi_positive_no_boundary} concludes the proof. 

\end{proof}

\section{Proof of the main theorem}
\label{sec:main_theorem}

For the precise definition of the functor of $\sM_{n,v}^{\Kss}$ we refer to  \cite{Alper_Blum_Halpern-Leistner_Xu_Reductivity_of_the_automorphism_group_of_K-polystable_Fano_varieties,Blum_Xu_Uniqueness_of_K-polystable_degenerations_of_Fano_varieties}. For the present article, the  important facts are the following:
\begin{itemize}
\item According to \cite[Thm 1.3]{Blum_Liu_Xu_Openness_of_K-semistability_for_Fano_varieties} and \cite[Thm 1.5]{Xu_A_minimizing_valuation_is_quasi-monomial}, $\sM_{n,v}^{\Kss}$ is a separated Artin stack of finite type over $k$. 
\item $\sM_{n,v}^{\Kss}$ admits a separated good moduli space $M_{n,v}^{\Kps}$, the $k$-points of which parametrize $K$-polystable Fano varieties of dimension $n$ and volume $v$ over $k$. We note that $M_{n,v}^{\Kps}$ is only known to be an algebraic space at this point, as opposed to a scheme. 
\item Given a flat morphism $f : X \to T$ between normal, projective varieties with normal klt fibers and $-K_{X/T}$ being $\bQ$-Cartier and ample, there is an induced moduli map $\nu : T \to \sM_{n,v}^{\Kss}$. That is, the Koll\'ar condition in the definition of $\sM_{n,v}^{\Kss}$ is automatically satisfied for such families \cite[thm 3.68]{Kollar_Families_of_varieties_of_general_type}.
\end{itemize}
We start with  \autoref{lem:lifting_to_Artin_stack}. In the proof of \autoref{thm:main}, where we apply \autoref{lem:lifting_to_Artin_stack}, we want to show the descent of the CM line bundle has positive self-intersection over a proper subspace $V$ of $M_{n,v}^{\Kps}$. Hence, we want to construct a generically finite cover of $V$ that supports a universal family. Because $\sM_{n,v}^{\Kss}$ is an Artin stack, this is not possible. However, we can cook up one cover ($T \to V$ in the lemma) that supports a universal family over a big open set  of $V$, and this universal family extends on  the whole cover to a family $f : X \to T$ to  which our theorems apply. That is, $f$ is flat, $X$ is normal and klt and $-K_{X/T}$ is $\bQ$-Cartier. It is a delicate task to find such an extension, so the second part of the proof of \autoref{lem:lifting_to_Artin_stack} is dedicated to this. The rough idea is to find a flat $f : X \to T$, such that
\begin{enumerate}
\item \label{itm:Artin_expansion_explanation:trivial}  $K_X + \Delta \sim_{\bQ} 0$, and 
\item \label{itm:Artin_expansion_explanation:KSBA} $(X , (1+ \varepsilon') \Delta)$ is a  KSBA stable family.
\end{enumerate}
Indeed, in this situation $-K_X$ is up to a scaling $\bQ$-linearly equivalent to $K_X + (1+ \varepsilon') \Delta$, which is ample over $T$. Although, invoking the KSBA moduli space, guaranteeing condition \autoref{itm:Artin_expansion_explanation:KSBA} is quite straightforward,  guaranteeing condition \autoref{itm:Artin_expansion_explanation:trivial} is much harder. Hence, in the finishing part of \autoref{lem:lifting_to_Artin_stack}, we need to invoke passing to  a $\bQ$-factorialization, and running an adequate MMP. Additionally, in each step of this process we will need to show that flatness is preserved. 

\begin{lemma}
\label{lem:lifting_to_Artin_stack}
Let $V  \subseteq M_{n,v}^{\Kps}$ be proper closed subspace. Then there is a diagram as follows
\begin{equation}
\label{eq:lifting_to_Artin_stack:statement}
\xymatrix@C=120pt@R=0pt{
\explabove{T}{\footnotesize smooth, projective variety} \ar[d]^{\phi}_{\textrm{birational}} \\
\explleft{S}{proper, normal variety} \ar[dd]^{\xi}_{\textrm{finite}}  &
 \explaboveshiftslantleft{100pt}{T^0}{big open set in $S$} \ar@{_(->}[l]_{\iota}^{\textrm{open}} \ar@{_(->}[ul]^j_{\textrm{open}} \ar[d]_{\xi_0}^{\textrm{finite}} \ar[r] & \sM_{n,v}^{\Kss} \ar[dd] \\
 & \tV \ar@{_(->}[dl]_{\textrm{big open set}} \ar@{^(->}[dr] &  \\
\explleft{V}{proper algebraic space} \ar@{_(->}[rr]_{\textrm{closed}} & & M_{n,v}^{\Kps}
}
\end{equation}
where the family induced by $T^0 \to \sM_{n,v}^{\Kss}$ extends to $f : X \to T$ such that 
\begin{enumerate}
\item $X$ is normal and klt,
\item $f$ is flat,
\item the fibers of $f$ are reduced, and 
\item $-K_{X/T}$ is $\bQ$-Cartier and $f$-ample.
\end{enumerate}

\end{lemma}

\begin{proof}
Set $\sX:= V \times_{M_{n,v}^{\Kps}}\sM_{n,v}^{\Kss}$. According to \cite[p 2351, Main Properties (3)]{Alper_Good_moduli_spaces_for_Artin_stacks}, $\sX \to V$ is a good moduli space as well. Additionally, as $V$ is proper, by \cite[Thm A]{Alper_Halpern-Leistner_Heinloth_Existence_of_moduli_spaces_for_algebraic_stacks}, $\sX$ satisfies the existence part of the valuative criterion of properness.

As $\sM_{n,v}^{\Kss}$ is an Artin stack, there is a smooth, surjective morphsim $Z_{\pre} \to \sM_{n,v}^{\Kss}$ from a scheme.
Set $Z'$ to be the normalization of a component of a general complete intersection of an affine chart of $ Z_{\pre} \times_{M_{n,v}^{\Kps}} V$ of that dimension such that $Z' \to V$ is dominant and generically finite. 

Let $Y \to M_{n,v}^{\Kps}$ be a finite cover by a scheme, which exists by \cite[Tag 04V1]{stacks-project}. Replacing $Z'$ by a component of the normalization of $Y \times_{M_{n,v}^{\Kps}} Z'$ dominating $V$ we may assume that $Z'$ factors also though $Y$. That is, we have a commutative diagram as follows, where $Z' \to V$ is dominant and generically finite:
\begin{equation*}
\xymatrix{
\explleft{Z'}{normal variety} \ar@/^7pc/[drr]^{\textrm{generically finite and dominant}} \ar[dd] \ar[r] &  \explabove{Y \times_{M_{n,v}^{\Kps}} V}{scheme} \ar[dr]^{\textrm{finite}} \ar@{^(->}[d]_{\textrm{closed}} \\
 & Y \ar[dr] & V \ar@{^(->}[d]^{\textrm{closed}} \\
\sM_{n,v}^{\Kss} \ar[rr] & & M_{n,v}^{\Kps}
}
\end{equation*}
As $Y \times_{M_{n,v}^{\Kps}} V$ is a scheme, finite over $V$, it is proper. Define $Z$ first to be the normalization in the functionfield of $Z'$ of a reduced structure of a component of $Y \times_{M_{n,v}^{\Kps}} V$ dominating $V$. Then, $Z$ is  a normal proper variety with a rational map  $Z \dashrightarrow \sM_{n,v}^{\Kss}$, such that the composition $Z \dashrightarrow M_{n,v}^{\Kps}$ is a finite morphism with image being $V$. With other words, we have a commutative diagram as follows:
\begin{equation*}
\xymatrix{
\sM_{n,v}^{\Kss} \ar[d] & \explright{Z}{normal, proper variety} \ar@{-->}[l]  \ar[d]^{\textrm{finite, surjective}} \\
M_{n,v}^{\Kss} & V \ar@{_(->}[l] 
}
\end{equation*}
 As $\sX$ satisfies the existence part of the valuative criterion for properness, after replacing $Z$ by a finite cover we may assume that $Z \dashrightarrow  \sM_{n,v}^{\Kss}$ is a morphism in codimension $1$. Hence, there is a big regular open set $\tZ \subseteq Z$ and a family $\tf : \tX   \to \tZ$ of $K$-semi-stable Fano varieties. In particular, $\tf$ is flat, $\tX$ is normal, $-K_{\tX}$ is $\bQ$-Cartier and ample over $\tZ$, and $\tX$ is klt. We may also assume that $\tZ$ maps finitely to an open set $\tV$ of $V$.

\emph{We claim that by possibly shrinking $\tZ$, but still keeping it big in $Z$, we may find a $\bQ$-divisor  $\tilde{\Delta}$ on $\tX$ such that $K_{\tX} + \tilde{\Delta} \sim_{\bQ,\tZ} 0$ and $\left(\tX_z, \tilde{\Delta}_z\right)$ is klt for every $z \in \tZ$}. For this choose an integer $m>0$ and an ample Cartier divisor $H$ on $\tZ$ such that
$-mK_{\tX/\tZ} + \tf^* H$ is very ample, and that
\begin{equation}
\label{eq:lifting_to_Artin_stack:choosing_m}
\frac{2}{m} < \min\left\{1, \min\left\{ \left. \alpha\left(\tX_{\oz} \right) \right|  \oz \in \tZ\textrm{ is a geometric point }\right\} \right\},
\end{equation}
 which minima exist by \cite[Thm 1.1]{Blum_Liu_Xu_Openness_of_K-semistability_for_Fano_varieties}.
Choose now  general elements $\Gamma^i$ of 
$\left|-mK_{\tX/\tZ} + \tf^* H\right|$ for $i=1,\dots,m$. We show our claim by choosing  $\tilde{\Delta}:= \sum_{i=1}^m \frac{\Gamma^i}{m^2} \sim_{\bQ, \tZ} -K_{\tX}$. To show that this is a good choice, it is enough to show that for every codimension $1$ point $\xi$ of $\tZ$, $\left( \tX_\xi, \tilde{\Delta}_\xi \right)$ is klt. By the genericity assumption on $\Gamma^i$, using  Bertini on the general fiber, there are only finitely many codimension one points of $\tZ$ over which $\left(\tX, \Gamma^i\right)$ is not lc. Using the genericity assumption again, we can also assume that these codimension one points are different for different values of $i$. Hence, for  a fixed codimension $1$ point $\xi$, there is at most one  index, say $j$, such that $\left( \tX_{\xi}, \Gamma^i_{\xi} \right)$ is not lc. If there is no such index, set $j$ to be a random one. Then, we may write
\begin{equation}
\label{eq:lifting_to_Artin_stack:subconvex_combination}
 \tilde{\Delta}_{\xi}= \sum_{i=1}^m \frac{\Gamma^i_{\xi}}{m^2} =   
\frac{1}{2} \frac{2\Gamma^j_\xi}{m^2} 
 + 
\sum_{i \neq j} \frac{1}{m^2} \Gamma^i_\xi.
\end{equation}
Note now that 
$\left(\tX_\xi,\frac{2\Gamma^j_{\xi}}{m^2} \right)$ is klt by \autoref{eq:lifting_to_Artin_stack:choosing_m}, and that $\left( \tX_\xi, \Gamma^i_{\xi} \right)$ is lc for all $j \neq i$ by the choice of $j$.  Then, the fact implied by \autoref{eq:lifting_to_Artin_stack:choosing_m} that $\frac{1}{2} + \frac{m-1}{m^2} < 1$, together with equation \autoref{eq:lifting_to_Artin_stack:subconvex_combination} yields  that $\left(\tX_\xi, \tilde{\Delta}_\xi \right)$ is klt. Hence we have showed our claim.

Note that our claim above also implies that for every $0< \varepsilon \ll 1$, $\left(\tX, (1+ \varepsilon)\tilde{\Delta} \right)$ is a log canonical model over $\tZ$.  Additionally, by the claim,  we may consider the moduli map $\phi : \tZ \to \sM^{\KSBA}$ induced by $\left(\tX, (1+\varepsilon)\tilde{\Delta} \right)$ to the moduli space of stable log-varieties, for some $0 < \varepsilon \ll 1$. As  $\sM^{\KSBA}$ is a proper DM stack with projective coarse moduli space, there is a finite surjection $W \to \sM^{\KSBA}$ from a projective scheme. Then, by taking the main component of the normalization of $W \times_{\sM^{\KSBA}} \tZ$ we obtain a normal variety $T^0$ fitting in a commutative diagram
\begin{equation*}
\xymatrix{
T^0 \ar[r] \ar[d]_{\textrm{finite, surjection}}  & \explright{W}{projective scheme} \ar[d] \\
\tZ \ar[r] & \sM^{\KSBA}
}
\end{equation*} 
By replacing both $\tZ$ and $T^0$ by one of their big open sets, we may assume that $T^0$ is also regular. Then, by compactifying $T^0$ resolving the indetermnancies of the map from this compaticifcation to both $Z$ and $W$ and finally also resolving the compatification itself we obtain a smooth compatification $T \supseteq T^0$:
\begin{equation*}
\xymatrix@C=55pt{
\explleft{T}{smooth, projective variety} \ar@/^2pc/[rr] \ar[d]_{\textrm{proper \& generically finite}} &   \ar@{_(->}[l]^{\iota}_{\textrm{open}} T^0 \ar[r] \ar[d]^{\textrm{finite}}  & W \ar[d] \\
Z \ar[d]_{\textrm{finite}} & \ar@{_(->}[l]_{\textrm{big open set}} \tZ \ar[r] \ar[d]^{\textrm{finite}} & \sM^{\KSBA} \\
V & \ar@{_(->}[l]_{\textrm{big open set}} \tV
}
\end{equation*} 
In particular, $T \to \sM^{\KSBA}$ induces $f_{\pre} : \left(X_{\pre}, (1+ \varepsilon) \Delta_{\pre}\right) \to T)$ such that $\left(X_{\pre}, \Delta_{\pre} \right)|_{T^0} = \left( \tX, \tilde{\Delta} \right) \times_{\tZ} T^0$. Additionally $\left(X_{\pre}, (1+ \varepsilon) \Delta_{\pre}\right)$ is klt as it is a family of stable log-varieties with klt general fiber. 

Note that $f_{\pre}: X_{\pre} \to T$ is flat. We are going to take $\bQ$-factorialization of $X_{\pre}$ and then we will run a particular MMP and finally we will take a particular canonical model.  The point is that all these operations preserve flatness. The  reasons is that at each step the statement corresponding to the following one holds, where $d:= \dim T$:
\begin{multline}
\label{eq:lifting_to_Artin_stack:log_canonical}
\textrm{For each closed }t \in T \textrm{ and general hyperplanes }H_1,\dots,H_d \textrm{ through t}:\\ \left(X_{\pre}, (1+ \varepsilon) \Delta_{\pre} + \sum_{i=1}^d f_{\pre}^* H_i \right) \textrm{ is log canonical.}
\end{multline}
Then, equidimensionality of $f_{\pre}$ follows from \cite[Prop 34]{de_Fernex_Kollar_Xu_The_dual_complex_of_singularities}, as $X_t= \bigcap_{i=1}^d H_i$ is a union of lc-centers of codimension $d$, and \cite[Prop 34]{de_Fernex_Kollar_Xu_The_dual_complex_of_singularities} states that an lc-center is contained in the intersection of $d$ $\bQ$-Cartier divisors of coefficient $1$ from the boundary (here the $H_i$), then the codimension of the lc-center has to be at least $d$. As $\left(X_{\pre}, \Delta_{\pre}\right)$ is klt, $X_{\pre}$ is Cohen-Macayulay, an hence by the above shown equidimensionality, $f_{\pre}$ is indeed flat. Additionally, \cite[Prop 34]{de_Fernex_Kollar_Xu_The_dual_complex_of_singularities} tells us that locally there is a finite cover where the pullbacks of the $H_i$ become simple normal crossing. This shows that $\bigcap_{i=1}^d H_i$ is reduced, otherwise the intersection of the above pullbacks would be non-reduced. Hence, we obtain that the fibers of $f_{\pre}$ are reduced. 

By doing a $\bQ$-factorialization we obtain a $\bQ$-factorial model $f' : (X', (1+ \varepsilon) \Delta') \to T$ with a proper, small birational morphism $X' \to X$. Hence  the strict transform from $X$ to $X'$ of every $\bQ$-Cartier divisor is crepant. In particular,
\begin{enumerate}
\item $X'$ is $\bQ$-factorial
\item \label{itm:lifting_to_Artin_stack:X_prime_klt} $(X',  (1+ \varepsilon)\Delta')$ is klt, 
\item  \label{itm:lifting_to_Artin_stack:flat} the condition correpsonding to \autoref{eq:lifting_to_Artin_stack:log_canonical} is satisfied for $f_{\pre}$, $X_{\pre}$ and $\Delta_{\pre}$ replaced by $f'$, $X'$ and $\Delta'$, respectively; hence $f'$ is flat and has reduced fibers,
\item \label{itm:lifting_to_Artin_stack:X_prime_big_and_nef} $K_{X'} + (1+\varepsilon) \Delta'$ is only big and nef over $T$, and 
\item \label{itm:lifting_to_Artin_stack:X_prime_num_triv_generically} over $T^0$, we have $K_{X'} + \Delta' \sim_{\bQ,T} 0$. 
\end{enumerate}
Hence, by points \autoref{itm:lifting_to_Artin_stack:X_prime_big_and_nef} and \autoref{itm:lifting_to_Artin_stack:X_prime_num_triv_generically}, $\Delta'$ is  big over $T$, and therefore we may run an $(X', \Delta')$ MMP \cite[Thm 1.2]{Birkar_Cascini_Hacon_McKernan_Existence_of_minimal_models}. As we have already a minimal model over $T^0$, this MMP is an isomorphism over $T^0$. Let $f_{\min} : \left(X_{\min}, \Delta_{\min} \right) \to T$ be the outcome of this MMP. Hence:
\begin{enumerate}[resume]
\item \label{itm:lifting_to_Artin_stack:klt} $\left( X_{\min}, \Delta_{\min}\right)$ is klt as we are running an MMP on $(X', \Delta')$, which is klt by point \autoref{itm:lifting_to_Artin_stack:X_prime_klt},
\item \label{itm:lifting_to_Artin_stack:X_min_flat} with notation as in \autoref{eq:lifting_to_Artin_stack:log_canonical}: using point \autoref{itm:lifting_to_Artin_stack:flat} and the fact that our MMP is also an an MMP for $\left(X',\Delta' + \sum_{i=1}^d \left(f'\right)^* H_i \right)$, we obtain that the condition correpsonding to \autoref{eq:lifting_to_Artin_stack:log_canonical} is satisfied for $f_{\pre}$, $X_{\pre}$, $\Delta_{\pre}$ and $(1+\varepsilon)$ replaced by $f_{\min}$, $X_{\min}$, $\Delta_{\min}$ and $1$, respectively; in particular, $f_{\min}$ is flat and has reduced fibers.
\item \label{itm:lifting_to_Artin_stack:X_min_num_triv_globally} As over $T^0$ we have  $K_{X_{\min}} + \Delta_{\min} \sim_{\bQ,T} 0$, and as  $K_{X_{\min}} + \Delta_{\min}$ is semi-ample over $T$ by  \cite[Thm 1.1]{Hacon_Xu_Existence_of_log_canonical_closures}, we obtain using the Rigidity lemma \cite[Lem 1.6]{Kollar_Mori_Birational_geometry_of_algebraic_varieties} that $K_{X_{\min}} + \Delta_{\min} \sim_{\bQ,T} 0$ holds over the entire $T$.
\end{enumerate} 
Now, we pass to the log canonical model $f : (X, (1+ \varepsilon')\Delta) \to T$ of $\left(X_{\min}, (1+ \varepsilon')\Delta_{\min}\right)$ over $T$ for some $0<\varepsilon' \ll \varepsilon$. Note that the latter pair is klt by point \autoref{itm:lifting_to_Artin_stack:klt}. We have $\left(X_{\min}, (1+ \varepsilon')\Delta_{\min}\right)|_{T^0} = \left(X', (1+ \varepsilon')\Delta'\right)|_{T^0}$. Hence, over $T^0$, $(X, (1+ \varepsilon')\Delta)$ is the log canonical model of $\left(X', (1+ \varepsilon')\Delta'\right)|_{T^0}$, that is, it agrees over $T^0$ with $\left(X_{\pre}, (1+ \varepsilon') \Delta_{\pre}\right)$. Hence, $(X, (1+ \varepsilon')\Delta)$ is a compactification of $\left(\tX, (1+\varepsilon') \tilde{\Delta} \right) \times_{\tZ} T^0$ with the additional feature that 
\begin{equation*}
K_{X} + \Delta \expl{\sim_{\bQ,T}}{by \autoref{itm:lifting_to_Artin_stack:X_min_num_triv_globally}}
0 
\qquad \Rightarrow \qquad
- K_X \sim_{\bQ,T}  \Delta \sim_{\bQ,T} \underbrace{\frac{1}{\varepsilon'} (K_X + (1 + \varepsilon') \Delta)}_{\bQ\textrm{-Cartier, ample over }T}
\end{equation*}
Also,  $(X, \Delta)$ and $\left(X_{\min}, \Delta_{\min}\right)$ are crepant, by point \autoref{itm:lifting_to_Artin_stack:X_min_num_triv_globally}. Hence, by point \autoref{itm:lifting_to_Artin_stack:X_min_flat}, \autoref{eq:lifting_to_Artin_stack:log_canonical} is satisfied for $f_{\pre}$, $X_{\pre}$, $\Delta_{\pre}$ and $(1+\varepsilon)$ replaced by $f$, $X$, $\Delta$ and $1$, respectively. In particular, $f$ is also flat and has reduced fibers. 

We also note that as $(X, (1+\varepsilon')\Delta)$ is klt and $K_X$ is $\bQ$-Cartier, $X$ is also klt.

Take now the Stein factorization of $T \to V$. As $T^0$ is finite over an open set of $V$, we obtain diagram \autoref{eq:lifting_to_Artin_stack:statement} from our statement.

\end{proof}

\begin{lemma}
\label{lem:descent_of_CM_line_bundle}
The CM line bundles $\lambda$ on $\sM_{n,v}^{\Kss}$ descends to $M_{n,v}^{\Kps}$. That is, there is a $\bQ$-line bundle $L$ on $M_{n,v}^{\Kps}$ such that $\pi^* L = \lambda$, where $\pi : \sM_{n,v}^{\Kss} \to M_{n,v}^{\Kps}$ is the natural morphism. 
\end{lemma}

\begin{proof}
By \cite[Theorem 10.3]{Alper_Good_moduli_spaces_for_Artin_stacks}, it is enough to show that for every closed $k$-point $z$ of  $\sM_{n,v}^{\Kss}$, the stabilizer of $z$ acts trivially on the fiber $\lambda_z^{\otimes N}$, for some integer $N$ which does not depend on $z$. 

The $k$-points of $\sM_{n,v}^{\Kss}$ correspond to K-semistable Fano varieties over $k$, and their stabilizers correspond to the automorphism group of the variety.

Fix a K-semistable Fano variety $F$, its automorphism group $G=\Aut(F)$  is a linear algebraic group, and the fiber $\lambda_{[F]}$ of the CM line bundle over $[F]\in \sM_{n,v}^{\Kss}$ is a one dimensional representation of $G$. Let $G_0$ be the connected component of the identity of $G$, and $G_0=R\ltimes U$ be its Levi decomposition. 

We first show that $G_0$ acts trivially on $\lambda_{[F]}$. The unipotent part $U$ acts trivially because all one dimensional representations of unipotent groups are trivial. To prove that also the reductive part $R$ acts trivially, we have to show that for every one parameter subgroup $\gamma \colon \mathbb{G}_m\to R$, the weight of the action of $\gamma$ on $\lambda_{[F]}$ is zero. This weight equals the Donaldson-Futaki invariant of the product test configuration of $X$ induced by $\gamma$. This invariant vanishes because $F$ is K-semistable.

The quotient $G/G_0$ is a finite group, it does not necessarily act trivially on $\lambda_{[F]}$, however it acts trivially on $\lambda_{[F]}^{\otimes M}$ for every integer $M$ divisible by the cardinality of $G/G_0$. To conclude, we have to show that the cardinality of $G/G_0$ is bounded as we vary $[F]$ in $\sM_{n,v}^{\Kss}$. 

By the boundness of K-semistable Fano varieties of dimension $n$ and volume $v$ proved in \cite{Jiang_Boundedness_of_threefolds_of_Fano_type_with_Mori_fibration_structures}, there exists a  projective family  $f \colon \mathcal{Y}\to T$ of Fano varieties over a smooth base with the following property: for every $[F] \in \sM_{n,v}^{\Kss}(k)$ there exists a a point $t(F)\in T$ with $f^{-1}(t(F))\cong F$.  As $-K_{\mathcal{Y}}$ is ample,  the relative polarzied Isom scheme $I:=\Isom_{T}\left(\mathcal{Y}, -K_{\mathcal{Y}}\right)$ is a finite type group scheme over $T$ \cite[Exc I.1.10.2]{Kollar_Rational_curves_on_algebraic_varieties}.  We can then look at the Stein factorization $I^0$ of $I \to T$ and at the quotient $A:= I/I^0$. By the definition of the Stein factorization and by the functoriality of $I$, for each $[F] \in \sM_{n,v}^{\Kss}(k)$ we have $I_{t(F)} \cong \Aut(F)$, and $I^0_{t(F)} \cong \Aut(F)^0$.  Additionally, as $I$ and $T$ are of finite type over $k$, the group scheme $A$ is finite over the variety $T$, hence the cardinality of the fibers is bounded by an integer $M$. Given $[F] \in \sM_{n,v}^{\Kss}(k)$, the group $G/G_0=\Aut(F)/\Aut(F)^0$ is isomorphic to the fiber of $A$ over $t(F)$, hence its cardinality is bounded by $M$.

\end{proof}

\begin{proof}[Proof of \autoref{thm:main}]
Let $L$ be the descent of $\lambda$ to $M_{n,v}^{\Kps}$. According to  \cite[Thm 6.1]{Li_Wang_Xu_Qasi_projectivity_of_the_moduli_space_of_smooth_Kahler_Einstein_Fano_manifolds},
it is enough to show that $L$ is nef on $M_{n,v}^{\Kps}$ and that on every proper irreducible  closed subset  $V' \subseteq V$ intersecting $M_{n,v}^{\uKs}$, $L|_{V'}$ is big. 

\emph{First, we show that $L$ is nef on $M_{n,v}^{\Kps}$.} Let $C \to M_{n,v}^{\Kps}$ be a finite morphism from a smooth projective curve. Let us apply \autoref{lem:lifting_to_Artin_stack} to $C \to M_{n,v}^{\Kps}$. As $\dim C=1$, in \autoref{lem:lifting_to_Artin_stack} most things collapse. That is, using the notations of \autoref{lem:lifting_to_Artin_stack}, we have $T= T^0 = S$. Hence, by calling $D$ the above three spaces that agree, we obtain a diagram as follows
\begin{equation}
\label{eq:main:lifting_curve}
\xymatrix{
\explleft{D}{smooth, projective curve} \ar[d]_{\textrm{finite}}^{\tau} \ar[r] & \sM_{n,v}^{\Kss} \ar[d]^{\pi}  \\
C \ar[r] & M_{n,v}^{\Kps}  
}
\end{equation}
The morphism $D \to \sM_{n,v}^{\Kss}$ corresponds to a flat family $f_D : X_D \to D$ of $K$-semistable Fanos.  Then, by \autoref{eq:main:lifting_curve}, we have $\lambda_{f_D} = \lambda|_D = (\pi^* L)|_D = \tau^*(L|_C)$, and therefore, it is enough to show that $\deg \lambda_{f_D} \geq 0$. However, this is exactly the statement of \autoref{thm:semi_positivity_curve}. This concludes our first claim.

Second let $V' \subseteq V$ be an irreducible  closed subset of $V$ such that $V' \cap M_{n,v}^{\uKs} \neq \emptyset$. In our second, and final claim, \emph{we show that $L|_{V'}$ is big.} Let us apply \autoref{lem:lifting_to_Artin_stack} to $V' \to M_{n,v}^{\Kps}$, and let us use the notations of \autoref{lem:lifting_to_Artin_stack} for the obtained  spaces and morphisms. 

As $T \to M_{n,v}^{\Kps}$ is generically finite, $f : X \to T$ has maximal variation. Hence, \autoref{thm:semi_positive_no_boundary}.\autoref{itm:semi_positive_no_boundary:big} applies saying that $\lambda_f$ is big. Note that by \autoref{eq:lifting_to_Artin_stack:statement}, if $\lambda$ denotes the CM line bundle on $\sM_{n,v}^{\Kss}$, then we obtain:
\begin{equation}
\label{eq:main:same_on_big_open}
\iota^* \xi^* L  = j^* \phi^* \xi^* L = \xi_0^* L = \lambda|_{T^0} =    j^* \lambda_f  
\explshift{-180pt}{=}{$\phi$ is birational and $\phi|_{j(T^0)}$ is an isomorphism} \iota^*  
\explshift{100pt}{\phi_*}{cycle theoretic pushforward}
 \lambda_f
\end{equation}
As, $T^0$ is a big open set in $S$, \autoref{eq:main:same_on_big_open} implies that $\xi^* L = \phi_* \lambda_f$. As $\lambda_f$ is big, so is $\phi_* \lambda_f$, and then by the finiteness of $\xi$ so is $L|_{V'}$.  This concludes the proof of our second claim.

\end{proof}

\section{Boundedness of the volume}
\label{sec:applications}

\begin{proof}[Proof of \autoref{cor:bounding_volume}]
We have
\begin{multline*}
\vol(-K_X-\Delta) 
\expl{=}{$-K_X - \Delta$ is ample}
(-K_X - \Delta)^{\dim X}
= \left((-K_{X/\bP^1} - \Delta) - f^* K_{\bP^1}\right)^{\dim X}
\\ = (-K_{X/\bP^1} - \Delta)^{\dim X} + (\dim X) 2 \vol (-K_F-\Delta_F)
= - \deg \lambda_{f, \Delta} + (\dim X) 2 \vol (-K_F-\Delta_F)
\\ 
\expl{\leq}{\autoref{thm:semi_positive_boundary}}
(\dim X) 2 \vol (-K_F-\Delta_F)
\end{multline*}
For the second inequality, if $F$ is smooth and $\Delta=0$, we apply the bound on the volume of K-semi-stable Fano varieties obtained in \cite[Thm 1.1]{Fujita_Optimal_bounds_for_the_volumes_of_Kahler-Einstein_Fano_manifolds}  to $F$; if $F$ is singular but still $\Delta=0$, we apply  \cite[Thm 3]{Liu_The_volume_of_singular_Kahler-Einstein_Fano_varieties}. In the log case, we can  obtain the requested inequality applyng \cite[Proposition 4.6]{Li_Liu_Kahler-Einstein_metrics_and_volume_minimization} to case where $v$ is a valuation by vanishing order at a smooth point outside the support of $\Delta_F$, and recalling that Ding semistability is equivalent to K-semistability, as shown for instance in \cite[Section 6]{Fujita_A_valuative_criterion_for_uniform_K-stability_of_Q-Fano_varieties} 

\end{proof}

\section{Examples}
\label{sec:examples}

In this section, we give examples showing the sharpness of our theorems. 
\begin{example}
\label{ex:negative_degree}
 Here, we give an example of a family of Fano varieties which are not K-semistable and such that the degree of the Chow-Mumford line bundle is strictly negative. The members of this family are smooth del Pezzo surfaces of degree $8$, and the family is isotrivial but not trivial. The relevance of this example for the study of the Chow-Mumford line bundle was already pointed out by J. Fine and J. Ross in \cite[Example 5.2]{Fine_Ross_A_note_on_positivity_of_the_CM_line_bundle}. Let us warn the reader that, in contrast with \cite{Fine_Ross_A_note_on_positivity_of_the_CM_line_bundle}, our projective bundles parametrizes rank one quotients rather than sub-bundles. Let 
\begin{enumerate}
\item $T:=\mathbb{P}^1$, 
\item $V:=\sO_{T}(-2)\oplus \sO_T(1)\oplus \sO_T(1)$ (note  that $\deg V=0$),
\item $p\colon Y:=\bP V\to T$ the natural projection,
\item $C$ the curve on $Y$ defined by the quotient $V\to \sO_T(-2)$,
\item $X:=\Bl_C Y$, $\pi\colon X\to Y$ the natural morphism, and $E$ the exceptional divisor of $\pi$, and
\item $f\colon X\to T$ the natural morphism. 
\end{enumerate}
Then, $f$ is a family of smooth degree $8$ del Pezzo surfaces. We want to compute
\begin{equation*}
\deg \lambda_{f}=-\left(-K_{X/T}\right)^3=-\left(\pi^*\sO_{Y}(3)-E\right)^3
\end{equation*}
We compute the four monomials appearing in the above expression separately. 
\begin{itemize}
\item  $(\pi^*\sO_{Y}(3))^3
\explshift{-60pt}{=}{projection formula}
\sO_{Y}(3)^3
\explparshift{72pt}{20pt}{=}{$\dim T=1$, and \cite[Rem 3.2.4]{Fulton_Intersection_theory}} 
-c_1(p^*V)\sO_{Y}(1)^2
\expl{=}{$\deg V=0$}
0$.
\item  $(\pi^*\sO_{Y}(3))^2\cdot E
\expl{=}{projection formula}
\sO_{Y}(3)^2\cdot \pi_*E 
\expl{=}{$\pi_* E = 0$}
0.$
\end{itemize}
Before describing the other two terms we need to have a better understanding of $E$. The ideal $\sJ_C$ of $C$ corresponds to the graded ideal $I$ of $\Sym V$ generated by the degree $1$ ''monomials'' $\sO_T(1) \oplus \sO_T(1)$. Hence, the sheaf $\sJ_C/\sJ^2_C$ corresponds  to the rank $2$ locally free graded module over $\Sym \left( \sO_T(-2) \right)$ generated again by $\sO_T(1) \oplus \sO_T(1)$ in degree $1$, or equivalently to the rank $2$ locally free graded module generated by $\sO_T(3) \oplus \sO_T(3)$ in degree $0$. Hence, $E \cong \bP W$, for $W:=\sO_C(3)\oplus \sO_C(3)$, and $\sO_E(-E) \cong \sO_{\bP W}(1)$. 
In particular, the natural map $\bP W \to C$ can be identified with $\pi|_E \colon \bP^1\times C\to C\cong \bP^1$, and  $\sO_{\bP W}(1)\equiv D+3F$ (see \cite[Lemma II.7.9]{Hartshorne_Algebraic_geometry}), where $D$ and $F$ are the horizontal and the vertical rulings of $E$ over $C$. We have:
\begin{itemize}
\item 
$
\pi^*\sO_{Y}(3)\cdot E^2
\explshift{-50pt}{=}{proj. formula}
\sO_{Y}(3)\cdot \pi_*(E^2)
\expl{=}{$\sO_E(-E) \cong \sO_{\bP W}(1)\equiv D+3F$}
\sO_{Y}(3)\cdot\pi_*(-D-3F)
=
\sO_{Y}(3)\cdot(- C)
\expl{=}{$\sO_{Y}(1) \cdot C=-2$}
6.$
\item $E^3=(-D-3F)^2=6$.
\end{itemize}
Wrapping up, we obtain
$$
\deg \lambda_{f}= - (\pi^*\sO_{Y}(3))^3 + 3 (\pi^*\sO_{Y}(3))^2\cdot E - 3 \pi^*\sO_{Y}(3) \cdot E^2 + E^3 =-0 + 3\cdot 0 - 3 \cdot 6 + 6 =-12<0.
$$
\end{example}

\begin{example}
\label{ex:positive_and_big}
In this example we exhibit a family $f : X \to T$ of smooth degree $8$ del Pezzo surfaces over a curve such that $\deg \lambda_f>0$, or equivalently $(-K_{X/T})^3<0$, but $-K_{X/T}$ is big. So, the statement $(-K_{X/T})^3<0$ is a negativity condition independent of $-K_{X/T}$ being big or not. However, let us also note that there is one missing piece of our example: it  is a family of non-$K$-semi-stable Fano varieties, although we suspect that a $K$-semi-stable one exists also. 

Modify \autoref{ex:negative_degree} replacing $V$ with its dual; so we take $V=\sO_T(2)\oplus \sO_T(-1)\oplus \sO_T(-1)$  and we blow-up the curve defined by the quotient $V\to \sO_{\bP^1}(2)$. In this case, $\deg \lambda_{f}=12>0$. However, $-K_{X/T}$ is still big.  Indeed, write $V=\sL \oplus \sK \oplus \sM$, where $\sL=\sO_T(2)$ and $\sK=\sM=\sO_T(-1)$. Then for every integer $m>0$:
\begin{multline*}
H^0(X,-mK_{X/T}) = H^0(X,\pi^* \sO_{Y}(3) - E)=  H^0(Y,  \sO_{Y}(3) \otimes \sI_{C}^m) \subseteq H^0(Y,  \sO_{Y}(3)) 
\\ = \bigoplus_{\parbox{70pt}{\tiny \begin{center}$i,j,l \geq 0$ \\  $i+j+l = 3m$, \\ $2i - j -l \geq 0$\end{center}}} H^0\left(T, \sL^i \sK^j \sM^l\right)
= \bigoplus_{i = m}^{3m} H^0(T, \sO_T(3i - 3m))^{\bigoplus 3m -i +1}
\end{multline*}
 As $\sI_C$ is generated in $\Sym V$ by $\sK \oplus \sM$, we obtain that
\begin{equation}
\label{eq:positive_and_big}
H^0\left(X,-mK_{X/T}\right) = \bigoplus_{\parbox{70pt}{\tiny \begin{center}$i,j,l \geq 0$ \\  $i+j+l = 3m$, \\ $2i - j -l \geq 0$\\ $j + l \geq m$\end{center}}} H^0\left(T, \sL^i \sK^j \sM^l\right)
= \bigoplus_{i = m}^{2m} H^0(T, \sO_T(3i - 3m))^{\bigoplus 3m -i +1} .  
\end{equation}
To show that $-K_{X/T}$ is big, it is enough to prove that $\displaystyle\lim_{m \to \infty} \frac{h^0( -mK_{X/T})}{m^3}>0$. Equation \autoref{eq:positive_and_big} yields:
\begin{equation*}
\frac{h^0\left(X,-mK_{X/T}\right)}{m^3} 
= \frac{\sum_{i = m}^{2m} (3i - 3m+1)( 3m -i +1)}{m^3}
 = \frac{1}{m} \sum_{i = m}^{2m} \left(3\frac{i}{m} - 3 + \frac{1}{m}\right)\left( 3 - \frac{i}{m} + \frac{1}{m}\right)
\end{equation*}
Hence,
\begin{equation*}
\lim_{m \to \infty} \frac{h^0\left(X,-mK_{X/T}\right)}{m^3} 
= \int_{1}^{2} (3x -3) ( 3 - x)dx = 2
\end{equation*}
So, we showed indeed that \emph{$-K_{X/T}$ is big}, and we even computed that $\vol(-K_{X/T}) = 12$ (a coincidence with the previous number $12$ above).

\end{example}

\begin{example}
\label{ex:not_nef}

Here we give an example of a family $f : X \to T$ of smooth Del-Pezzo surfaces of degree $6$ such that $\delta_{X_t}=1$ for all closed point $t \in T$, $\deg\lambda_{f}=0$ but $-K_{X/T}$ not nef. This shows that the hypothesis $\delta>1$ in \autoref{thm:bounding_nef_threshold} is necessary.

For this, we modify \autoref{ex:negative_degree} in two respects: 
\begin{enumerate}
\item We take $V$ to be the dual vector bundle, that is, $V:=\sO_T(2)\oplus \sO_T(-1)\oplus \sO_T(-1)$.
\item Instead of one curve, we blow up $3$ curves. That is,  we set $X:= \Bl_{C_1, C_2, C_3} Y$, where $C_i$ is the curve defined by the quotient $V \to \sL_i$, where $\sL_i$ is the $i$-th direct summand of $V$.
\end{enumerate}
Let  $E_i$, $F_i$ and $W_i$ (and for $i=1$ also $D_1$) to be defined for each $C_i$ as $E$, $F$ and $W$ (and for $i=1$ also  $D$) was defined for $C$ in \autoref{ex:negative_degree}. We do not define  $D_i$ also as in \autoref{ex:negative_degree} because for $i=2,3$,  $W_i = \sO_{C_i} \oplus \sO_{C_i}(3)$, so $E_2, E_3 \not\cong \bP^1 \times \bP^1$. Instead, for $i=2,3$, set $D_i$ to be the divisor class of $\sO_{\bP W_i}(1)$. 

Note that the $E_i$ are disjoint, and hence any intersection product involving different $E_i$ is automatically $0$. We write out below the computations where the result is different than in \autoref{ex:negative_degree}, where $i=2$ or $3$:
\begin{itemize}
\item 
$
\pi^*\sO_{Y}(3)\cdot E^2_1
\expl{=}{proj. formula}
\sO_{Y}(3)\cdot \pi_*(E^2_1)
\expl{=}{$\sO_{E_1}(-E_1) \equiv D_1-3F_1$}
\sO_{Y}(3)\cdot\pi_*(-D_1+3F_1)
=
\sO_{Y}(3)\cdot(- C_1)
\expl{=}{$\sO_{Y}(1) \cdot C_1=2$}
-6.$
\item $E^3_1=(-D_1+3F_1)^2=-6$.
\item $\pi^*\sO_{Y}(3)\cdot E^2_i=
\sO_{Y}(3)\cdot \pi_*(E^2_i)
=
\sO_{Y}(3)\cdot\pi_*(-D_i)
=\sO_{Y}(3)\cdot(- C_i)=3.$

\item  
$E_i^3= (-D_i)^2= D_i^2
\expl{= }{
\textrm{\cite[Rem 3.2.4]{Fulton_Intersection_theory}}}
c_1(\pi^*W_i)D_i
=\deg(W_i)F_i\cdot D_i=\deg W_i=3$.
\end{itemize}
Set $E:=E_1 + E_2 + E_3$. Then, 
we conclude that 
\begin{multline*}
\deg \lambda_f=- 3(\pi^* \sO_Y(3) - E)^3 
= - \pi^* \sO_Y(3) \cdot E^2  + E^3 
\\ = 
- 3\pi^* \sO_Y(3) \cdot (E_1^2 + E_2^2 + E_3^2)  + (E_1^3 + E_2^3 + E_3^3) 
 = -3\cdot(- 6) -2\cdot 3\cdot 3  + (-6) + 2\cdot 3)= 0
\end{multline*}
The fibres  of $f$ are smooth del Pezzo surfaces of degree $6$, they are well-known to be K-poly-stable (so, in particular, K-semi-stable), but they are not uniformly K-stable because they have a positive dimensional automorphism group. The delta invariant is thus equal to $1$ (e.g., \autoref{def:K_stable}). Furthermore, for $i=2,3$:
\begin{equation*}
\left( -K_{X/T}|_{E_i} \right)^2 = \left((\sO_Y(3) \cdot C_i) F_i - E_i|_{E_i} \right)^2 = \left( D_i - 3 F_i \right)^2  = D_i^2 - 6 F_i \cdot D_i = 3 - 6= -3.
\end{equation*}
Hence, $-K_{X/T}|_{E_i}$ is not nef, and then also  \emph{$-K_{X/T}$ is not nef.}

\end{example}

\begin{example}
\label{ex:anti_canonical_no_section}
In this example, for each choice of an integer $d>0$ we exhibit  families $f: X \to T$ of uniformly K-stable del Pezzo surfaces of degree $4$ over a smooth projective curve. In this situation, \autoref{thm:semi_positive_no_boundary}.\autoref{itm:semi_positive_no_boundary:big} tells us that $\deg \lambda_f>0$, or equivalently $(-K_{X/T})^3<0$. So, one would expect $-K_{X/T}$ to have  only a few sections. Here, we show that both the expected and the unexpected behavior can happen. More precisely, $|-K_{X/T}|_{\bQ}= \emptyset$ for $d>3$, and $\kappa(-K_{X/T})\geq 1$ for $d=1$.

Let $p_1,\dots ,p_4$ be four points in $\bP^2$ in general position, and denote by $L_{ij}$ the line trough $p_i$ and $p_j$. Let $\iota\colon T\to \bP^2$ be a degree $d$ smooth curve in $\bP^2$ which avoids the four points.  Let $\Gamma\cong T$ be the graph of $\iota$ in $\bP^2\times T$, and let $T_i$ be the curve $\{p_i\}\times T$ in $\bP^2\times T$. We want to look at the blow-up $\pi \colon Y\to \bP^2\times T$ of $\Gamma$ and $T_i$, for $i=1,\dots, 4$. Denote by $g\colon Y\to T$ the natural projection.

The family $g\colon Y\to T$ is generically a family of degree $4$ smooth del Pezzo surfaces of maximal variation. The only exception is at the points $t \in T$ where $\iota(T)$ intersects one of the lines $L_{ij}$. In these cases, $Y_t = \Bl_{p_1,p_2,p_3,p_4,p}\bP^2$, where $p$ lies on $L_{ij}$. In particular, $-K_{Y_t}$ is big and semi-ample, and there is a unique curve $C$ for which $C \cdot -K_{Y_t} =0$: the proper transform of $L_{ij}$. The anti-canonical model is the contraction of $L_{ij}$ to an $A_1$ singularity, so in particular it has canonical singularities.

Let $f\colon X\to T$ be the relative anti-canonical model of $g$ (remark that $R^ig_*(-mK_{Y/T})=0$ for $i>0$ and $m$ big enough, by Kawamata-Viehweg vanishing theorem, so we do have base change). The family $f$ satisfies the hypotheses of \autoref{thm:semi_positive_no_boundary}\autoref{itm:semi_positive_no_boundary:big}, so $\deg \lambda_f>0$. 

We show that, if $d>3$, then $|-K_{X/T}|_{\bQ}=\emptyset$ and if $d=1$, then $\kappa(-K_{X/T}) \geq 1$. In either case the crucial remark is that  $H^0(X,-mK_{X/T})$ can be identified with the subspace of$H^0(\bP^2 \times T, \sO_{\bP^2 \times T}(3m)) \cong H^0(\bP^2, \sO_{\bP^2}(3m))$ which vanish along $\iota(T)$ and $p_i$ with multiplicity at least $m$. Hence:
\begin{itemize}
\item If $d>3$, then there are no such sections, as $d$ is exactly the degree of $\iota(T)$.
\item If $d=1$, then $\iota(T)$ is a line $L$. So, $|-K_{X/T}|$ is the set of cubics $C$ on $\bP^2$ such that $C$ goes through $p_i$ and $\Supp C$ contains $L$. Hence, $C = L + C'$, where $C'$ is a conic through $p_i$. There is a one parameter family of such conics. 
\end{itemize}

\end{example}

\appendix

\section{Computations concerning the definition of the CM line bundle}
\label{sec:appendix}

The following work is needed to prove the statements of \autoref{sec:CM_definition}. These  are verifications of technical issues concerning the singular situation.

We need the following lemmas as we work with singular varieties, and hence Riemann-Roch computations do not work directly.  It turns out that if the spaces are normal then singularities do not mess up any of the terms involving any of the definitions of the CM line bundle. However, in the non-normal situation, which we do not deal with in the present article, \autoref{lem:Mumford_line_bundles_over_curve}  seems to suggest that one has to face extra difficulties.

\begin{lemma}
\label{lem:pushforward_degree_coherent_sheaf}
Consider the following situation:
\begin{itemize}
 \item $f : X \to T$ is a projective morphism to a normal quasi-projective variety (allowing $T = \Spec k$), 
 \item $M$ is  an $f$-ample $\bQ$-divisor on $X$, 
 \item $\sE$ is a coherent sheaf on $X$, and
 \item $r \geq 0 $ is an integer such that $\dim \Supp \sE_t \leq r$ for $t \in T$ the generic point and $\dim \Supp \sE_t \leq r+1$ for $t \in T$ a codimension $1$ point. 
\end{itemize}
Then there are $\bQ$-divisors  $D_i$ (resp.  $d_i \in \bQ$), determined uniquely up to $\bQ$-linear equivalence (resp. determined uniquely), such that for all $q$ divisible enough, if $\dim T>0$, then 
\begin{equation*}
c_1( f_* (  \sO_X(qM) \otimes \sE )) = \sum_{i=0}^{1 + r} q^i D_i,
\end{equation*}
(resp. if $ T=\Spec k$, then
\begin{equation*}
h^0 ( X, \sO_X(qM) \otimes \sE ) = \sum_{i=0}^{ \dim \Supp \sE} q^i d_i {\Big)}. 
\end{equation*}
\end{lemma}

\begin{proof}
In the case of $T =\Spec k$,  $h^0 (X,  \sO_X(qM) \otimes \sE  )$ equals the Hilbert polynomial for $q$ divisible enough, and hence the statement follows. So, from now we assume that $\dim T>0$.

Let $s>0$ be an integer such that $sM$ is relatively very ample. As the statement is for all $q$ divisible enough, by replacing $M$ with $sM$ we may assume that $M$ is relatively very ample and $f_* \sO_X(M)$ is locally free, in which case we will exhibit $\bZ$-divisors $D_i$. Furthermore, as the statement is about codimension $1$ behavior over $T$, and $T_{\reg}$ is a big open set of $T$, by replacing $T$ with $T_{\reg}$ we may assume that $T$ is regular.

As $M$ is relatively very ample, it induces an embedding $\iota : X \hookrightarrow P:= \Proj_T f_* \sO_X(M)$. Let $\pi : P \to T$ be the natural morphism.  As $P$ is regular, $\iota_* \sE$ has a locally free resolution $\sP^\bullet$, which in particular is a perfect complex on $P$. Hence, for $q$ divisible enough, the following holds (where, following \cite{Knudsen_Mumford_The_projectivity_of_the_moduli_space_of_stable_curves_I}, $\det$ is the alternating tensor product of the determinants of the elements of a locally free resolution, which exists as $T$ is regular):
\begin{multline*}
c_1 ( f_* (  \sO_X(qM) \otimes \sE ) ) =
c_1(\det f_* (  \sO_X(qM) \otimes \sE ) ) 
\expl{=}{relative Serre vanishing}
c_1(\det Rf_* (  \sO_X(qM) \otimes \sE ))
\\ 
\explparshift{95pt}{10pt}{=}{ $\iota^* \sO_P(1)\cong \sO_X(M)$, and projection formula}
c_1(\det R \pi_*  ( \sO_P(q) \otimes \iota_*  \sE ))
=
c_1(\det R\pi_* \left(  \sO_P(q) \otimes \sP^\bullet \right))
\explparshift{250pt}{-150pt}{=}{ for some line bundles $\sM_i$ according to 
\cite[Thm 4, p 55]{Knudsen_Mumford_The_projectivity_of_the_moduli_space_of_stable_curves_I} (see p 50 for the definition of  $Q_{(r)}$, which is the same as the last itemized condition in the statement of our lemma)}
c_1\left( \bigotimes_{i=0}^{1 + r} \sM_i^{q^i} \right)
\explshift{40pt}{=}{$D_i:=c_1(\sM_i)$}
\sum_{i=0}^{1 + r} q^i D_i.
\end{multline*}
\end{proof}

Note that in the following lemma we do not assume any $\bQ$-Cartier hypothesis on $K_{X/T}$. Still, our intersection in \autoref{eq:Mumford_line_bundles_over_curve:statement} is well defined as $M$ is $\bQ$-Cartier.

\begin{lemma}
\label{lem:Mumford_line_bundles_over_curve}
\label{lem:Hilbert_poly_coeffs}
Let $f : X \to T$ be a surjective morphism from a normal projective variety of dimension $n+d$ to a smooth  variety of dimension $d \geq 0$ with $n \geq 1$, and let $M$ be a $\bQ$-Cartier $f$-ample divisor on $X$.
\begin{enumerate}
 \item If $\dim T>0$, then for all divisible enough integers $q>0$, 
\begin{equation}
\label{eq:Mumford_line_bundles_over_curve:statement}
c_1( f_* \sO_X(qM)) = \frac{q^{n+1}}{(n+1)!} f_*\left( M^{n+1} \right) - \frac{q^n}{2 \cdot n!}f_* \left(K_{X/T} \cdot M^n \right) + p^{n-1}(q),
\end{equation}
where $p^{n-1}(x)$ is polynomial of degree at most $n-1$ with $x$ as a variable and $\bQ$-divisors as coefficients. 
\item 
If $T = \Spec k$, then 
$
\chi(X, qM) = \frac{M^n}{n!} q^{n} - \frac{K_X \cdot M^{n-1}}{2 (n-1)!} q^{n-1} + O(q^{n-2}).
$

\end{enumerate}
In particular, if $T$ is a curve and $M^{n+1}>0$, then $\deg f_* \sO_X(qM)>0$ for all positive integers $q$ divisible enough.
\end{lemma}

\begin{proof}
As Grothendieck-Riemann-Roch works directly only for smooth $X$ (or also on locally complete intersection singularities, which does not include klt singularities with Cartier index greater than $1$), we need to compare $X$ with a resolution. Let $\sigma : Z \to X$ be a resolution of singularities and set $g:= f \circ \sigma$. 

First, \emph{we claim that for all  integers $i>0$ and $1 || q$, in the respective cases:
\begin{enumerate}
 \item $\deg R^i g_* \sO_Z(q \sigma^* M)  = p_i^{n-1}(q)$  for some polynomial $p_i^{n-1}(x)$  of degree at most $n-1$  and $\bQ$-divisor coefficients, and 
 \item $ h^i (Z, q \sigma^* M)  = O(q^{n-2})$.
\end{enumerate}}
 Indeed,  fix an integer $i>0$. There is a spectral sequence with $E^2$-terms $R^p f_* (  \sO_X(qM) \otimes R^r \sigma_* \sO_Z )$ abutting to $R^i g_* \sO_Z(q \sigma^* M) $ for $i=p+r$. As $M$ is $f$-ample and $q$ is divisible enough, this spectral sequence degenerates. Therefore,
\begin{equation*}
R^i g_* \sO_Z(q \sigma^* M) \cong f_* (  \sO_X(qM) \otimes R^i \sigma_* \sO_Z ).
\end{equation*}
Then \autoref{lem:pushforward_degree_coherent_sheaf} applied to $\sE:=R^i \sigma_* \sO_Z $ concludes our claim, using that  $ \Supp R^i \sigma_* \sO_Z $ is contained in the non-normal locus, which  is at most $n-2$ dimensional in the generic fiber and at most $n-1$ dimensional over the fibers over codimension $1$ points. 

Having shown our claim, in the $\dim T>0$ case the statement of the proposition is shown by the following computation, which holds for every $q$ divisible enough (so $qM$ is $f$-very ample, Cartier and without higher cohomologies on the fibers):
\begin{multline*}
c_1( f_* \sO_X(qM)) = \ch_1 (f_* \sO_X(qM)) 
\explshift{-80pt}{= }{$X$ is normal, and hence $\sigma_* \sO_Z \cong \sO_X$}
\ch_1 (g_* \sO_Z(q\sigma^* M))
\explparshift{200pt}{100pt}{=}{ the above claim, where $p_i^{n-1}(q)$ are the polynomials the existence of which is stated in the claim}
 \ch_1 (g_! \sO_Z(q \sigma^* M))  - \sum_{i \geq 1} (-1)^i p_i^{n-1}(q)
\\ 
\explshift{400pt}{=}{ Grothendieck-Riemann-Roch, as $Z$ and $T$ are smooth, and setting $p^{n-1}(x):=- \sum_{i \geq 1} (-1)^i p_i^{n-1}(x)$}
 g_* \left( \left(\ch(\sO_Z(q \sigma^* M)) \todd(T_g) \right)_{n+1} \right)  +  p^{n-1}(q)
=
g_* \left( \left(  \left( \sum_{i=0}^{n+1} q^i\frac{( \sigma^* M)^i }{i!} \right) \todd(T_g) \right)_{n+1}\right) + p^{n-1}(q)
\\ 
\explshift{150pt}{=}{$\tilde{p}^{n-1}(x):= p^{n-1}(x) + \sum_{i=0}^{n-1}x^i \frac{(\sigma^* M)^i}{i!} \todd_{n+1-i}(T_g)$}
\frac{q^{n+1}}{(n+1)!} f_* \sigma_* (\sigma^* M)^{n+1} + \frac{q^n}{2n!} f_* \sigma_* \left((\sigma^* M)^n \cdot (-K_{Z/T})\right) + \tilde{p}^{n-1}(q)
\\ 
\explshift{70pt}{=}{ $\sigma_* (\sigma^* M)^{n+1}=M^{n+1}$,  and $\sigma_* \left((\sigma^* M)^n \cdot (-K_{Z/T})\right) = M^n \cdot \sigma_* (-K_{Z/T}) = - M^n \cdot K_{X/T}$ by the projection formula}
\frac{q^{n+1}}{(n+1)!} f_* \left(M^{n+1}\right) - \frac{q^n}{2n!} f_*  \left(M^n \cdot K_{X/T} \right) + \tilde{p}^{n-1}(q).%
\end{multline*}
In the case of $T= \Spec k$, a similar computation concludes the proof:
\begin{multline*}
\chi(X,qM)= h^0(X,qM) 
\expl{=}{$X$ is normal} 
h^0(Z, q \sigma^* M) 
\expl{=}{our claim above}
\chi(Z, q \sigma^* M) + O(q^{n-2})
\\ 
\explshift{100pt}{=}{Grothendieck-Riemann-Roch, as $Z$ is smooth}
\int_Z \ch(\sO_Z(q \sigma^* M)) \todd(T_Z) + O(q^{n-2})
=\int_Z  \left( \sum_{i=1}^{n} q^i\frac{( \sigma^* M)^i }{i!} \right) \todd(T_Z) + O(q^{n-2})
\\ 
\explparshift{280pt}{280pt}{=}{ $(\sigma^* M)^{n}=M^{n}$, using our assumption and that $\sigma$ is birational, and  $(\sigma^* M)^{n-1} \cdot (-K_{Z})  = - M^{n-1} \cdot K_{X}$ by the projection formula}  
\frac{q^{n}}{n!} (\sigma^* M)^{n} + \frac{q^{n-1}}{2(n-1)!} (\sigma^* M)^{n-1} \cdot (-K_{Z}) + O(q^{n-2}) = \frac{q^{n}}{n!} M^{n} - \frac{q^{n-1}}{2(n-1)!}  M^{n-1} \cdot K_{X/T} + O(q^{n-2})%
\end{multline*}
\end{proof}

\begin{remark}
\label{rem:big}
In the situation of \autoref{lem:Mumford_line_bundles_over_curve}, we also have that if  $T$ is a curve and $M^{n+1}>0$, then $M$ is big on $X$. Let us stress that $M$ is not assumed to be nef on $X$, hence this does not follow directly from standard criteria such as \cite[Theorem 2.2.14]{Lazarsfeld_Positivity_in_algebraic_geometry_I}. Indeed: 
\begin{multline*}
h^0(X,qM)=h^0(T,f_*\sO_X(qM)) \geq  \chi(T,f_*\sO_X(qM))  
\\
\expl{=}{Riemann-Roch on $T$}
\deg f_*\sO_X(qM) + \rk  f_*\sO_X(qM) (1-g) 
\expl{ =}{\autoref{lem:Mumford_line_bundles_over_curve}}
\frac{q^{n+1}}{(n+1)!} M^{n+1} + O(q^n).
\end{multline*}
\end{remark}

\begin{proof}[Proof of \autoref{prop:2_defs_CM_same}]
{\scshape Step 1: we may assume that $T$ is smooth.} 
If $T$ is already smooth, there is nothing to prove, so assume that it is not smooth. Hence, by our assumptions, the fibers are normal and $\Supp \Delta$ does not contain any of the fibers.
Take a resolution $\tau : T' \to T$. Then, according to \autoref{sec:base_change_relative_canonical_normal_fibers}, in the respective cases,
\begin{enumerate}
 \item $f_{T'} : X_{T'} \to T'$ and $L_{T'}$, and 
 \item $f_{T'}: \left(X_{T'}, \Delta_{T'} \right) \to T'$ and $L_{T'}$,
\end{enumerate}
 satisfy all our original assumptions, including that $sL \sim -(K_{X/T} + \Delta)$ in the case of point \autoref{itm:2_defs_CM_same:leading_term} by \autoref{prop:relative_canonical_base_change_normal}.\autoref{itm:relative_canonical_base_change:base_change}. 
 
 \emph{We claim that  $\tau_* \lambda_{f_{T'},L_{T'}}= \lambda_{f,L}$ (resp.  $\tau_* \lambda_{f_{T'},\Delta_{T'}}= \lambda_{f,\Delta}$)}. This is verified in the following computations, where $\sigma : X_{T'} \to X$ is the induced morphism:
\begin{enumerate}
 \item $\tau_* \lambda_{f_{T'},L_{T'}}= \tau_* \left( f_{T'}\right)_* \left(\mu L_{T'}^{n+1}+(n+1)L^n_{T'} \cdot K_{X_{T'}/T'}\right) 
 \expl{=}{$\tau \circ f_{T'} = f \circ \sigma$, and $\sigma^* K_{X/T} = K_{X_{T'}/T'}$ by \autoref{sec:base_change_relative_canonical_normal_fibers} }
f_* \sigma_* \left(\mu \sigma^* L^{n+1}+(n+1) \sigma^* L^n \cdot \sigma^* K_{X/T}\right)$ 
\\ $
\explshift{35pt}{= }{$\sigma_* \sigma^* = \id$}
f_*  \left(\mu L^{n+1}_{T'}+(n+1)L^n_{T'} \cdot K_{X/T}\right)= \lambda_{f,L}$, and
 \item $\tau_* \lambda_{f_{T'},\Delta_{T'}} = -\tau_*  \left(f_{T'} \right)_* \left(-\left(K_{X_{T'}/T'} + \Delta_{T'}\right)^{n+1} \right) = - f_* \sigma_* \left(-\left(K_{X_{T'}/T'} + \Delta_{T'}\right)^{n+1} \right)
 \\ =- f_* \left(-\left(K_{X/T} + \Delta\right)^{n+1} \right)= \lambda_{f, \Delta}. $
\end{enumerate}
Having shown our claim, Step 1 follows. Indeed, if we prove, in the case of point \autoref{itm:2_defs_CM_same:Tian_Paul}, that $s^n \lambda_{f_{T'},L_{T'}}= c_1(L_{\CM,f_{T'}, sL_{T'}})$, then
\begin{equation*}
s^n \lambda_{f, L} = s^n \tau_* \lambda_{f_{T'},L_{T'}}= \tau_* c_1( L_{\CM,f_{T'}, sL_{T'}})
\expl{=}{\autoref{lem:Knudsen_mumford_base_change}} 
\tau_*\tau^* c_1( L_{\CM, f, sL})
 = c_1(  L_{\CM, f, sL})
\end{equation*}
The case of \autoref{itm:2_defs_CM_same:leading_term} is verbatim the same with  $s^n$, 
$\lambda_{f, L}$ and  $L_{\CM, f, sL}$ replaced by $-s^{n+1}$, $\lambda_{f,\Delta}$ and $\sM_{n+1}$, respectively.

\noindent { \scshape Step 2: The proof assuming that $T$ is smooth. } 
Set $M_i:=c_1(\sM_i)$. 
Taking into account that
\begin{equation*}
\frac{q^{n+1}}{(n+1)!}-\frac{q (q-1) \cdot \dots \cdot (q-n)}{(n+1)!}= \frac{n}{2} \frac{q^n}{n!} + O(q^{n-1}), 
\end{equation*}
according to \autoref{lem:Mumford_line_bundles_over_curve}, 
\begin{equation}
\label{eq:prop:2_defs_CM_same:Mumford_expansion}
M_{n+1} = f_* (sL)^{n+1} \textrm{, and }  M_n = f_* \left(\frac{-(sL)^n \cdot K_{X/T}}{2} + \frac{n}{2} (sL)^{n+1} \right). 
\end{equation}
(where $L= -(K_{X/T} + \Delta)$ in the case of point \autoref{itm:2_defs_CM_same:leading_term}.) Hence, the next computation concludes the proof in the respective cases:
\begin{enumerate}
 \item

\begin{multline*}
c_1(L_{\CM, f, sL}) 
= \left(n(n+1)+ \mu_{sL} \right)M_{n+1} -2 (n+1) M_n 
\\ = \left(n(n+1)+ \mu_{sL} \right) f_*((sL)^{n+1}) - 2(n+1) f_* \left( \frac{-K_{X/T} \cdot (sL)^n}{2 } + \frac{n (sL)^{n-1}}{2} \right)
\\ = \frac{\mu_L}{s} s^{n+1} f_* L^{n+1} -   s^n (n+1) f_* (K_{X/T}\cdot L^n) = s^n \lambda_{f,L}
\end{multline*}
\item 
\begin{equation*}
M_{n+1} = f_* \left((-s(K_{X/T} + \Delta))^{n+1} \right) = - s^{n+1} \lambda_{f, \lambda}
\end{equation*}
\end{enumerate}

\end{proof}

The next lemma is a technical statement used in \autoref{prop:CM_base_change}. 

\begin{lemma}
\label{lem:highest_Knudsen_Mumford_coefficient_integral}
Let $h : V \to S$ be  a flat $n$-relative dimensional morphism from a reduced projective scheme to a smooth projective curve, and let $\sL$ be an $h$-very ample line bundle on $V$. Let $\pi : Z \to V$ be the normalization of $V$ with $g : Z \to S$ being the induced morphism, and assume also that $\pi^* \sL$ is $g$-very ample. Then the $n+1$-th (so highest) Knudsen-Mumford coefficients of $\sL$ with respect to $g$ (as in \autoref{notation:Paul_Tian}) agrees with that of $\pi^* \sL$.
\end{lemma}

\begin{proof}
Consider the exact sequence on $V$ given by the normalization:
\begin{equation*}
\xymatrix{
0 \ar[r] & \sO_V \ar[r] & \pi_* \sO_Z \ar[r] & \sE \ar[r] & 0
}
\end{equation*}
This yields a natural inclusion $h_* \left( \sL^q \right) \hookrightarrow  g_* \left( \pi^* \sL^q \right)$.
Hence, it is enough to prove that for $q$ divisible enough, $\deg h_* ( \sL^q \otimes \sE) = O(q^{n})$, which is given by \autoref{lem:pushforward_degree_coherent_sheaf} as $\dim \Supp \sE \leq n$.  
\end{proof}

\bibliographystyle{skalpha}
\bibliography{Codogni_Patakfalvi}

\end{document}